\RequirePackage{fix-cm}
\RequirePackage{amsmath}
\documentclass[smallextended,12pt]{svjour3}       
\smartqed  
\usepackage{graphicx}
 \usepackage{txfonts,bm} 
\usepackage{dsfont}
\usepackage{amssymb}
\usepackage{enumerate}
\usepackage{enumitem}
\usepackage{dsfont}
\usepackage{rotating}
\usepackage{cancel}
\usepackage{centernot}
\usepackage{centernotlarge}
\usepackage{xfrac}
\usepackage{multirow}
\usepackage{caption}
\usepackage{subcaption}
\usepackage{algorithm}
\usepackage{algpseudocode}
\usepackage{url}
\usepackage{amsopn}
\usepackage{natbib}

\RequirePackage[OT1]{fontenc}
\RequirePackage[colorlinks,citecolor=blue,urlcolor=blue,linkcolor=blue]{hyperref}
\usepackage{booktabs,multirow}
\usepackage{rotating}
\usepackage{environ}
\renewcommand{\Pr}{\operatorname{\mathds{P}}}
\newcommand{\RR}{\mathds{R}}

\newcommand{\one}[1]{\mathds{1}_{\{#1\}}}
\newcommand{\ds}{\displaystyle}
\newcommand{\E}{\operatorname{\mathds{E}}}
\newcommand{\Var}{\operatorname{\mathsf{Var}}}
\newcommand{\Cov}{\operatorname{\mathsf{Cov}}}
\newcommand{\Corr}{\operatorname{\mathsf{Corr}}}
\newcommand{\re}{\operatorname{\mathsf{res\hspace*{.3ex}eff}}}
\newcommand{\rr}{\operatorname{\mathsf{red\hspace*{.3ex}ratio}}}
\newcommand{\card}{\operatorname{card}}
\newcommand{\tr}{\operatorname{tr}}
\newcommand{\logit}{\operatorname{logit}}

\newcommand{\diag}{\operatorname{diag}}
\newcommand{\rank}{\operatorname{rank}}
\newcommand{\md}{\operatorname{md}}

\newcommand{\CZ}{\mathcal{Z}}

\newcommand{\CS}{\mathcal{S}}
\newcommand{\CW}{\mathcal{W}}

\newcommand{\sfc}{\mathsf{C}}
\newcommand{\sfv}{\mathsf{V}}
\newcommand{\sfe}{\mathsf{e}}
\newcommand{\Sig}{\textrm{$\centernotlarge{\mathrm{\Sigma}}$}}
\newcommand{\ud}{{\rm d}}
\newcommand{\ta}{T}
\newcommand{\h}[1]{\widehat{#1}}
\newenvironment{pr}[1]
{\noindent$\emph{#1}$\hspace*{.6ex}}
{\hspace{\fill}$\square$\medskip\\}
 \journalname{}
\begin{document}

\title{Optimality of Training/Test Size and Resampling Effectiveness of Cross-Validation Estimators of the Generalization Error
       \thanks{Both authors acknowledge support provided by the Department of Biostatistics, University at Buffalo (in the form of start-up package to the second author).}
      }

\titlerunning{Optimality of Training/Test Size}        

\author{Georgios Afendras \and
        Marianthi Markatou
}


\institute{G. Afendras \at
              Department of Biostatistics, University at Buffalo, 811 Kimball Tower, Buffalo, NY 14214, USA \\
              Tel.: +1 (716) 829-5094 \\
              Fax: +1 (716) 829-2200 \\
              \email{gafendra@buffalo.edu}           
           \and
           M. Markatou \at
              Department of Biostatistics and School of Biomedical Sciences, University at Buffalo, 726 Kimball Tower, Buffalo, NY 14214, USA \\
              Tel.: +1 (716) 829-2894 \\
              Fax: +1 (716) 829-2200 \\
              \email{markatou@buffalo.edu}
}

\date{}

\maketitle

\begin{abstract}
An important question in constructing Cross Validation (CV) estimators of the generalization error is whether rules can be established that allow ``optimal'' selection of the size of the training set, for fixed sample size $n$. We define the {\it resampling effectiveness} of random CV estimators of the generalization error as the ratio of the limiting value of the variance of the CV estimator over the estimated from the data variance. The variance and the covariance of different average test set errors are independent of their indices, thus, the resampling effectiveness depends on the correlation and the number of repetitions used in the random CV estimator. We discuss statistical rules to define optimality and obtain the ``optimal'' training sample size as the solution of an appropriately formulated optimization problem. We show that in a broad class of loss functions the optimal training size equals half of the total sample size, independently of the data distribution. We optimally select the number of folds in $k$-fold cross validation and offer a computational procedure for obtaining the optimal splitting in the case of classification (via logistic regression). We substantiate our claims both, theoretically and empirically.
\end{abstract}

\section{Introduction}
\label{sec:intro}
Advances and accessibility to technology for capturing and storing large amounts of data has transformed many scientific fields. Yet, for many important scientific questions very large data sets are not available. In these cases, when learning algorithms are used, resampling techniques allow the estimation of the generalization error, an important aspect of algorithmic performance.

The generalization error is defined as the error an algorithm makes on cases that the algorithm has never seen before. The generalization error is important because it relates to the algorithm's prediction capabilities on independent data. The literature includes both, theoretical investigations of risk performance of machine learning algorithms as well as numerical comparisons.

Estimation of the generalization error can be achieved via the use of resampling techniques. The process consists of splitting the available data into a learning or training set and a test set a large number of times and averaging over these repetitions. A very popular resampling technique is cross validation. A detailed overview of the vast literature on cross-validation is outside the scope of this paper. The interested reader is referred to \citet{Stone1974,Stone1977} for foundational aspects of cross validation, \citet[Ch.s~3,8]{BFOS1984}, \citet{Geisser1975}, and to \citet{AC2010} for a comprehensive survey. Here, it is sufficient to notice that although cross validation procedures are extensively used in practice, their performance is studied by very few theoretical papers.

An important question in constructing cross validation estimators of the generalization error is the selection of the size of the training (and hence test) set for ``optimal'' estimation of the generalization error. In this paper, we address the question of training sample size selection and refer to work by \citet{Burman1989,Burman1990} in identifying the rules for ``optimal'' selection, where ``optimality'' is defined via minimization of the variance of the cross validation estimators.

Estimating the generalization error of a learning method is important not only for understanding its future prediction accuracy, but also for choosing a classifier from a given set, or for combining classifiers \citep{Wolpert1992}. Furthermore, one may be interested in testing hypotheses about equality of generalization errors between learning algorithms. In \citet{vdWBvW2009}, in the field of biostatistics, an inference framework is developed for the difference in errors between two prediction procedures. In that paper, for each splitting in repeated sub-sampling the predictions of the two classifiers are compared by a Wilcoxon test, and the resulting $p$-values are combined. \citet{vdWBvW2009} state that ``$\ldots$ little attention is paid to inference for comparing prediction accuracies of two models or, more generally, prediction procedures''. Carrying out inference on equality of generalization errors requires insight into the variance of the estimator of the generalization error. But providing a variance of the generalization error estimator is a more difficult problem because of the complexity of existing estimators and the many dependencies they exhibit on, for example the loss function, the resampling method used, the ratio of the training to test sample size and other factors. \citet{NB2003} were the first to provide estimators for the variance of the random cross validation estimator of the generalization error. Furthermore, \citet{MTBH2005} proposed a moment approximation-based estimator of the same cross validation estimator of the generalization error, and compared this estimator with those provided by \citeauthor{NB2003}. Other relevant work on variance estimation includes \citet{WL2014,FHdBB2013,MDS2011}.

Selecting the size of the training set and understanding the effect of this selection on the generalization error, its bias and its variance, is of interest to many areas of scientific investigation. Examples include pattern recognition and machine learning \citep{Highleyman1962,FH1989,RJ1991,GMSV1998,Kearns1997}, statistics \citep{BP2004,CCP2015}, remote sensing \citep{ZQSZ2013,JSM2014}, biostatistics and bioinformatics \citep{DS2005,DS2007,DS2011,DZS2008,DS2013,PCGHSSNTINHVBDHSSP2010,SFCWC2013} among others. In this paper, we offer an analysis of the problem of training sample size selection when the total sample size is fixed and interest centers on carrying out inference on the generalization error. The analysis is complex, and for this reason we consider three kinds of problems to cover a good range of possible applications. These include the case where the decision rule is the sample mean, the regression case and classification via logistic regression. The strategy we follow allows us to draw general conclusions about ``optimal'' selection of the training (and hance test) set sample size, and to recommend simple rules for practical use despite the complexity of the problem.

We analyze training sample size selection for cross validation estimators of the generalization error. Specifically, we study random cross validation estimators and $k$-fold estimators based on general loss functions. We first discuss two novel approaches for obtaining the resampling size $J$ of the random cross validation estimator of the generalization error \citep{NB2003,MTBH2005}. \citet[][p.~255, Sect.~5]{NB2003} state ``We also want to make recommendations on the value of $J$ to use for those methods that involve $\h{\mu}_{{\rm CV},J}$''. We offer methods for obtaining $J$ that are based on the $\Var(\h{\mu}_{{\rm CV},J})$. Furthermore, we establish rules for the ``optimal'' selection of the training set sample size for cross validation estimators of the generalization error. We show that when the decision rule is the sample average and the loss function is sufficiently smooth and belongs in \citeauthor{Efron1983}'s $q$-class (\citeyear{Efron1983,Efron2004}), the optimal value of the training sample size is $\lfloor n/2 \rfloor$, independent of the data distribution, where $\lfloor\cdot\rfloor$ indicates the upper integer part of a real number. When the decision rule is regression and the loss function is squared error loss, or the task is classification via logistic regression, the optimal value of the training sample size is again $\lfloor n/2 \rfloor$, independent of the data distribution. In general, the optimal value of the training sample size depends primarily on the loss function and secondly on the data distribution. We offer a construction of a new loss function that exemplifies this dependence. Additionally, we provide a rule for selecting the ``optimal'' number of folds $k$, in a $k$-fold cross validation experiment. We follow with a discussion of the merits and limitations of different forms of cross validation, and offer recommendations for use in practice.

This paper is organized as follows. Section \ref{sec:relevant} presents relevant previous work and section \ref{sec:J} discusses two approaches to obtain the resampling size $J$. Section \ref{sec:training/test} presents the analysis for optimally choosing the training sample size, while section \ref{sec:simu} presents simulation results. Section \ref{sec:real} applies the methods to real data sets, while section \ref{sec:conclusions} offers discussions and recommendations. Finally, proofs of the obtained results are presented in the Appendix, and also in the online supplement.

\section{Relevant work and notation}
\label{sec:relevant}

In this section we review briefly related work and set up the notation we will use in this paper.

Interest in the problem of ``optimal'' partitioning of a fixed sample $n$ in two sets, training and testing, dates back to 1962. \citet{Highleyman1962} studied this problem in the context of designing a pattern  recognition machine and defined the ``optimum'' partitioning of the total sample as that partitioning which minimizes the variance of an unbiased estimator of the error probability of the pattern recognition machine. \citet{Highleyman1962} states that the same rule, that is the minimization of variance of the error of a classification rule, can be used with estimators of the error that can be corrected for bias. Similar work has been carried out by \citet{LG1999}. These authors address the problem of splitting a fixed sample $n$ into a training, a testing and a validation set using as a criterion the mean squared error. Only the squared error loss is studied in the simple case of location parameter estimation. The cross validation procedures they study are hold-out cross validation, $k$-fold and randomized permutation cross validation. The last cross validation procedure is defined by resampling the test set $J$ times, $J\le{n\choose n_2}$, $n$ is the total sample size and $n_2$ is the test set sample size.

The same problem appears also in the Biostatistics literature and in particular in the area of Bioinformatics. \citet{DS2011} consider the problem of developing genomic classifiers, empirically address the question of what proportion of the samples should be devoted to the training set, and study the impact that this proportion has on the mean squared error of the prediction accuracy estimate. The context within which \citet{DS2011} address this question is the one where the number of predictors $p$ is considerably larger than the number of samples $n$, i.e.\ it is the high-dimensional, small sample size case. \citet{DS2011} evaluated two approaches to data split. The first consists of using results of simulations that were designed to understand qualitatively the relationships among data set characteristics and optimal split proportions to evaluate commonly used rules of thumb, such as allocating $1/2$ or $2/3$ of the data set to training and the remaining to the test sets. The second approach consists of developing a nonparametric procedure that is based on decomposing the mean squared error into three components, two corresponding to appropriately defined variance terms and one corresponding to a squared bias term. There are fundamental differences between the work of these authors and our work. First, their setting of small $n$, large $p$ is different than ours and secondly, even in this setting, they do not address the case of cross validation estimators.

In this paper, we do not address the high-dimensional, small sample size case. We address the allocation of cases to training/test set when the total sample size $n$ is fixed and known, reflecting a limited or medium size data with $p$ considerably smaller than $n$. In this setting, we study the rules for optimal split of the data for the random cross validation and the $k$-fold CV estimators of the generalization error, which we describe in detail below. But before we do this, we briefly describe two very relevant to our work papers that we use in setting the rules for selecting ``optimally'' the size of training, and hence test, set.

In a series of impressive papers \citet{Burman1989,Burman1990} studied the performance of $k$-fold cross validation and of the repeated learning-testing method, which is another name to describe random cross validation. In what follows, we will succinctly describe Burman's results. Reference to the $k$-fold cross validation (which Burman calls $v$-fold CV) the following, relevant to our paper, points are made. 1) In the case of $k$-fold CV with a small number of folds $k$, the uncorrected CV estimate may have large bias; thus, a corrected estimate is proposed, and it is denoted by ${\rm CV}_{nk}^*$; 2) the bias of the corrected $k$-fold CV estimator of the generalization error is {\it always smaller} than the bias of the uncorrected estimator; 3) however, $\Var({\rm CV}_{nk}^*-s_n)\simeq\Var({\rm CV}_{nk}-s_n)$ when $k>3$ and $n\ge24$; here, $s_n$ is defined as $s_n=\int L(z,F_n)\ud F(z)$, $L$ is a loss function, $F_n$ is the empirical cumulative distribution function of the sample $z$, and $F$ represents the underlying true and unknown distribution; 4) the bias and variance of the $k$-fold CV estimator of the generalization error decrease as the number of folds increases. The last two results indicate that for the case of $k$-fold CV estimator of the generalization error it is sufficient to study only the variance of the uncorrected estimator as opposed to the mean squared error of the uncorrected estimator. This is because the variance of the corrected, almost unbiased estimator, is approximately the same with the variance of the uncorrected estimator. This observation greatly simplifies the analysis of obtaining ``optimal'' splits of the sample.

For the case of repeated learning-testing method, or in our terminology random cross validation, the following points are relevant to the work presented here. 1) A bias correction term is needed to obtain an almost unbiased estimate of the generalization error. This term is provided in \citet{Burman1989}; 2) the variances of the corrected and uncorrected estimators decrease as the number of repeats increases. \citet{Burman1989} presents a simulation study where the two variances are approximately equal when the number of repetitions $J>4$ and the total sample size is $n\ge24$ for all values of the proportion of samples allocated to the test set; 3) the bias does not depend on the number $J$ of repeats; 4) the only way to reduce variance is by increasing the number of repeats $J$. We provide rules that allow ``optimal'' selection of the number of repeats in the random cross validation case. The relevance of these results for our work is that again we can use the variance of the uncorrected random CV estimate of the generalization error, instead of the more complicated mean squared error, to set ``optimal'' rules for data splitting into training and test sets. This is because the variance of the bias-corrected cross validation estimator is approximately the same with the variance of the uncorrected cross validation estimator.

Next we set the notation we use in this paper and explicitly discuss the two estimators of the generalization error we work with.

Fix a positive integer $n$ and consider the set $N=\{1,2,\ldots,n\}$. Let $Z_i$, $i=1,\ldots,n$ be data collected such that the data universe, $\CZ_N=\{Z_1,Z_2,\ldots,Z_n\}$, is a set of independent, identically distributed observations that follow an unknown distribution $F$. Let $S$ be a subset of size $n_1$, $n_1<n$, taken from $N$, $S^c\doteq N\smallsetminus S$, the complement of $S$ with respect to $N$. The subset of observations $\CZ_S\doteq\{Z_i\colon i\in S\}$ is called a training set, used for constructing a learning rule. The test set contains all data that do not belong in $\CZ_S$; it is defined as $\CZ_{S^c}\doteq\CZ_N\smallsetminus \CZ_S$, the complement of $\CZ_S$ with respect to the data universe $\CZ_N$. Let $n_2$ denote the number of elements in the test set, $n_2=n-n_1$, $n_2<n$.

Now let $L\colon \RR^p\times\RR\to\RR$ be a function, assume that $Y$ is a target variable, and $\h{f}(x)$ is a decision rule. The function $L\big(\h{f}(X),Y\big)$ that measures the error between the target variable and the decision rule is called a loss function. Examples of loss functions widely used in the literature, and in the present paper, include the squared error loss function, absolute error loss, and 0/1 loss function used in classification.

The {\it generalization error} of an algorithm is defined as
\[
\mu^{(n)}\doteq\E[L(\CZ_N,Z)],
\]
where $Z$ is an independent copy of the data $Z_i$ and the expectation is taken over everything that is random. That is, we take into account the variability in both, training and test set. Furthermore, let $S_j$, $j=1,2,\ldots,J$, be random sets such that
\[
\CS=\CS_{n,n_1}\doteq\{S\colon S\subset N=\{1,\ldots,n\}, \ \card(S)=n_1\},
\]
where $\card(\CS)={n\choose n_1}$; $S_j$, $j=1,2,\ldots,J$, are {\it uniformly distributed} on $\CS$, such that each $S_j$ is independently sampled, with corresponding complement set (with respect $N$) $S_j^c$. The number $J$ is called the {\it resampling size}. If $\CZ_{S_j}$ is defined similarly as above for all $j$, the usual {\it average test set error} is
\[
\h{\mu}_j\doteq\frac{1}{n_2}\sum_{i\in S_j^c}L(\CZ_{S_j},i),
\]
and is a function of both the training set $S_j$ and the test set $S_j^c$. \citet{NB2003} defined the {\it random cross validation estimator of the generalization error} as
\begin{equation}
\label{eq.mu_J}
\h{\mu}_{{\rm CV},J}\doteq\frac{1}{J}\sum_{j=1}^J\h{\mu}_j.
\end{equation}
By a straightforward calculation, it follows that
\begin{equation}
\label{eq.Var(mu_J).0}
\Var(\h{\mu}_{{\rm CV},J})=\frac{1}{J^2}\sum\limits_{j=1}^{J}\Var(\h{\mu}_j)+\frac{2}{J^2}\mathop{\sum\sum}\limits_{1\le{j}<j'\le{J}}\Cov(\h{\mu}_j,\h{\mu}_{j'}).
\end{equation}

In the cases where the prediction rule takes real values, \citet{Efron1986} presented a wide class of loss functions, called $q$-class. Specifically, if $q$ is an absolutely continuous and concave real function, the corresponding $q$-class loss function is
\[
^qL(\h{\mu},y)\doteq q(\h{\mu})+q'(\h{\mu})(y-\h{\mu})-q(y),
\]
where $q'(\cdot)$ is the almost sure derivative of the {\it generator} $q$. Commonly used loss functions, such as squared error loss and 0/1 loss belong to this class. In Section \ref{sec:training/test} we use the $q$-class of loss functions to elucidate the training sample size selection rules.

\section[The resampling size ${J}$]{The resampling size $\bm{J}$}
\label{sec:J}
We now provide guidance for the selection of the resampling size $J$ entering the construction of the random cross validation estimator of the generalization error. In Subsection \ref{ssec:res-eff} we discuss two novel methods for selecting the resampling size $J$, the $\pi$-effectiveness and the $r$-reduction methods. Here, we note that the $\pi$-effectiveness and $r$-reduction, and hence the selection of $J$, are affected by the training set sample size. We illustrate these relationships in Section \ref{sec:simu}.

We first improve upon the following result given in \citet{NB2003}. Under the assumption that ``the distribution of $L(\overline{X}_{S_j},X_i)$ does not depend on the particular realization of $S_j$, $i$'' \citet{NB2003} prove that the quantities $\Var(\h{\mu}_j)$ and $\Cov(\h{\mu}_j,\h{\mu}_{j'})$ in \eqref{eq.Var(mu_J).0} do not  depend on the indices $j$ and $j'$ \citep[see, also,][]{MTBH2005}. In what follows we first prove that the aforementioned statement is generally true, that is this assumption is not needed for the result to hold. This is shown in Proposition \ref{prop.v,c} of this section.

\begin{proposition}
\label{prop.v,c}
Let $n_1$ and $n_2$ be fixed, and let $L$ be a loss function such that $\E[L^2(\CZ_{S_j},i)]$ is finite for each realization of $S_j$ and $i$, where $S_j$ follow a uniform distribution (described in detail in the proof of the proposition). Then, the quantities $\Var(\h{\mu}_j)$ and $\Cov(\h{\mu}_j,\h{\mu}_{j'})$ are finite and do not depend on the indices $j$ and $j'$, for both random and $k$-fold cross validation.
\end{proposition}
\begin{proof}
First, $\E[L^2(\CZ_{S_j},i)]<\infty$ gives $\Var[L(\CZ_{S_j},Y_i)]<\infty$, and an application of Cauchy--Schwarz inequality gives that $\Cov\big(L(\CZ_{S_j},i),L(\CZ_{S_j},{i'})\big)$ exists for each realization of $S_j$, $i$ and $i'$. Thus, by definition of $\h{\mu}_j$, $\Var(\h{\mu}_j)<\infty$. Again by Cauchy--Schwarz inequality we obtain that $\Cov\big(\h{\mu}_j,\h{\mu}_{j'}\big)$ exists.

In the case of random CV, the random vector $(\h{\mu}_j,\h{\mu}_{j'})^\ta$ is uniformly distributed on
\[
\CW\doteq\{(X_i,X_{i'})^\ta\colon X_i\doteq\h{\mu}_j|S_j=A_i, \ X_{i'}\doteq\h{\mu}_{j'}|S_j=A_{i'}, \ A_i,A_{i'}\in\CS\},
\]
with $\card(\CW)={n\choose n_1}^2$. Using the total probability theorem, the cumulative distribution function of $(\h{\mu}_{j},\h{\mu}_{j'})^\ta$ is
\[
F_{\h{\mu}_{j},\h{\mu}_{j'}}(x,x')
={n\choose n_1}^{-2}\mathop{\sum\sum}_{(X_i,X_{i'})^\ta\in\CW}F_{X_i,X_{i'}}(x,x')
\doteq F(x,x').
\]
From the form of $F$ we obtain $(\h{\mu}_{j},\h{\mu}_{j'})^\ta\stackrel{\mathrm{d}}{=}(\h{\mu}_{j'},\h{\mu}_{j})^\ta$; and their distribution does not depend on the indices $j$ and $j'$.

In the $k$-fold CV case set
\[
\CW^*\doteq\{(X_i,X_{i'})^\ta\colon X_i\doteq\h{\mu}_j|S_j=A_i, \ X_{i'}\doteq\h{\mu}_{j'}|S_j=A_{i'}, \ A_i,A_{i'}\in\CS
             \ \textrm{with} \ A_i^c\cap A_{i'}^c=\varnothing\},
\]
with $\card(\CW^*)=n!\big/\big\{[(n/k)!]^2(n-2n/k)!\big\}$, and observe that the random vector $(\h{\mu}_j,\h{\mu}_{j'})^\ta$ is uniformly distributed on $\CW^*$. Using the same arguments as above the proof is completed.
\hfill$\square$
\end{proof}

Proposition \ref{prop.v,c} allows one to write $\Var(\h{\mu}_j)=\sfv$, $\Cov(\h{\mu}_j,\h{\mu}_{j'})=\sfc$, effectively obtaining
\begin{equation}
\label{eq.Var(mu_J)}
\Var(\h{\mu}_{{\rm CV},J})=\frac{\sfv-\sfc}{J}+\sfc.
\end{equation}
That is, the variance of the random cross validation estimator of the generalization error is a function of the variance of the average test set error and the covariance between two different average test set errors, which are constants with respect to $j$, $j'$. Figure \ref{fig.J,var(mu_J)} shows the behavior of $\Var(\h{\mu}_{{\rm CV},J})$ as a function of $J$. By Cauchy-Schwarz inequality we obtain that $\sfv-\sfc\ge0$, where the equality characterizes the trivial cases in which $L$ is a constant with probability 1, e.g., if $F$ is degenerate. So (see Figure \ref{fig.J,var(mu_J)}),
\begin{equation}
\label{eq.Var(mu_J)2}
\Var(\h{\mu}_{{\rm CV},J})\searrow \sfc\ge0,\quad\textrm{as} \ J\to\infty,
\end{equation}
and again the equality characterizes the same trivial cases.
\begin{figure}[htp]
\centering{\includegraphics{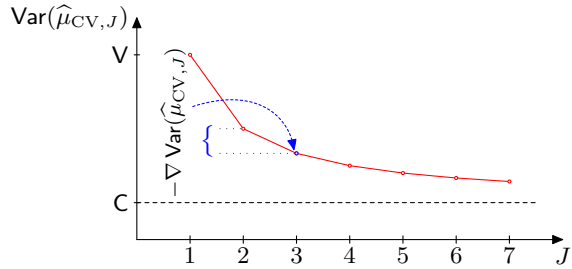}}
\caption{$\big(J,\Var(\h{\mu}_{{\rm CV},J})\big)$-plot for $J=1,\ldots,7$, for the non-trivial case in which $0<\sfc<\sfv$. $\nabla g(J)=g(J)-g(J-1)$, for a general function $g$.}
\label{fig.J,var(mu_J)}
\end{figure}

From \eqref{eq.Var(mu_J)} and \eqref{eq.Var(mu_J)2} we obtain the following important result.
\begin{theorem}
\label{theo.var_bound}
The variance of $\h{\mu}_{{\rm CV},J}$ given by \eqref{eq.mu_J} satisfies the following double inequality:
\[
\max\left\{\sfc,\frac{\sfv}{J}\right\}\le\Var(\h{\mu}_{{\rm CV},J})\le\sfv,
\]
where $\sfv=\Var(\h{\mu}_j)$ and $\sfc=\Cov(\h{\mu}_j,\h{\mu}_{j'})$.
\end{theorem}
\begin{proof}
Using Cauchy--Schwarz inequality and \eqref{eq.Var(mu_J)2} we get $0\le\rho\le1$, where $\rho=\sfc/\sfv=\Corr(\h{\mu}_j,\h{\mu}_{j'})$. Also, \eqref{eq.Var(mu_J)} gives $\Var(\h{\mu}_{{\rm CV},J})=\left(\frac{1}{J}+\frac{J-1}{J}\rho\right)\sfv$. Finally, for each $J$ fixed, the function $f(\rho)=\frac{1}{J}+\frac{J-1}{J}\rho$, $0\le\rho\le1$, is increasing, and in view of \eqref{eq.Var(mu_J)2} the proof is completed.
\hfill$\square$
\end{proof}

The importance of the theorem will become obvious when we define appropriate rules for selecting the sample size of the training set. Since the variance of the CV estimator of the generalization error is bounded above by the variance of the test error $\sfv$, we can use $\sfv$ to create an ``optimal'' rule for selecting the training (and hence the test) sample size.

\subsection{The $\pi$-effectiveness and $r$-reduction acceptable resampling size.}
\label{ssec:res-eff}
We discuss two novel methods for selecting the resampling size $J$.

The quantities $\sfv$ and $\sfc$ in \eqref{eq.Var(mu_J)} depend on the sizes of the training and test set, $n_1$ and $n_2$, as well as on $F$ and $L$. Recall that an experimentalist selects that value of $J$ for which there is no appreciable reduction in the variance of CV estimator of the generalization error. Taking into account the behavior of $\Var(\h{\mu}_{{\rm CV},J})$ we give the following definition.

\begin{definition}
\label{defi.re,rr}
\rm
We define:
\item{(a)} The {\it resampling effectiveness} of $\h{\mu}_{{\rm CV},J}$ by  $\re(\h{\mu}_{{\rm CV},J})\doteq\frac{\sfc}{\Var(\h{\mu}_{{\rm CV},J})}$.
\item{(b)} The {\it reduction ratio} of $\h{\mu}_{{\rm CV},J}$ by $\rr(\h{\mu}_{{\rm CV},J})\doteq\frac{-\nabla\Var(\h{\mu}_{{\rm CV},J})}{\Var(\h{\mu}_{{\rm CV},J})}$, where $\nabla$ denote the backward difference with respect to $J$.
\end{definition}

Notice that while $\re(\h{\mu}_{{\rm CV},J})$ provides a comparison of the variance of $\h{\mu}_{{\rm CV},J}$ for a given $J$ with the limiting variance $\sfc$, the $\rr(\h{\mu}_{{\rm CV},J})$ provides a comparison between the two variances for given $J$ and $J-1$. In this sense $\rr(\h{\mu}_{{\rm CV},J})$ is a local measure of change in variance that is used to obtain $J$. This is what experimentalists are observing in order to set the value of $J$.

\begin{figure}[htp]
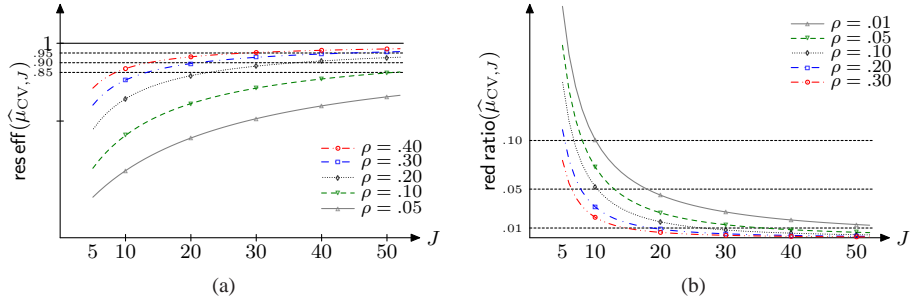

\begin{subfigure}[a]{.48\textwidth}
\includegraphics[width=1\linewidth]{Fig.21}
\caption{}
\label{sfig.ResEff}
\end{subfigure}
\hfill
\begin{subfigure}[a]{.48\textwidth}
\includegraphics[width=1\linewidth]{Fig.22}
\caption{}
\label{sfig.RedRatio}
\end{subfigure}
\caption{The behavior of: (a) the resampling effectiveness of $\h{\mu}_{{\rm CV},J}$ for $J=5,\ldots,55$ and $\rho=.05,.1,.2,.3,.4$; (b) the reduction ratio of $\h{\mu}_{{\rm CV},J}$ for $J=5,\ldots,55$ and $\rho=.01,.05,.1,.2,.3$.}
\label{fig.J-re,J-rr}
\end{figure}

Now we are in a position to give a minimum acceptable resampling size, either via the resampling effectiveness of $\h{\mu}_{{\rm CV},J}$ or via its reduction ratio. Observe that in non-trivial cases the resampling effectiveness of $\h{\mu}_{{\rm CV},J}$ takes a value in the interval $(0,1)$; for the trivial cases this is assumed to be 1. From \eqref{eq.Var(mu_J)} it follows immediately that
\begin{equation}
\label{eq.reseff}
\re(\h{\mu}_{{\rm CV},J})=\left(1+\frac{1-\rho}{\rho J}\right)^{-1},
\end{equation}
where $\rho=\sfc/\sfv=\Corr(\h{\mu}_j,\h{\mu}_{j'})$. Figure \ref{sfig.ResEff} shows the behavior of the resampling effectiveness of $\h{\mu}_{{\rm CV},J}$ for various $J,\rho$-values. Notice that when the correlation is low a larger resampling size is required in order to achieve high resampling effectiveness. As the correlation value increases the resampling size decreases; for example, when $\rho=.2$, $J=36$ is sufficient to obtain a resampling effectiveness of 90\%, while when $\rho=.3$, $J=21$ obtains the same resampling effectiveness.

If $\pi$ is the desired resampling effectiveness rate, then $\re(\h{\mu}_{{\rm CV},J})\ge\pi$, and the minimum value of the resampling number $J$ is
\begin{equation}
\label{eq.Ja}
J_{\mathsf{re}}(\pi)=\left\lfloor\frac{\pi(1-\rho)}{(1-\pi)\rho}\right\rfloor,
\end{equation}
where $\lfloor x \rfloor$ denotes the upper integer part of $x$. We call the number $J_{\mathsf{re}}(\pi)$ as the {\it $\pi$-effectiveness minimum resampling size}.

Note that, we define the resampling effectiveness (and thus select that value of $J$) as a number indicating percentage of the ratio $\sfc\big/{\Var(\h{\mu}_{{\rm CV},J})}$, which is between 0 and 1. The bigger the ratio ${\sfc}\big/{\Var(\h{\mu}_{{\rm CV},J})}$, the closer the $\Var(\h{\mu}_{{\rm CV},J})$ is to the asymptotic value $\sfc$ of the variance of the CV estimator of the generalization error. That indicates that the value of $\pi$ should be close to 1, i.e.\ .8, .9, .95 etc.

An alternative way to obtain $J$ is as follows. From \eqref{eq.Var(mu_J)} we obtain
\begin{equation}
\label{eq.redratio}
\rr(\h{\mu}_{{\rm CV},J})=\frac{1-\rho}{(J-1)+(J-1)^2\rho}.
\end{equation}
We then fix the desired reduction ratio to a value $r>0$ and require that the reduction ratio of $\h{\mu}_{{\rm CV},J}$ should not exceed this value, that is,  $\rr(\h{\mu}_{{\rm CV},J})\le r$. The minimum positive integer satisfying the preceding inequality is given by
\begin{equation}
\label{eq.Jr}
J_{\mathsf{rr}}(r)=\left\lfloor1-\frac{1}{2\rho}+\sqrt{\frac{1}{4\rho^2}+\frac{1-\rho}{\rho r}}\right\rfloor,
\end{equation}
which we call the {\it $r$-reduction ratio minimum resampling size}. We define the reduction ratio ${-\nabla\Var(\h{\mu}_{{\rm CV},J})}\big/{\Var(\h{\mu}_{{\rm CV},J})}$ (and thus select that value of $J$) as a number indicating the relative reduction of the variance of the CV estimator, which is a positive number. The smaller this ratio, the closer $\Var(\h{\mu}_{{\rm CV},J})$ is to the asymptotic value $\sfc$ of the variance of the CV estimator of the generalization error. That indicates that the value of $r$ should be close to 0, i.e.\ .1, .05, .025, .01 etc.

Figure \ref{sfig.RedRatio} shows the reduction ratio curves as a function of the correlation $\rho$. Again, the smaller the value of the correlation between $\h{\mu}_j$, $\h{\mu}_{j'}$, the larger the value of $J$ to obtain a desired variance reduction ratio.

\begin{table}[htp]
 \centering{\normalsize
 \caption{Specific values of $J_{\mathsf{re}}(\pi)$ for $\pi=.8,.85,.9,.95$ and $J_{\mathsf{rr}}(r)$ for $r=.1,.05,.025,.01$, for various values of $\rho$.}
 \label{table.Jre,Jrr}
 \begin{tabular}
 {@{\hspace{0ex}}l@{\hspace{7ex}}c@{\hspace{1ex}}l@{\hspace{5.5ex}}c@{\hspace{5.5ex}}c@{\hspace{5.5ex}}r@{\hspace{0ex}}
                                 c@{\hspace{5.5ex}}
                                 c@{\hspace{1ex}}c@{\hspace{5.5ex}}c@{\hspace{5.5ex}}c@{\hspace{5.5ex}}r@{\hspace{0ex}}}
 \addlinespace
 \toprule
  & & \multicolumn{4}{c}{$\bm{J_{\mathsf{re}}(\pi)}$} & & & \multicolumn{4}{c}{$\bm{J_{\mathsf{rr}}(r)}$}\\
    \cline{2-6} \cline{8-12}
    $~\rho$ & $\pi:$ & .8 & .85 & .9 & .95 & & $r:$ & .1 & .05 & .025 & .01 \\
  \hline
 .2 & & 16 & 23 & 36 & 76 & & & 6 & 8 & 12 & 19 \\
 .3 & & 10 & 14 & 21 & 45 & & & 5 & 7 & 10 & 15 \\
 .4 & & \hspace{1.15ex}6 & \hspace{1.15ex}9 & 14 & 29 & & & 4 & 6 & \hspace{1.15ex}8 & 13 \\
 .5 & & \hspace{1.15ex}4 & \hspace{1.15ex}6 & \hspace{1.15ex}9 & 19 & & & 4 & 5 & \hspace{1.15ex}7 & 11 \\
 .6 & & \hspace{1.15ex}3 & \hspace{1.15ex}4 & \hspace{1.15ex}6 & 13 & & & 3 & 4 & \hspace{1.15ex}6 &  9 \\
 .7 & & \hspace{1.15ex}2 & \hspace{1.15ex}3 & \hspace{1.15ex}4 & \hspace{1.15ex}9 & & & 3 & 4 & \hspace{1.15ex}5 &  7 \\
 \bottomrule
 \end{tabular}}
 \end{table}

Table \ref{table.Jre,Jrr} contains the specific values of $J_{\mathsf{re}}(\pi)$ and $J_{\mathsf{rr}}(r)$ for various values of $\pi$, $r$ and $\rho$. Notice that large values of $J$ are required when the correlation is small and the desired resampling effectiveness and reduction ratio are respectively large or small. Furthermore, the smaller the reduction ratio the larger the value of $J$, for any given correlation $\rho$. That is, when $\rho=.3$ and the desired $r=.01$, one obtains $J=15$. This selection corresponds to resampling effectiveness of approximately $86.5\%$.

\begin{remark}
\label{rem.est-v,c}
Equations \eqref{eq.Ja}, \eqref{eq.Jr} depend on $\rho$ that is generally unknown. This parameter needs to be estimated from the data. \citet{NB2003} propose an approximation of the correlation coefficient given by $\h{\rho}=n_2/n$, where $n_2$ is the cardinality of the test set. This estimator, however, may overestimate or underestimate the true correlation. \citet{MTBH2005} suggested moment approximation estimators of $\sfv$ and $\sfc$ in certain cases that may be used to estimate the correlation $\rho$, and hence obtain the values $J_{\mathsf{re}}(\pi)$ and $J_{\mathsf{rr}}(r)$.  Additionally, if the loss function belongs in Efron's $q$-class of loss functions, then the parameters can be easily estimated. See Sub-subsection \ref{sssec:q-class} for details.
\end{remark}

\section{Optimal choice of training set size}
\label{sec:training/test}

A point that relates to the design of the random cross validation estimator of the generalization error is the selection of the training set. Practitioners select the training sample size $n_1$ to be 5 or even 10 times larger than the test size $n_2$. Our results indicate that when ``optimality'' is quantified via minimization of the variance of the average test set error, the training sample size, for fixed value of $n$, equals the test set sample size. In what follows, we present and justify our choice of the optimality rule and apply it in Subsections \ref{ssec:mean} and \ref{ssec:regression}.

Relation \eqref{eq.Var(mu_J)} leads to the following two potential rules. The first rule, suggesting minimization of the limiting variance $\sfc$, is not useful because, as we will see in Section \ref{sec:simu}, $\sfc$ is a constant. The second rule suggests that minimization of the $\Var(\h{\mu}_{{\rm CV},J})$ is equivalent to minimizing the variance $\sfv$ of $\h{\mu}_j$ (this is equivalent to minimizing $\sfv-\sfc$). In light of Theorem \ref{theo.var_bound}, the variance of $\h{\mu}_{{\rm CV},J}$ is bounded above by $\sfv$ and places a small upper bound on the limiting variance $\sfc$. These observations lead us to select quantifying ``optimality'' via this rule. Therefore, for fixed $n$, we define the {\it optimal value of the training set size} by
\begin{equation}
\label{eq.opt-n1}
n_1^{\rm opt}\doteq\arg\min_{n_1\ge \frac{n}2}\{\sfv\}.
\end{equation}

The situation is different for the $k$-fold CV. Because in $k$-fold CV there is dependence between $k$, $n_1$ and $n_2$, for fixed $n$, the aforementioned optimization rules are not useful. We propose selecting the optimal value of $k$ as
\begin{equation}
\label{eq.opt-k-fold}
k^{\rm opt}\doteq\arg\min_{k}\{\Var(\h{\mu}_{k\textrm{-fold}})\}.
\end{equation}

To be able to work with the optimization rules given in \eqref{eq.opt-n1}, \eqref{eq.opt-k-fold} we need either the exact or an approximate form, of the variances involved. In what follows we consider three kinds of problems to cover a good range of possible applications. These correspond to using the sample mean or linear regression (with and without normality of the errors) as decision rules and classification via logistic regression.

\subsection{The case of sample mean}
\label{ssec:mean}

We begin by studying the simple case where the decision rule is the sample mean. The variance formulas we use in our rules, and the conditions under which they are obtained, are given in \citet[pp.~1131,1138]{MTBH2005}. Here, for reasons of completeness, we briefly summarize the results we use.

Let $Y=\card(S_j\cap S_{j'})$ and $Y^*=\card(S_j^c\cap S_{j'}^c)$. \citet[Theo.~3.1]{MTBH2005} gave the following moment approximations of $\sfv$, $\sfc$
\begin{equation}
\label{eq.v,c}
\sfv=\frac{\alpha}{n_1}+\frac{\beta}{n_2}+\frac{\gamma+\delta}{n_1n_2}+O(\sfrac{1}{n_1^2}),
\
\sfc=\frac{\alpha\E(Y)}{n_1^2}+\frac{\E(Y^*)}{n_2^2}\left[\beta+\frac{\gamma\E(Y)}{n_1^2}+\frac{\delta}{n_1}\right]+O(\sfrac{1}{n_1^2}),
\end{equation}
where
\begin{equation}
\label{eq.a,b,c,d}
\begin{split}
\alpha&\doteq\sigma^2\E^2[L'_\mu(X)],
\qquad\qquad
~\beta\doteq\Var[L_\mu(X)],\\
\gamma&\doteq\sigma^2\Var[L'_\mu(X)],
\qquad\qquad
\delta\doteq\sigma^2\Cov[L_\mu(X),L''_\mu(X)],
\end{split}
\end{equation}
with $L^{(i)}_\mu(x)\equiv \frac{\ud^i}{\ud u^i}L(u,x)\big|_{u=\mu}$, $i=0,1,2$.

Note that, since $n_1\ge n/2$, the notation $O(\sfrac{1}{n_1^2})$ has the same meaning with the notation $O(\sfrac{1}{n^2})$, which will be used in the sequel.

In random CV each of $Y$ and $Y^*$ follows a hypergeometric distribution, see \citet[Lemmas 3.1, 3.2]{MTBH2005}. Specifically, $\E(Y)=n_1^2/n$ and $\E(Y^*)=n_2^2/n$; thus, the approximated covariance, $\sfc$, in this case is
\begin{equation}
\label{eq.c-rCV}
\sfc\simeq\frac{\alpha+\beta}{n}+\frac{\gamma}{n^2}+\frac{\delta}{n_1n}.
\end{equation}
In $k$-fold case $Y$ and $Y^*$ are constants; specifically,
\begin{equation}
\label{eq.Y,Y*_k-fold}
Y=\frac{(k-2)n}{k} \quad \textrm{and} \quad Y^*=0.
\end{equation}

\subsubsection{Training set sample size optimality in Random CV}
\label{sssec:mean.usual}
To solve the optimization problem \eqref{eq.opt-n1} we consider the expression for the variance of $\h{\mu}_j$ given by \eqref{eq.v,c}. When $n$ is sufficiently large, the term $O(\sfrac{1}{n^2})$ is a negligible quantity, and using the fact that $n_2=n-n_1$, and the identity $n/[n_1(n-n_1)]=1/n_1+1/(n-n_1)$, \eqref{eq.v,c} becomes
\begin{equation}
\label{eq.appr-V}
\sfv\simeq\frac{A}{n_1}+\frac{B}{n-n_1},
\quad\textrm{where} \
A\doteq\alpha+\frac{\gamma+\delta}{n}, \ B\doteq\beta+\frac{\gamma+\delta}{n}.
\end{equation}
Furthermore, by definition $\beta>0$; hence, as $n$ becomes large the parameter $B$ takes positive values. Under this observation,
using the approximation formula \eqref{eq.appr-V} the solution to the optimization problem in \eqref{eq.opt-n1} is given by Theorem \ref{theo.opt-n_1}.
\begin{theorem}
\label{theo.opt-n_1}
The solution of the optimization problem \eqref{eq.opt-n1} when $\sfv$ is given by \eqref{eq.appr-V} is
\[
n_1^{\rm opt}=
\left\{
\begin{array}{c@{\hspace{3ex}if\hspace{2ex}}c}
\lfloor n/2 \rfloor, & A\le B,\\
\min\left\{n-1,\left\lfloor\frac{\sqrt{A}}{\sqrt{A}+\sqrt{B}}n\right\rceil\right\}, & A>B,\\
\end{array}
\right.
\]
where $A$, $B$ are defined by \eqref{eq.a,b,c,d} and \eqref{eq.appr-V}, and $\lfloor x \rceil$ stands for the nearest integer of $x$.
\end{theorem}
\begin{proof}
Obviously, $n_1^{\rm opt}=\arg\min_{t\in\{\lfloor n/2\rfloor,\ldots,n-1\}}\{g(t)\}$, where $g(t)=\frac{A}{t}+\frac{B}{n-t}$, $0<t<n$.
If $A\le0$ the desired result becomes trivial.

We will study the case $A>0$. The derivative of $g$ is $g'(t)=\frac{Bt^2-A(n-t)^2}{t^2(n-t)^2}$, and it has  the same sign with the numerator
\[
N(t)=Bt^2-A(n-t)^2
    =[(A^{1/2}+B^{1/2})t-A^{1/2}n][(B^{1/2}-A^{1/2})t+A^{1/2}n].
\]
If $A=B$, $N(t)$ is a linear polynomial with root $t_0=n/2$, and takes negative values before $t_0$ and positive values after this. If $A\ne B$, $N(t)$ is a quadratic polynomial with two distinct roots $t_1=\frac{\sqrt{A}}{\sqrt{A}+\sqrt{B}}n\in(0,n)$ and $t_2=\frac{\sqrt{A}}{\sqrt{A}-\sqrt{B}}n$. When $0<A<B$ then $t_2<0$, and when $A>B$ then $t_2>n$; for both cases we see that $N(t)<0$ for all $t\in(0,t_1)$ and $N(t)>0$ for all $t\in(t_1,n)$. By definition of $n_1^{\rm opt}$ and the monotonicity of $g$ the proof is completed.
\hfill$\square$
\end{proof}

In practice, the numbers $A$ and $B$ (namely, the parameters $\alpha$, $\beta$, $\gamma$ and $\delta$) are unknown and must be estimated. But there are cases where the estimation of these parameters is unnecessary; i.e., when for the theoretical values $A$ and $B$ we know that $A \le B$ or $A>B$, independently of the data distribution (e.g., see the following example).

\begin{example}
\label{exm.Sq-n_1opt}
Assume that the data are from a population with finite eighth moment and the loss function is squared error, that is $L(\overline{X}_{S_j},X_i)=(X_i-\overline{X}_{S_j})^2$, $i\in S_j^c$. Then, $L_\mu(X)=(X-\mu)^2$ with derivatives $L'_\mu(X)=-2(X-\mu)$ and $L''_\mu(X)=2$. We compute $\alpha=0$, $\beta=\mu_4-\sigma^4>0$, $\gamma=4\sigma^4>0$ and $\delta=0$, where $\sigma^2=\Var(X)$ and $\mu_4=\E(X-\mu)^4$. Therefore, $A=4\sigma^4/n<\mu_4-\sigma^4+4\sigma^4/n=B$; and so, $n_1^{\rm opt}=\lfloor n/2 \rfloor$, independently of the distribution of the data.
\end{example}

\subsubsection{Optimality of $k$ in $k$-fold CV}
\label{sssec:mean.k-fold}
Using \eqref{eq.Y,Y*_k-fold}, \eqref{eq.v,c} gives $\sfv\simeq\frac{k\alpha}{(k-1)n}+\frac{k\beta}{n}+\frac{k^2(\gamma+\delta)}{(k-1)n^2}$ and $\sfc\simeq\frac{k(k-2)\alpha}{(k-1)^2n}$. Also, from \eqref{eq.Var(mu_J)} and with $J=k$, we obtain $\Var(\h{\mu}_{k\textrm{-fold}})=\frac{1}{k}\sfv+\frac{k-1}{k}\sfc$. From the preceding relations, after some algebra, we obtain
\begin{equation}
\label{eq.Var(mu-k-fold)}
\Var(\h{\mu}_{k\textrm{-fold}})\simeq\frac{\alpha+\beta}{n}+\frac{k}{k-1}\cdot\frac{\gamma+\delta}{n^2}.
\end{equation}

\begin{proposition}
\label{prop.opt-k}
The solution to optimization problem \eqref{eq.opt-k-fold} where the approximation variance of $\h{\mu}_{k\textrm{-fold}}$ is given by \eqref{eq.Var(mu-k-fold)} is
\[
k^{\rm opt}=
\left\{
\begin{array}{c@{\hspace{3ex}if\hspace{2ex}}c}
\md(n), & \gamma+\delta\le0,\\
  n,    & \gamma+\delta>0,\\
\end{array}
\right.
\]
where $\gamma$ and $\delta$ are defined by \eqref{eq.a,b,c,d}, and $\md(n)$ is the minimum divisor of $n$ that is greater than $1$.
\end{proposition}

\begin{proof}
Since the function $\frac{k}{k-1}$ decreases in $k$, the result is obtained in view of relations \eqref{eq.opt-k-fold} and \eqref{eq.Var(mu-k-fold)}.
\hfill$\square$
\end{proof}

\begin{remark}
\label{rem.opt-k}
Proposition \ref{prop.opt-k} states that if $\gamma+\delta>0$ then, the optimal value of $k$ in terms of minimization of variance of $\h{\mu}_{k\textrm{-fold}}$ is given by the leave-one-out CV (LOOCV). However, one can replace LOOCV by $k$-fold CV, for the large values of $n$ because the quantity $\frac{\gamma+\delta}{n^2}=O(\sfrac{1}{n^2})$ and the ratio $\frac{k}{k-1}$ satisfies the inequality
\[
1<\frac{n}{n-1}\le\frac{k}{k-1}\le2.
\]
Therefore, for relatively small values of $k$ the ratio $k/(k-1)$ takes values close to $1$. For example, if $k\in\{5,\ldots,11\}$ the ratio $k/(k-1)$ is between $1.25$ and $1.1$, and selecting $k$ between $5$ and $10$, i.e.\ $k=10$, guarantees that the variance of the generalization error of $\h{\mu}_{k\textrm{-fold}}$ is close to the minimum variance.
\end{remark}

In practice, the parameters $\gamma$ and $\delta$ are unknown and, usually, must be estimated. But, sometimes this estimation is unnecessary; for example, consider the squared error loss, then $\gamma+\delta=4\sigma^4>0$ (see Example \ref{exm.Sq-n_1opt}), and thus the LOOCV is proposed, independently of the distribution of the data. Example \ref{exm.Sq.k-fold} illustrates the comments in Remark \ref{rem.opt-k} for the cases in which the LOOCV is proposed by Proposition \ref{prop.opt-k}.

\begin{example}
\label{exm.Sq.k-fold}
Let the data be from $N(\mu,\sigma^2)$ distribution and the squared error loss is used. Then, $\alpha=0$, $\beta=2\sigma^4$, $\gamma=4\sigma^4$ and $\delta=0$; and thus, $\Var(\h{\mu}_{k\textrm{-fold}})=\sigma^4\left(\frac{2}{n}+\frac{4k}{(k-1)n^2}\right)$. Defining the relative efficiency of $\h{\mu}_{k\textrm{-fold}}$ as the ratio of the variance of the estimator for $k$ folds over its variance at $k^{\rm opt}=n$ folds, we have that the relative efficiency is given by $\frac{2/n+4k/[(k-1)n^2]}{2/n+4/[(n-1)n]}$. Table \ref{table.RE.Sq.Norm.k-fold} shows the specific values of this relative efficiency for various values of $n$ and $k$, indicating that the relative efficiency is a function of both, sample size and number of folds. It approaches 1 when $n\ge50$ and $k\ge5$.

\begin{table}[htp]
 \caption{Relative efficiency of $\h{\mu}_{k\textrm{-fold}}$ for squared error loss, normally distributed sample and various values of $n$ and $k$.}
 \label{table.RE.Sq.Norm.k-fold}
 \centering{\footnotesize
 \begin{tabular}
 {@{\hspace{0ex}}l@{\hspace{1.5ex}}c@{\hspace{1.5ex}}c@{\hspace{1.5ex}}c@{\hspace{1.5ex}}c@{\hspace{1.5ex}}r@{\hspace{0ex}}}
 \addlinespace
 \toprule
 \multicolumn{6}{c}{Relative efficiency of $\h{\mu}_{k\textrm{-fold}}$} \\
 \hline
     ${\ds n=24, \ k=2 \atop \ds 1.073}$
   & ${\ds n=30, \ k=2 \atop \ds 1.060}$
   & ${\ds n=30, \ k=3 \atop \ds 1.029}$
   & ${\ds n=40, \ k=2 \atop \ds 1.046}$
   & ${\ds n=40, \ k=4 \atop \ds 1.015}$
   & ${\ds n=50, \ k=2 \atop \ds 1.038}$ \\
 [2.5ex]
     ${\ds n=50, \ k=5 \atop \ds 1.009}$
   & ${\ds n=100,\ k=5 \atop \ds 1.005}$
   & ${\ds n=100,\ k=10\atop \ds 1.002}$
   & ${\ds n=150,\ k=5 \atop \ds 1.003}$
   & ${\ds n=150,\ k=10\atop \ds 1.001}$
   & ${\ds n=150,\ k=15\atop \ds 1.001}$ \\
 \bottomrule
 \end{tabular}}
 \end{table}

\end{example}

\subsubsection{Application to Efron's $q$-class of loss functions}
\label{sssec:q-class}
The results of sections \ref{sssec:mean.usual} and \ref{sssec:mean.k-fold} apply to the $q$-class of loss functions. We have
\[
^qL(\overline{X}_{S_j},X_i)=q(\overline{X}_{S_j})+q'(\overline{X}_{S_j})(X_i-\overline{X}_{S_j})-q(X_i),
\quad i\in S_j^c\quad\textrm{for all} \ j=1,\ldots,J;
\]
where $q$ is differentiable having at least up to five derivatives, and such that the fourth derivative of $^qL$ is bounded.

The main gain is that the quantities $\alpha$, $\beta$, $\gamma$ and $\delta$ (therefore $A$ and $B$ too) can be computed easily. Observe that
\[
^qL_\mu^{(i)}(x)\equiv\frac{\ud^i}{\ud\mu^i} {^qL}(\mu,x)=q^{(i+1)}(\mu)(x-\mu)-(i-1)q^{(i)}(\mu),\quad i=1,\ldots,4.
\]
Hence,
\[
\begin{array}{l@{\qquad}l}
\alpha=0, & \beta=[q'(\mu)]^2\sigma^2+\Var[q(X)]-2q'(\mu)\Cov[X,q(X)],\\
\gamma=[q''(\mu)]^2\sigma^4, &\delta=q'(\mu)q'''(\mu)\sigma^4-q'''(\mu)\Cov[X,q(X)]\sigma^2.
\end{array}
\]
\begin{remark}
\label{rem.est-a,b,c,d-q}
The function $q$ is a known function, but the parameters are unknown (unless the distribution $F$ is known). We estimate these parameters from the sample values $X_1,\ldots,X_n$ by replacing the parameters $\mu$, $\sigma^2$, $\Var[q(X)]$ and $\Cov[X,q(X)]$ by $\overline{X}=\frac{1}{n}\sum_{i=1}^{n}X_i$, $\h{\sigma}^2=\frac{1}{n-1}\sum_{i=1}^{n}(X_i-\overline{X})^2$, $\h{\Var}_q=\frac{1}{n-1}\sum_{i=1}^{n}(q(X_i)-\overline{q})^2$ and $\h{\Cov}_{X,q}=\frac{1}{n-1}\sum_{i=1}^{n}(X_i-\overline{X})[q(X_i)-\overline{q}]$ respectively, where $\overline{q}=\frac{1}{n}\sum_{i=1}^{n}q(X_i)$.
\end{remark}

\begin{remark}
\label{rem.Copt_n1opt=n/2}
Because $\alpha=0$ and $\beta>0$ for each loss function in the $q$-class, Theorem \ref{theo.opt-n_1} guarantees that $n_1^{\rm opt}=\lfloor n/2 \rfloor$, for all data distributions and all loss functions that belong in the $q$-class. Further, $\sfv^{\rm opt}\simeq2\beta/n$ and $\sfc^{\rm opt}\simeq\beta/n$. Hence, the optimal value of the correlation coefficient is $\rho^{\rm opt}\simeq1/2$. Since $\sfc$ is unaffected by the splitting, for training sample size $n_1$ we have $\rho=\sfc/\sfv=\sfc^{\rm opt}/\sfv\le\sfc^{\rm opt}/\sfv^{\rm opt}=1/2$. Therefore, the choice $n_1=n_1^{\rm opt}$ leads to the optimized values of resampling effectiveness and reduction ratio of $\h{\mu}_{{\rm CV},J}$, see Figure \ref{fig.J-re,J-rr}.
\end{remark}

\subsection{The regression case}
\label{ssec:regression}

In the regression case we solve the problem of optimal training sample size identification under squared error loss. Furthermore, we offer a solution under both, normality of the errors and under relaxation of the normality assumption. In the second case we require $\E(|\varepsilon_i|^{4+\epsilon})<\infty$ for some $\epsilon>0$. The results presented below are different from the results presented in \citet{MTBH2005} in following aspects: 1) the analysis presented in \citet{MTBH2005} is conditional on the training sample $S_j$; here, we use the distribution of $S_j$ in light of Proposition \ref{prop.v,c}; 2) we study the case of $k$-fold CV; and 3) we present closed form expressions of the expectation, variance and covariance between two test set errors.

We first derive expressions for the quantities $\sfv$, $\sfc$ entering the computation of the variance of $\h{\mu}_{{\rm CV},J}$ by exploiting relation \eqref{eq.V}.

 Let $Z_i=(y_i,\bm{x}_i)$, $i=1,\ldots,n$, be random variables with  $\bm{x}_i=(1,x_{i,1},\ldots,x_{i,p-1})^\ta\in\RR^p$, $X=(\bm{x}_1,\ldots,\bm{x}_n)^\ta\in\RR^{n\times p}$ is the design matrix, $\bm{\beta}=(\beta_0,\beta_1,\ldots,\beta_{p-1})^\ta\in\RR^p$ is the parameter vector and $\bm{y}=(y_1,\ldots,y_n)^\ta$ is such that $\bm{y}=X\bm{\beta}+\bm{\varepsilon}$, $\bm{\varepsilon}=(\varepsilon_1,\ldots,\varepsilon_n)^\ta$ is a vector of errors, $\E(\bm{\varepsilon})=\bm{0}$ and $\Var(\bm{\varepsilon})=\sigma^2I$, $I$ is the $n\times n$ identity matrix. Let a training set $\CZ_{S_j}$ be of size $n_1$; then, $X_{S_j}$ is the $n_1\times p$ matrix according to $S_j$ and $\h{\bm{\beta}}_{S_j}$ indicates the estimator of ${\bm\beta}$ computed by using the data in the training set $\CZ_{S_j}$.

{\bf Assumption:} If $S_j$ is the index set of a training set with $n_1$ indices, then
\begin{equation}
\label{eq.V}
\lim_{n_1\to\infty}\frac{1}{n_1}X_{S_j}^\ta X_{S_j}=V^{-1},
\end{equation}
where $V^{-1}$ is finite and positive definite.

\begin{remark}
\label{rem.Noether}
In the aforementioned formulation $\bm{x}_i\sim F_{\bm{X}}$ with mean vector $\bm{\mu}_{\bm{X}}\in\RR^p$ and finite variance-covariance matrix $\Sig_{\bm{X}}\in\RR^{p\times p}$. Thus, if $S_j=\{i_1,\ldots,i_{n_1}\}$ is a training index set, the vectors $\bm{x}_{i_1},\ldots,\bm{x}_{i_{n_1}}$ are an independent, identically distributed collection from $F_{\bm{X}}$. Therefore, $n_1^{-1}X_{S_j}^\ta X_{S_j}=n_1^{-1}\sum_{k=1}^{n_1}\bm{x}_{i_k}\bm{x}_{i_k}^\ta$ is the usual sample estimator of $\E(\bm{X}\bm{X}^\ta)=\Sig_{\bm{X}}+\bm{\mu}_{\bm{X}}\bm{\mu}_{\bm{X}}^\ta$, which is a $p\times p$ positive definite matrix, say $V^{-1}$. The Strong Law of Large Numbers implies that $n_1^{-1}X_{S_j}^\ta X_{S_j}$ converges almost surely to $V^{-1}$, as $n_1$ tends to infinity. Therefore, \eqref{eq.V} is a natural assumption, which is generally true under the simple condition that the covariates have finite second order moments. Also, Wu's Lemma \citep[see][p.~510]{Wu1981} proves that if \eqref{eq.V} is satisfied then the following generalized Noether condition holds $\max_{1\le k \le n_1}\bm{x}_{S_j,k}^\ta\big(X_{S_j}^\ta X_{S_j}\big)^{-1}\bm{x}_{S_j,k}\to0$, as $n_1\to\infty$, where $\bm{x}_{S_j,k}^\ta$ denotes the $k$-th row of the matrix $X_{S_j}$. Notice that the fixed design case is obtained as a special case of the aforementioned framework.
\end{remark}

Under the above conditions,
\begin{equation}
\label{eq.N_p}
\sqrt{n_1}\big(\h{\bm{\beta}}_{S_j}-\bm{\beta}\big)\stackrel{\rm d}{\longrightarrow}N_p(\bm{0},\sigma^2V),\quad\textrm{as}\ \ n_1\to\infty.
\end{equation}

For each $S_j=\{i_1,\ldots,i_{n_1}\}\subset\{1,\ldots,n\}$ we define the $n_1\times n$ matrix $E_{S_j}\doteq\big(\bm{e}_{i_1}^\ta,\ldots,\bm{e}_{i_{n_1}}^\ta\big)^\ta$, where $\bm{e}_i$ is the $i$-th element of the usual basis of $\RR^n$, and the $n\times n$ diagonal matrix
\[
I_{S_j}\doteq E_{S_j}^\ta E_{S_j}=\diag(\one{1\in S_j},\ldots,\one{n\in S_j}).
\]
Observe that $\rank(I_{S_j})=\card(S_j)$ and $I_{S_j}I_{S_{j'}}=I_{S_j\cap S_{j'}}$; also, $X_{S_j}=E_{S_j} X$ and $\bm{y}_{S_j}=E_{S_j}\bm{y}$. Thus, \eqref{eq.V} is reformulated as:
\begin{equation}
\label{eq.lim.card}
[\card(S_j)]^{-1}X^\ta I_{S_j} X\to V^{-1},\quad\textrm{as} \ \card(S_j)\to\infty.
\end{equation}
Generally, the matrix $V$ is unknown; but it is estimated easily from the data by $n(X^\ta X)^{-1}$. Hereafter, we use this estimation as the true matrix $V$. In the Appendix, Table \ref{table.V} presents an illustration of the convergence of $n^{-1}(X^\ta X)$ to $V^{-1}$.

In the analysis that follows, we restrict ourselves to the case $\E(Y_i|\bm{X}_i=\bm{x}_i)$, i.e.\ the explanatory variables are treated as fixed. This formulation is known as {\it the fixed design case}.

\subsubsection{Squared error loss}
\label{sssec:sq}

The following propositions establish the expressions for $\E(\h{\mu}_j)$, $\Var(\h{\mu}_j)$ and
$\Cov(\h{\mu}_{j},\h{\mu}_{j'})$ need to obtain the variance of the CV estimator of generalization error.

\paragraph{Random CV Estimation:}
Using the above results, the following proposition offers closed form expressions for the expectation, variance and covariance of average test set errors.

\begin{proposition}
\label{prop.E,V,C-randomCV}
Under normality of errors, \eqref{eq.V} and squared error loss
\[
\begin{split}
&\E(\h{\mu}_j)=\sigma^2\left(1+\frac{p}{n_1}\right),\\
&\Var(\h{\mu}_j)=\sigma^4\left\{\frac{2}{n_2}+\frac{4p}{n_1n_2}+\frac{(3n+1)\theta}{(n-1)n_1n_2}+\frac{[2n(n_2-1)-n_1p]p}{(n-1)n_1^2n_2}\right\},\\
&\Cov(\h{\mu}_{j},\h{\mu}_{j'})=\sigma^4\left\{\frac{2}{n}+\frac{n+2n_1}{n(n-1)n_1}p+\frac{2(n+n_1(n_1-2)-1)}{(n-1)(n-2)n_1^2}\theta\right.\\
                              &\qquad\qquad\qquad\qquad\,\left.+\frac{(n-2)(n+n_1^2+2n_1n_2-1)-(n_1-1)^2}{(n-1)^2(n-2)n_1^4}(p-\theta)\right\},
\end{split}
\]
where $\theta=\sum_{i=1}^n h_{ii}^2$, $\theta\in[0,p]$, and $h_{ii}$ is the $i$-th diagonal element of the hat matrix $H=X(X^\ta X)^{-1}X^\ta$. Since $\frac{[2n(n_2-1)-n_1p]p}{(n-1)n_1^2n_2}=O(\sfrac{1}{n^2})$ and $\frac{(3n+1)\theta}{(n-1)n_1n_2}=\frac{3n\theta}{n_1n_2}+O(\sfrac{1}{n^2})$ then
\[
\Var(\h{\mu}_j)=\sigma^4\left\{\frac{2}{n_2}+\frac{4p+3\theta}{n_1n_2}\right\}+O(\sfrac{1}{n^2}).
\]
\end{proposition}
\begin{proof}
See Online Appendix.
\hfill$\square$
\end{proof}

We now relax the assumption of normality of errors by requiring a much weaker condition. That is, we require that there exists an $\epsilon>0$ such that $\E(|\varepsilon_i|^{4+\epsilon})<\infty$, effectively requiring the existence of moments of order $4+\epsilon$. To relax the assumption of normality of the errors we need to guarantee the convergence of the moments of the estimators of $\beta$ to the moments of the corresponding asymptotic distribution. \citet{AM2015} address this issue, as well as the associate rate of convergence.

\begin{lemma}
\label{lem.W}
Define the random vector $\bm{W}_{i,i'}=(W_i,W_{i'})^\ta$, where $W_i=\sqrt{n_1}\bm{x}_i^\ta(\h{\bm{\beta}}_{S_j}-\bm{\beta})$, $W_{i'}=\sqrt{n_1}\bm{x}_{i'}^\ta(\h{\bm{\beta}}_{S_j}-\bm{\beta})$, $S_j$ is the index set of a training set and $i,i'\in S_j^c$. Then
\[
\bm{W}_{i,i'}\stackrel{\rm d}{\longrightarrow}N_2(\bm{0},\Sig^{\bm W}_{i,i'}),\quad\textrm{as}\ \ n_1\to\infty,
\quad\textrm{where} \ \
\Sig^{\bm W}_{i,i'}
=\sigma^2
\left(
  \begin{array}{@{\hspace{0ex}}c@{\hspace{2ex}}c@{\hspace{0ex}}}
    \bm{x}_i^\ta V \bm{x}_i & \bm{x}_i^\ta V \bm{x}_{i'} \\
    \bm{x}_i^\ta V \bm{x}_{i'} & \bm{x}_{i'}^\ta V \bm{x}_{i'}
  \end{array}
\right).
\]
\end{lemma}

\begin{proof}
It follows easily by \eqref{eq.N_p} and {\it delta}-method.
\hfill$\square$
\end{proof}

The following proposition provides the needed expressions for computing the variance of the cross validation generalization error estimate and the optimal training sample size.

\begin{proposition}
\label{prop.E,V,C-randomCV-nonnormality}
Under the assumption that $\E(|\varepsilon_i|^{4+\epsilon})<\infty$, $\epsilon>0$, and in the case of squared error loss,
\[
\Var(\h{\mu}_j)=
\frac{\mu_4-\sigma^4}{n_2}+\frac{(4p+3\theta)\sigma^4}{n_1n_2}+O(\sfrac{1}{n^2}),
\]
where $\mu_4=\E(\varepsilon_i^4)$.
\end{proposition}

\begin{proof}
See Online Appendix.
\hfill$\square$
\end{proof}

\paragraph{$k$-fold case:}
In the $k$-fold CV case $n_1=\frac{(k-1)n}{k}$ and $n_2=\frac{n}{k}$. Then,

\begin{proposition}
\label{prop.E,V,C-kfoldCV}
Under normality of errors, \eqref{eq.V} and squared error loss
\[
\begin{split}
&\E(\h{\mu}_j)=\sigma^2\left(1+\frac{kp}{(k-1)n}\right),\\
&\Var(\h{\mu}_j)=\sigma^4\left\{\frac{2k}{n}+\frac{4k^2p}{(k-1)n^2}+\frac{3k^2\theta}{(k-1)n^2}+\frac{pk^3}{(k-1)^2n^2}\right\}+o(\sfrac{1}{n^2})=\frac{2k\sigma^4}{n}+O(\sfrac{1}{n^2}),\\
&\Cov(\h{\mu}_{j},\h{\mu}_{j'})=\sigma^4
\left\{\frac{2k^4(p-\theta)}{(k-1)^4n(n-1)}-\frac{k^2\theta}{(k-1)^2n(n-1)}\right\}+o(\sfrac{1}{n^2})=O(\sfrac{1}{n^2}).
\end{split}
\]
\end{proposition}
\begin{proof}
See Online Appendix.
\hfill$\square$
\end{proof}

\begin{proposition}
\label{prop.optim-nonnormality}
\begin{enumerate}[label=\rm(\alph*)]
\item Under normality of errors,
 \begin{enumerate}[label=\rm\roman*.]
  \item in the random CV case $n_1^{\rm opt}=\lfloor n/2 \rfloor$;
  \item in the $k$-fold case $k^{\rm opt}=n$, that is, the LOOCV is proposed.
 \end{enumerate}
\item For general error distribution with $\E(\varepsilon_i^{4+\epsilon})<\infty$ for some $\epsilon>0$, and in the case of random CV $n_1^{\rm opt}=\lfloor n/2 \rfloor$.
\end{enumerate}
\end{proposition}
\begin{proof}
(a) We omit the term $O(\sfrac{1}{n^2})$ of $\Var(\h{\mu}_j)$ in Proposition \ref{prop.E,V,C-randomCV} and consider the function $g(t)=\frac{2}{n-t}+\frac{4p+3\theta}{t(n-t)}$, $0<t<n$. As in proof of Theorem \ref{theo.opt-n_1}, we find that $g$ decreases up to $\frac{\sqrt{(4p+3\theta)/n}}{\sqrt{(4p+3\theta)/n}+\sqrt{2+(4p+3\theta)/n}}n\in(0,n/2)$ and increases after this, thus, $n_1^{\rm opt}=\lfloor n/2 \rfloor$, which does not depends on the parameters or the observations.

(b) From Proposition \ref{prop.E,V,C-kfoldCV}, since $\Cov(\h{\mu}_j,\h{\mu}_{j'})$ is a $O(\sfrac{1}{n^2})$ quantity, omitting the terms $o(\sfrac{1}{n^2})$ we get
$
\frac{\Var(\h{\mu}_{\textrm{$k$-fold}})}{\sigma^4}
=\frac{2}{n}
 +\frac{k[(p-\theta)n+(3n-4)p+3\theta(n-1)]}{(k-1)n^2(n-1)}
 +\frac{k^2p}{(k-1)^2n^2}
 +\frac{2k^3(p-\theta)}{(k-1)^3n(n-1)},
$
which implies that $k^{\rm opt}=n$.

(c) Using the same arguments as in proof of Theorem \ref{theo.opt-n_1} the optimal value of $n_1$ in \eqref{eq.opt-n1} follows.
\hfill$\square$
\end{proof}

\begin{remark}
\label{rem.opt}
Similarly to the case where the desision rule was $\overline{X}_{S_j}$, for the optimal value of $n_1=\lfloor n/2 \rfloor$, $\sfv^{\rm opt}=\frac{4\sigma^4}{n}+O(\sfrac{1}{n^2})$ and $\sfc^{\rm opt}=\frac{2\sigma^4}{n}+O(\sfrac{1}{n^2})$, hence, $\rho^{\rm opt}\simeq 1/2$. Therefore, the resampling size $J$ can be chosen either via specification of the resampling effectiveness or specification of the reduction ratio of $\h{\mu}_{{\rm CV},J}$, for a given $\pi$ or $r$, by the following relations [see \eqref{eq.Ja}, \eqref{eq.Jr}], cf.\ Table \ref{table.Jre,Jrr},
\[
J_{\mathsf{re}}(\pi)=\left\lfloor\frac{\pi}{1-\pi}\right\rfloor,
\quad\textrm{or}\quad
J_{\mathsf{rr}}(r)=\left\lfloor\sqrt{1+r^{-1}}\right\rfloor.
\]
\end{remark}

\begin{remark}
\label{rem.regr_general}
The above results can be extended to a subclass of the $q$-loss functions class; this subclass contains differentiable functions that can be expressed as functions of the errors.
\end{remark}

\subsubsection{Classification via logistic regression}
\label{sssec:0/1}

To illustrate the difficulty of obtaining a close form solution for the training set sample size in general, when other than the squared error loss functions are used, we discuss the case of classification via logistic regression.

\citet{HTF2009} formalize logistic regression as a linear classification method, where $y$ is the label of the data and $\bm{x}$ is a feature vector. The loss function here is the logistic loss given as $\log(1+e^{-yp_i})$, where $y$ is a label and $p_i$ is the algorithm prediction, such that
\[
\logit(p_i)=\bm{x}^\ta_i{\bm{\beta}}+\varepsilon_i,
\qquad
p_i=\Pr(y_i=1|\bm{x}_i).
\]

In this case the minimization problem \eqref{eq.opt-n1} does not have a closed-form solution for obtaining the optimal value of $n_1$, and the problem is reduced to numerical optimization. We offer the following algorithm in order to obtain, numerically, the optimal value of the training set sample size.

In view of \eqref{eq.opt-n1}, first we show the following general result.

\begin{theorem}
\label{theo.var(mu_j)}
Let $L$ be a loss function, $S_j$ an index set of size $n_1<n$ ($n_2=n-n_1$) and $i,i'\in S_j^c$. If the asymptotic values of $\E[L(\h{y}_{S_j,i},y_i)|S_j,i]$, and $\E[L(\h{y}_{S_j,i},y_i)L(\h{y}_{S_j,i'},y_{i'})|S_j,i,i']$ do not depend on the particular realization of $S_j$, say $\sfe_i$ and $\sfe_{i,i'}$ respectively, then, the asymptotic value of $\E(\h{\mu}_j)$ and $\Var(\h{\mu}_j)$ are
\[
\begin{split}
\E(\h{\mu}_j)&=\frac{1}{n}\sum_{i=1}^n\sfe_i,\\
\Var(\h{\mu}_j)&=\frac{1}{n^2n_2}\left\{\sum_{i=1}^n(n\sfe_{i,i}-n_2\sfe_i^2)
               +\frac{2}{n-1}\mathop{\sum\sum}_{1\le i<i'\le n}[n(n_2-1)\sfe_{i,i'}-(n-1)n_2\sfe_i\sfe_{i'}]\right\}.
\end{split}
\]
\end{theorem}
\begin{proof}
See Online Appendix.
\hfill$\square$
\end{proof}

Theorem \ref{theo.var(mu_j)} provides expressions for the expected value and variance of $\h{\mu}_j$ that enter the specification of the optimality rule.

\begin{algorithm}[htp]
\caption{Optimal size of training set in classification via logistic regression.}
\label{algorithm}
\begin{algorithmic}[1]
\small
\item Compute the matrix $V=n(X^\ta X)^{-1}$ and via logistic regression estimate the parameters $\bm\beta$ and $\sigma^2$, say $\h{\bm{\beta}}$ and $\h\sigma^2$ respectively, using the entire sample.
    \medskip

\item Compute the following probabilities and quantities
      \[
      p_i=\Pr(y_i=1|\bm{x}_i),
      \quad
      \h{\zeta}_i=\frac{\bm{x}^\ta_i\h{\bm{\beta}}}{\h{\sigma}(\bm{x}^\ta_iV\bm{x}_i)^{1/2}},
      \quad
      \h{\rho}_{i,i'}=\frac{\bm{x}_{i}^\ta V \bm{x}_{i'}}{(\bm{x}_{i}^\ta V \bm{x}_{i}\bm{x}_{i'}^\ta V \bm{x}_{i'})^{1/2}},
      \quad i,i'=1,\ldots,n.
      \]
 \Comment{{\it Each $\h{\rho}_{i,i'}$ is between $-1$ and $1$.}}
\medskip

\item For $n_1=\lfloor n/2 \rfloor,\ldots,n-1$
      \begin{enumerate}[label=\alph*.]
      \item
      Compute the quantities $\sfe_i$, $\sfe_{i,i}$ and $\sfe_{i,i'}$ of Theorem \ref{theo.var(mu_j)}
      \[
      \begin{split}
      &\sfe_i=\sfe_{i,i}=\varPhi(-\sqrt{n_1}\h{\zeta}_i)p_i+\varPhi(\sqrt{n_1}\h{\zeta}_i)(1-p_i),\\
      &\sfe_{i,i'}= \varPhi_{2,\h{\rho}_{i,i'}}(-\sqrt{n_1}\h{\zeta}_i,-\sqrt{n_1}\h{\zeta}_{i'})p_ip_{i'}
                     +\varPhi_{2,-\h{\rho}_{i,i'}}(-\sqrt{n_1}\h{\zeta}_i,\sqrt{n_1}\h{\zeta}_{i'})p_i(1-p_{i'})\\
                 &\qquad~~
                     +\varPhi_{2,-\h{\rho}_{i,i'}}(\sqrt{n_1}\h{\zeta}_i,-\sqrt{n_1}\h{\zeta}_{i'})(1-p_i)p_{i'}
                     +\varPhi_{2,\h{\rho}_{i,i'}}(\sqrt{n_1}\h{\zeta}_i,\sqrt{n_1}\h{\zeta}_{i'})(1-p_i)(1-p_{i'}).
      \end{split}
      \]
      \item Using Theorem \ref{theo.var(mu_j)}, compute $\sfv(n_1)=\Var(\h{\mu}_j)$.
      \end{enumerate}
\item Set $n_1^{\rm opt}=\arg\min_{n_1=\lfloor n/2 \rfloor,\ldots,n-1}\left\{\sfv(n_1)\right\}$.
\end{algorithmic}
\end{algorithm}

Given $S_j$ and $i,i'\in S_j^c$ standardize the components of $\bm{W}_{i,i'}$ in Lemma \ref{lem.W} as $\varPsi_i=\frac{W_i}{\sigma(\bm{x}_i^\ta V \bm{x}_i)^{1/2}}$, $\varPsi_{i'}=\frac{W_{i'}}{\sigma(\bm{x}_{i'}^\ta V \bm{x}_{i'})^{1/2}}$ and consider the random vector $\bm{\varPsi}_{i,i'}=(\varPsi_{i},\varPsi_{i'})^\ta$. Then,
\[
\bm{\varPsi}_{i,i'}\stackrel{\rm d}{\longrightarrow}N_2(\bm{0},\Sig^{\bm \varPsi}_{i,i'}),\qquad\textrm{where} \ \
\Sig^{\bm W}_{i,i'}
=
\left(
  \begin{array}{@{\hspace{0ex}}c@{\hspace{2ex}}c@{\hspace{0ex}}}
    1 & \rho_{i,i'} \\
    \rho_{i,i'} & 1
  \end{array}
\right),
\ \textrm{with} \ \rho_{i,i'}=\frac{\bm{x}_{i}^\ta V \bm{x}_{i'}}{(\bm{x}_{i}^\ta V \bm{x}_{i}\bm{x}_{i'}^\ta V \bm{x}_{i'})^{1/2}}.
\]
Also, define
\[
\zeta_i=\frac{\bm{x}^\ta_i\bm{\beta}}{\sigma(\bm{x}^\ta_iV\bm{x}_i)^{1/2}},\quad i=1,\ldots,n.
\]

The decision rule for classification is then given as follows. For a training set $\CZ_{S_j}$, from the model we estimate the probability $p_i$ for each $i\in S_j^c$, say $\h{p}_{S_j,i}$ its estimator. After, we estimate the value $y_i$ as
\[
\h{y}_{S_j,i}
=\one{\h{p}_{S_j,i}\ge1/2}
=\one{\bm{x}^\ta_i\h{\bm{\beta}}_{S_j}\ge0}
=\one{\varPsi_i\ge-\sqrt{n_1}\zeta_i}.
\]
The loss function is $L_{0/1}(\h{y}_{S_j,i},y_i)=\one{\h{y}_{S_j,i}\ne y_i}$ \citep[see][]{McAllester2007}. We compute the values $\E[L_{0/1}(\h{y}_{S_j,i},y_i)|S_j,i]=\E[L^2_{0/1}(\h{y}_{S_j,i},y_i)|S_j,i]$ and $\E[L_{0/1}(\h{y}_{S_j,i},y_i)L_{0/1}(\h{y}_{S_j,i},y_i)|S_j,i,i']$, see in Online Appendix, which are given as
\begin{equation}
\label{eq.e_i,e_ii,e_ii',0/1}
\begin{split}
\sfe_i&=\sfe_{i,i}=\varPhi(-\sqrt{n_1}\zeta_i)p_i+\varPhi(\sqrt{n_1}\zeta_i)(1-p_i),\\
\sfe_{i,i'}&= \varPhi_{2,\rho_{i,i'}}(-\sqrt{n_1}\zeta_i,-\sqrt{n_1}\zeta_{i'})p_ip_{i'}
  +\varPhi_{2,-\rho_{i,i'}}(-\sqrt{n_1}\zeta_i,\sqrt{n_1}\zeta_{i'})p_i(1-p_{i'})\\
&\hspace*{-2.9ex}
  +\varPhi_{2,-\rho_{i,i'}}(\sqrt{n_1}\zeta_i,-\sqrt{n_1}\zeta_{i'})(1-p_i)p_{i'}
  +\varPhi_{2,\rho_{i,i'}}(\sqrt{n_1}\zeta_i,\sqrt{n_1}\zeta_{i'})(1-p_i)(1-p_{i'}),
\end{split}
\end{equation}
where $\varPhi$ is the cumulative distribution function of standard normal distribution and $\varPhi_{2,\rho}$ is the cumulative distribution function of $N_2\left(\bm{0},\left({1\atop\rho}~~{\rho\atop1}\right)\right)$.

Theorem \ref{theo.var(mu_j)} gives a formula to calculate the variance of $\h{\mu}_j$. But, the minimization process of \eqref{eq.opt-n1} does not give a close form for the optimal value of $n_1$, it is reduced to a numerical minimization process. Computing the quantities $p_i$, $\zeta_i$ and $\rho_{i,i'}$, via Theorem \ref{theo.var(mu_j)} and \eqref{eq.e_i,e_ii,e_ii',0/1}, we calculate the variance of $\h{\mu}_j$ for $n_1=\lfloor n/2 \rfloor,\ldots,n-1$. The value of $n_1$ which gives the minimum value of the variance is the optimal choice of $n_1$.

The parameters $\sigma^2$, $\bm{\beta}$ are unknown and we estimate those before the minimization process begins using the entire data set. We also compute $V$ as $\h{V}=n(X^\ta X)^{-1}$.

\section{Simulation study}
\label{sec:simu}

In this section we present simulation results using a variety of distributions and loss functions with the goal of illustrating empirically our theoretical results. We discuss the empirical performance of our rules organizing the presentation according to the cases studied above, and we note that general simple recommendations about optimal selection of training sample size are possible.

\subsection{Sample mean}
\label{ssec:simu.sample.mean}
Using the version of \includegraphics[scale = .5]{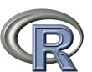}, {\it Ri386 3.1.2} on a {\it DELL Latitude E7240} PC we simulated $10^4$ samples of size $n=60, 100, 301, 750, 1501, 5000$ from distributions that can be categorized into symmetric with a variety of tail behaviors (normal, $U(-1,1)$, $t_{12}$, $t_6$) and asymmetric ($\exp(1)$, log-normal, Pareto$(15)$, Pareto$(6)$). The normal distribution is central in statistics, while the two $t$-distributions exhibit heavier than the normal, tail behavior. The log-normal distribution is used in biostatistics in biomarker studies, while the Pareto distribution is a power law distribution used in the description of social, geophysical, scientific, actuarial and many other observable phenomena. The selection of $t_6$ and Pareto(6) distributions is not arbitrary. Both the $t_6$ and Pareto(6) distributions possess less than six moments, and for our theory to apply we require the existence of up to six moments. Thus, these selections reflect the performance of the methods in limit cases. To illustrate the effect of the choice of loss function has on the size of training set we use the loss functions presented in Table \ref{table.loss}.

Our results indicate that for the $q$-class of loss functions the optimal training sample size is $\lfloor n/2 \rfloor$, independent of the data distribution. However, one can construct loss functions, such as the modified squared error loss given in Table \ref{table.loss}, for which the optimal training sample size is not $\lfloor n/2 \rfloor$. Notice that, the modified squared error loss functions does not belong in the $q$-class.

\begin{table}[htp]
 \caption{Loss functions: squared, $q$-class with $q(t)=-\sqrt{1+t^2}$, approximate absolute (where $d>0$), modified squared and double squared, and their first two derivatives with respect to $\mu$.}
 \label{table.loss}
 \centering{\footnotesize
 \begin{tabular}
 {@{\hspace{0ex}}l@{\hspace{4.2ex}}l@{\hspace{4.2ex}}c@{\hspace{4.2ex}}r@{\hspace{0ex}}}
 \addlinespace
 \toprule
 \bf name & $\bm{L_\mu(x)}$ & $\bm{L'_\mu(x)}$ & $\bm{L''_\mu(x)}$ \\
 \hline
 squared:
   & $(x-\mu)^2$
   & $-2(x-\mu)$
   & $2$ \\
 [1.5ex]
 $q$-class:
   & $\ds-\sqrt{1+\mu^2}-\frac{\mu(x-\mu)}{\sqrt{1+\mu^2}}+\sqrt{1+x^2}$
   & $\ds\frac{\mu-x}{(\mu^2+1)^{3/2}}$
   & $\ds\frac{-2\mu^2+3x\mu+1}{(\mu^2+1)^{5/2}}$ \\
 [1.5ex]
 approximated absolute:
   & $\ds\sqrt{(x-\mu)^2+d}$
   & $\ds\frac{\mu-x}{\sqrt{(x-\mu)^2+d}}$
   & $\ds\frac{d}{((\mu-x)^2+d)^{3/2}}$ \\
 [2.5ex]
 modified squared:
   & $(x-\mu)^2+\mu^2$
   & $4\mu-2x$
   & $4$ \\
 [1.5ex]
 double squared:
   & $\left(x^2-\mu^2\right)^2$
   & $-4\mu\left(x^2-\mu^2\right)$
   & $-4\left(x^2-\mu^2\right)+8\mu^2$ \\
 \bottomrule
 \end{tabular}}
 \end{table}

Tables \ref{table.Sq}--\ref{table.DouSq} (and Table \ref{table.Abs} included in the Online Appendix) present the estimated and theoretical values of $n_1^{\rm opt}/n$, $\rho^{\rm opt}$ and $k^{\rm opt}/n$ when the decision rule is the sample mean for aforementioned distributions using the loss functions presented in Table \ref{table.loss}. The estimated values of $n_1^{\rm opt}$, and hence the estimated proportion $n_1^{\rm opt}/n$, are obtained by estimating the values of the parameters $A$, $B$ of Theorem \ref{theo.opt-n_1}, using the data. Similarly, to estimate the value of $k^{\rm opt}$, and thus $k^{\rm opt}/n$, we estimate, using the data, the values of $\gamma$, $\delta$ of Proposition \ref{prop.opt-k} (see also \eqref{eq.a,b,c,d}). Because $n_1^{\rm opt}$ and $k^{\rm opt}$ are functions of the total sample size $n$, Tables \ref{table.Sq}--\ref{table.DouSq} present these estimated proportions along with their mean squared error. Therefore, if $n_1^{\rm opt}/n=0.5$ then $n_1^{\rm opt}=\lfloor n/2 \rfloor$, and if $k^{\rm opt}/n=1$ then $k^{\rm opt}=n$, a value that corresponds to LOOCV.

Table \ref{table.Sq} presents the estimated and, in parenthesis, the theoretical value of $n_1^{\rm opt}/n$, the estimated value of $\rho^{\rm opt}$, and the estimated from the data value of $k^{\rm opt}/n$. The second line, for each sample size, reports the mean squared error (MSE). The loss function here is squared error loss, which belongs in the $q$-class of loss functions. The results of Table \ref{table.Sq} indicate that $n_1^{\rm opt}=\lfloor n/2 \rfloor$, and $k^{\rm opt}=n$ corresponds to LOOCV for all distributions. Furthermore, the estimator of the correlation coefficient is highly accurate.

Similarly, Tables \ref{table.Efron} and \ref{table.Abs} (see Online Appendix) present simulation results when the loss functions are a $q$-loss with generator $q(t)=-\sqrt{t^2+1}$ and the approximated absolute error loss. The results of these Tables indicate that $n_1^{\rm opt}=\lfloor n/2 \rfloor$ and $k^{\rm opt}$ corresponds again to LOOCV. The approximated absolute error loss does not belong in the $q$-class of loss functions; we take $d=n^{-1}$.

We now contrast the above results with those presented in Tables \ref{table.ModSq} and \ref{table.DouSq}. Tables \ref{table.ModSq}, \ref{table.DouSq} present simulation results using the sample mean as the decision rule, for a variety of distributions and sample sizes indicated in the tables. Note that the minimum divisor of sample sizes 60, 100, 750 and 5000 is 2, while the minimum divisor for sample size 301 is 7 and of sample size 1501 is 19.

Tables \ref{table.ModSq}, \ref{table.DouSq} exemplify clearly the interaction between the loss function and the data distributions, with the selection of the optimal training sample size and the selection of the optimal number of folds in $k$-fold CV. Table \ref{table.ModSq} presents the proportion of the sample size allocated to the training set. While distributions such as $\exp(1)$ or log-normal select $n_1^{\rm opt}=\lfloor n/2 \rfloor$, other distributions such as $U(0,1)$ or $t_6(5)$ require $n_1^{\rm opt}$ to be approximately 80\% of the total sample size. On the other hand, in all cases the modified squared error loss function, offers $k^{\rm opt}=n$, the LOOCV.

Table \ref{table.DouSq} presents analogous results for the double squared error loss. The results indicate the impact of the loss function and data distribution on the selection of $k^{\rm opt}$. Note that, our limit cases here are represented by $t_9$ and Pareto(9) distributions (recall that we require the existence of at least eight moments). Note that all distributions, with the exception of the Pareto(15) and Pareto(9) select as $k^{\rm opt}=\md(n)$, for all sample sizes used. The Pareto distributions select $k^{\rm opt}=n$, indicating optimality for the LOOCV. The double squared error loss does not belong in the $q$-class of loss functions. The comparison with the results presented, for example in Tables \ref{table.Sq} and \ref{table.Efron}, where the loss functions were members of the $q$-class, clearly indicates the impact of the loss function on the optimal sample size selection for the training and hence the test set.

 \begin{table}[htp]
 \caption{Average of the estimated values and their empirical mean square error for various data distributions and various values of $n$, for the squared error loss and the case of sample mean.}
 \label{table.Sq}
 \centering{\footnotesize
 \begin{tabular}
 {@{\hspace{0ex}}l@{\hspace{4.5ex}}r@{\hspace{4.5ex}}
 r@{\hspace{2.5ex}}r@{\hspace{2.5ex}}r@{\hspace{0ex}}
 c@{\hspace{7ex}}
 r@{\hspace{2.5ex}}r@{\hspace{2.5ex}}r@{\hspace{2.5ex}}r
 @{\hspace{0ex}}}
 \addlinespace
 \toprule
 & \multicolumn{9}{c}{$\begin{array}{lr}\textrm{\bf Sample Size}&\textrm{\bf Estimator(Theoretical Value)}\\&\textrm{\bf MSE}\end{array}$} \\
 \hline
 $F$ & \multicolumn{1}{l}{$n$} & \multicolumn{1}{l}{$\frac{n_1^{\rm opt}}{n}$} & \multicolumn{1}{l}{$\rho^{\rm opt}$} & \multicolumn{1}{l}{$\frac{k^{\rm opt}}{n}$}
 &
 & \multicolumn{1}{l}{$n$} & \multicolumn{1}{l}{$\frac{n_1^{\rm opt}}{n}$} & \multicolumn{1}{l}{$\rho^{\rm opt}$} & \multicolumn{1}{l}{$\frac{k^{\rm opt}}{n}$}\\
 \hline
 \multirow{6}{*}{\rotatebox{90}{$N(0,1)$}}
 &\rm   60 &\rm .500(.500) &\rm .4821(.4844) &\rm 1(1) & &\rm  750 &\rm .500(.500) &\rm .4987(.4987) &\rm 1(1) \\
 [-.1ex]
 &\rm      &\rm         0  &\rm $<10^{-4}$     &\rm    0 & &\rm      &\rm         0  &\rm     $<10^{-7}$ &\rm    0 \\
 [.3ex]
 &\rm  100 &\rm .500(.500) &\rm .4896(.4904) &\rm 1(1) & &\rm 1501 &\rm .500(.500) &\rm .4990(.4990) &\rm 1(1) \\
 [-.1ex]
 &\rm      &\rm        0   &\rm $<10^{-5}$     &\rm    0 & &\rm      &\rm        0   &\rm     $<10^{-8}$ &\rm    0 \\
 [.3ex]
 &\rm  301 &\rm .502(.502) &\rm .4950(.4950) &\rm 1(1) & &\rm 5000 &\rm .500(.500) &\rm .4998(.4998) &\rm 1(1) \\
 [-.1ex]
 &\rm      &\rm          0 &\rm $<10^{-6}$     &\rm    0 & &\rm      &\rm          0 &\rm    $<10^{-10}$ &\rm    0 \\
 \cline{2-5} \cline{7-10}
 \multirow{6}{*}{\rotatebox{90}{$U(-1,1)$}}
 &\rm   60 &\rm .500(.500) &\rm .4643(.4643) &\rm 1(1) & &\rm  750 &\rm .500(.500) &\rm .4967(.4967) &\rm 1(1) \\
 [-.1ex]
 &\rm      &\rm         0  &\rm $<10^{-4}$     &\rm    0 & &\rm      &\rm         0  &\rm     $<10^{-7}$ &\rm    0 \\
 [.3ex]
 &\rm  100 &\rm .500(.500) &\rm .4773(.4773) &\rm 1(1) & &\rm 1501 &\rm .500(.500) &\rm .4980(.4980) &\rm 1(1) \\
 [-.1ex]
 &\rm      &\rm        0   &\rm $<10^{-5}$     &\rm    0 & &\rm      &\rm        0   &\rm     $<10^{-8}$ &\rm    0 \\
 [.3ex]
 &\rm  301 &\rm .502(.502) &\rm .4904(.4904) &\rm 1(1) & &\rm 5000 &\rm .500(.500) &\rm .4995(.4995) &\rm 1(1) \\
 [-.1ex]
 &\rm      &\rm          0 &\rm $<10^{-6}$     &\rm    0 & &\rm      &\rm          0 &\rm     $<10^{-9}$ &\rm    0 \\
 \cline{2-5} \cline{7-10}
 \multirow{6}{*}{\rotatebox{90}{$t_{12}$}}
 &\rm   60 &\rm .500(.500) &\rm .4850(.4884) &\rm 1(1) & &\rm  750 &\rm .500(.500) &\rm .4990(.4990) &\rm 1(1) \\
 [-.1ex]
 &\rm      &\rm         0  &\rm     $<10^{-4}$ &\rm    0 & &\rm      &\rm         0  &\rm     $<10^{-7}$ &\rm    0 \\
 [.3ex]
 &\rm  100 &\rm .500(.500) &\rm .4915(.4929) &\rm 1(1) & &\rm 1501 &\rm .500(.500) &\rm .4992(.4992) &\rm 1(1) \\
 [-.1ex]
 &\rm      &\rm        0   &\rm     $<10^{-5}$ &\rm    0 & &\rm      &\rm        0   &\rm     $<10^{-8}$ &\rm    0 \\
 [.3ex]
 &\rm  301 &\rm .502(.502) &\rm .4957(.4959) &\rm 1(1) & &\rm 5000 &\rm .500(.500) &\rm .4999(.4999) &\rm 1(1) \\
 [-.1ex]
 &\rm      &\rm          0 &\rm     $<10^{-6}$ &\rm    0 & &\rm      &\rm          0 &\rm     $<10^{-9}$ &\rm    0 \\
 \cline{2-5} \cline{7-10}
 \multirow{6}{*}{\rotatebox{90}{$t_{6}$}}
 &\rm   60 &\rm .500(.500) &\rm .4879(.4935) &\rm 1(1) & &\rm  750 &\rm .500(.500) &\rm .4993(.4994) &\rm 1(1) \\
 [-.1ex]
 &\rm      &\rm         0  &\rm     $<10^{-4}$ &\rm    0 & &\rm      &\rm         0  &\rm     $<10^{-7}$ &\rm    0 \\
 [.3ex]
 &\rm  100 &\rm .500(.500) &\rm .4934(.4961) &\rm 1(1) & &\rm 1501 &\rm .500(.500) &\rm .4994(.4994) &\rm 1(1) \\
 [-.1ex]
 &\rm      &\rm        0   &\rm     $<10^{-4}$ &\rm    0 & &\rm      &\rm        0   &\rm     $<10^{-8}$ &\rm    0 \\
 [.3ex]
 &\rm  301 &\rm .502(.502) &\rm .4965(.4970) &\rm 1(1) & &\rm 5000 &\rm .500(.500) &\rm .4999(.4999) &\rm 1(1) \\
 [-.1ex]
 &\rm      &\rm          0 &\rm     $<10^{-6}$ &\rm    0 & &\rm      &\rm          0 &\rm     $<10^{-9}$ &\rm    0 \\
 \cline{2-5} \cline{7-10}
 \multirow{6}{*}{\rotatebox{90}{$\exp(1)$}}
 &\rm   60 &\rm .500(.500) &\rm .4918(.4959) &\rm 1(1) & &\rm  750 &\rm .500(.500) &\rm .4972(.4975) &\rm 1(1) \\
 [-.1ex]
 &\rm      &\rm         0  &\rm     $<10^{-4}$ &\rm    0 & &\rm      &\rm         0  &\rm     $<10^{-6}$ &\rm    0 \\
 [.3ex]
 &\rm  100 &\rm .500(.500) &\rm .4958(.4975) &\rm 1(1) & &\rm 1501 &\rm .500(.500) &\rm .4995(.4995) &\rm 1(1) \\
 [-.1ex]
 &\rm      &\rm        0   &\rm     $<10^{-5}$ &\rm    0 & &\rm      &\rm        0   &\rm     $<10^{-9}$ &\rm    0 \\
 [.3ex]
 &\rm  301 &\rm .502(.502) &\rm .4972(.4975) &\rm 1(1) & &\rm 5000 &\rm .500(.500) &\rm .4999(.4999) &\rm 1(1) \\
 [-.1ex]
 &\rm      &\rm          0 &\rm     $<10^{-6}$ &\rm    0 & &\rm      &\rm          0 &\rm     $<10^{-9}$ &\rm    0 \\
 \cline{2-5} \cline{7-10}
 \multirow{6}{*}{\rotatebox{90}{Log Normal}}
 &\rm   60 &\rm .500(.500) &\rm .4957(.4971) &\rm 1(1) & &\rm  750 &\rm .500(.500) &\rm .4999(.4999) &\rm 1(1) \\
 [-.1ex]
 &\rm      &\rm         0  &\rm     $<10^{-4}$ &\rm    0 & &\rm      &\rm         0  &\rm     $<10^{-6}$ &\rm    0 \\
 [.3ex]
 &\rm  100 &\rm .500(.500) &\rm .4981(.4998) &\rm 1(1) & &\rm 1501 &\rm .500(.500) &\rm .4996(.4997) &\rm 1(1) \\
 [-.1ex]
 &\rm      &\rm        0   &\rm     $<10^{-6}$ &\rm    0 & &\rm      &\rm        0   &\rm     $<10^{-8}$ &\rm    0 \\
 [.3ex]
 &\rm  301 &\rm .502(.502) &\rm .4980(.4982) &\rm 1(1) & &\rm 5000 &\rm .500(.500) &\rm .4999(.5000) &\rm 1(1) \\
 [-.1ex]
 &\rm      &\rm          0 &\rm     $<10^{-6}$ &\rm    0 & &\rm      &\rm          0 &\rm    $<10^{-10}$ &\rm    0 \\
 \cline{2-5} \cline{7-10}
 \multirow{6}{*}{\rotatebox{90}{Pareto$(15)$}}
 &\rm   60 &\rm .500(.500) &\rm .4931(.4974) &\rm 1(1) & &\rm  750 &\rm .500(.500) &\rm .4997(.4998) &\rm 1(1) \\
 [-.1ex]
 &\rm      &\rm         0  &\rm     $<10^{-5}$ &\rm    0 & &\rm      &\rm         0  &\rm     $<10^{-7}$ &\rm    0 \\
 [.3ex]
 &\rm  100 &\rm .500(.500) &\rm .4966(.4984) &\rm 1(1) & &\rm 1501 &\rm .500(.500) &\rm .4995(.4996) &\rm 1(1) \\
 [-.1ex]
 &\rm      &\rm        0   &\rm     $<10^{-5}$ &\rm    0 & &\rm      &\rm        0   &\rm     $<10^{-8}$ &\rm    0 \\
 [.3ex]
 &\rm  301 &\rm .502(.502) &\rm .4975(.4978) &\rm 1(1) & &\rm 5000 &\rm .500(.500) &\rm .4999(.4999) &\rm 1(1) \\
 [-.1ex]
 &\rm      &\rm          0 &\rm     $<10^{-6}$ &\rm    0 & &\rm      &\rm          0 &\rm    $<10^{-10}$ &\rm    0 \\
 \cline{2-5} \cline{7-10}
 \multirow{6}{*}{\rotatebox{90}{Pareto$(6)$}}
 &\rm   60 &\rm .500(.500) &\rm .4944(.4991) &\rm 1(1) & &\rm  750 &\rm .500(.500) &\rm .4998(.4999) &\rm 1(1) \\
 [-.1ex]
 &\rm      &\rm         0  &\rm     $<10^{-4}$ &\rm    0 & &\rm      &\rm         0  &\rm     $<10^{-7}$ &\rm    0 \\
 [.3ex]
 &\rm  100 &\rm .500(.500) &\rm .4974(.4995) &\rm 1(1) & &\rm 1501 &\rm .500(.500) &\rm .4996(.4996) &\rm 1(1) \\
 [-.1ex]
 &\rm      &\rm        0   &\rm     $<10^{-5}$ &\rm    0 & &\rm      &\rm        0   &\rm     $<10^{-8}$ &\rm    0 \\
 [.3ex]
 &\rm  301 &\rm .502(.502) &\rm .4978(.4981) &\rm 1(1) & &\rm 5000 &\rm .500(.500) &\rm .4999(.4999) &\rm 1(1) \\
 [-.1ex]
 &\rm      &\rm          0 &\rm     $<10^{-6}$ &\rm    0 & &\rm      &\rm          0 &\rm    $<10^{-10}$ &\rm    0 \\
 \bottomrule
 \end{tabular}}
 \end{table}

 \begin{table}[htp]
 \caption{Average of the estimated values and their empirical mean square error for various data distributions and various values of $n$, for the Efron's $q$ error loss, with $q(t)=-\sqrt{t^2+1}$, for the case of sample mean.}
 \label{table.Efron}
 \centering{\footnotesize
 \begin{tabular}
 {@{\hspace{0ex}}l@{\hspace{4.5ex}}r@{\hspace{4.5ex}}
 r@{\hspace{2.5ex}}r@{\hspace{2.5ex}}r@{\hspace{0ex}}
 c@{\hspace{7ex}}
 r@{\hspace{2.5ex}}r@{\hspace{2.5ex}}r@{\hspace{2.5ex}}r
 @{\hspace{0ex}}}
 \addlinespace
 \toprule
 & \multicolumn{9}{c}{$\begin{array}{lr}\textrm{\bf Sample Size}&\textrm{\bf Estimator(Theoretical Value)}\\&\textrm{\bf MSE}\end{array}$} \\
 \hline
 $F$ & \multicolumn{1}{l}{$n$} & \multicolumn{1}{l}{$\frac{n_1^{\rm opt}}{n}$} & \multicolumn{1}{l}{$\rho^{\rm opt}$} & \multicolumn{1}{l}{$\frac{k^{\rm opt}}{n}$}
 &
 & \multicolumn{1}{l}{$n$} & \multicolumn{1}{l}{$\frac{n_1^{\rm opt}}{n}$} & \multicolumn{1}{l}{$\rho^{\rm opt}$} & \multicolumn{1}{l}{$\frac{k^{\rm opt}}{n}$}\\
 \hline
 \multirow{6}{*}{\rotatebox{90}{$N(0,1)$}}
 &\rm   60 &\rm .500(.500) &\rm .4578(.4580) &\rm 1(1) & &\rm  750 &\rm .500(.500) &\rm .4960(.4960) &\rm 1(1) \\
 [-.1ex]
 &\rm      &\rm         0  &\rm     $<10^{-4}$ &\rm    0 & &\rm      &\rm         0  &\rm     $<10^{-7}$ &\rm    0 \\
 [.3ex]
 &\rm  100 &\rm .500(.500) &\rm .4730(.4730) &\rm 1(1) & &\rm 1501 &\rm .500(.500) &\rm .4977(.4977) &\rm 1(1) \\
 [-.1ex]
 &\rm      &\rm        0   &\rm     $<10^{-4}$ &\rm    0 & &\rm      &\rm        0   &\rm     $<10^{-8}$ &\rm    0 \\
 [.3ex]
 &\rm  301 &\rm .502(.502) &\rm .4888(.4888) &\rm 1(1) & &\rm 5000 &\rm .500(.500) &\rm .4994(.4994) &\rm 1(1) \\
 [-.1ex]
 &\rm      &\rm          0 &\rm     $<10^{-6}$ &\rm    0 & &\rm      &\rm          0 &\rm     $<10^{-9}$ &\rm    0 \\
 \cline{2-5} \cline{7-10}
 \multirow{6}{*}{\rotatebox{90}{$U(-1,1)$}}
 &\rm   60 &\rm .500(.500) &\rm .4541(.4528) &\rm 1(1) & &\rm  750 &\rm .500(.500) &\rm .4954(.4954) &\rm 1(1) \\
 [-.1ex]
 &\rm      &\rm         0  &\rm     $<10^{-4}$ &\rm    0 & &\rm      &\rm         0  &\rm     $<10^{-7}$ &\rm    0 \\
 [.3ex]
 &\rm  100 &\rm .500(.500) &\rm .4700(.4693) &\rm 1(1) & &\rm 1501 &\rm .500(.500) &\rm .4974(.4974) &\rm 1(1) \\
 [-.1ex]
 &\rm      &\rm        0   &\rm     $<10^{-4}$ &\rm    0 & &\rm      &\rm        0   &\rm     $<10^{-8}$ &\rm    0 \\
 [.3ex]
 &\rm  301 &\rm .502(.502) &\rm .4874(.4874) &\rm 1(1) & &\rm 5000 &\rm .500(.500) &\rm .4993(.4993) &\rm 1(1) \\
 [-.1ex]
 &\rm      &\rm          0 &\rm     $<10^{-6}$ &\rm    0 & &\rm      &\rm          0 &\rm     $<10^{-9}$ &\rm    0 \\
 \cline{2-5} \cline{7-10}
 \multirow{6}{*}{\rotatebox{90}{$t_{12}$}}
 &\rm   60 &\rm .500(.500) &\rm .4578(.4587) &\rm 1(1) & &\rm  750 &\rm .500(.500) &\rm .4961(.4961) &\rm 1(1) \\
 [-.1ex]
 &\rm      &\rm         0  &\rm     $<10^{-4}$ &\rm    0 & &\rm      &\rm         0  &\rm     $<10^{-7}$ &\rm    0 \\
 [.3ex]
 &\rm  100 &\rm .500(.500) &\rm .4731(.4735) &\rm 1(1) & &\rm 1501 &\rm .500(.500) &\rm .4977(.4977) &\rm 1(1) \\
 [-.1ex]
 &\rm      &\rm        0   &\rm     $<10^{-4}$ &\rm    0 & &\rm      &\rm        0   &\rm     $<10^{-8}$ &\rm    0 \\
 [.3ex]
 &\rm  301 &\rm .502(.502) &\rm .4889(.4889) &\rm 1(1) & &\rm 5000 &\rm .500(.500) &\rm .4994(.4994) &\rm 1(1) \\
 [-.1ex]
 &\rm      &\rm          0 &\rm     $<10^{-6}$ &\rm    0 & &\rm      &\rm          0 &\rm     $<10^{-9}$ &\rm    0 \\
 \cline{2-5} \cline{7-10}
 \multirow{6}{*}{\rotatebox{90}{$t_{6}$}}
 &\rm   60 &\rm .500(.500) &\rm .4573(.4591) &\rm 1(1) & &\rm  750 &\rm .500(.500) &\rm .4961(.4961) &\rm 1(1) \\
 [-.1ex]
 &\rm      &\rm         0  &\rm     $<10^{-4}$ &\rm    0 & &\rm      &\rm         0  &\rm     $<10^{-7}$ &\rm    0 \\
 [.3ex]
 &\rm  100 &\rm .500(.500) &\rm .4728(.4737) &\rm 1(1) & &\rm 1501 &\rm .500(.500) &\rm .4977(.4977) &\rm 1(1) \\
 [-.1ex]
 &\rm      &\rm        0   &\rm     $<10^{-4}$ &\rm    0 & &\rm      &\rm        0   &\rm     $<10^{-8}$ &\rm    0 \\
 [.3ex]
 &\rm  301 &\rm .502(.502) &\rm .4889(.4890) &\rm 1(1) & &\rm 5000 &\rm .500(.500) &\rm .4994(.4994) &\rm 1(1) \\
 [-.1ex]
 &\rm      &\rm          0 &\rm     $<10^{-6}$ &\rm    0 & &\rm      &\rm          0 &\rm     $<10^{-9}$ &\rm    0 \\
 \cline{2-5} \cline{7-10}
 \multirow{6}{*}{\rotatebox{90}{$\exp(1)$}}
 &\rm   60 &\rm .500(.500) &\rm .4609(.4632) &\rm 1(1) & &\rm  750 &\rm .500(.500) &\rm .4964(.4964) &\rm 1(1) \\
 [-.1ex]
 &\rm      &\rm         0  &\rm     $<10^{-4}$ &\rm    0 & &\rm      &\rm         0  &\rm     $<10^{-7}$ &\rm    0 \\
 [.3ex]
 &\rm  100 &\rm .500(.500) &\rm .4750(.4761) &\rm 1(1) & &\rm 1501 &\rm .500(.500) &\rm .4977(.4977) &\rm 1(1) \\
 [-.1ex]
 &\rm      &\rm        0   &\rm     $<10^{-4}$ &\rm    0 & &\rm      &\rm        0   &\rm     $<10^{-8}$ &\rm    0 \\
 [.3ex]
 &\rm  301 &\rm .502(.502) &\rm .4896(.4897) &\rm 1(1) & &\rm 5000 &\rm .500(.500) &\rm .4995(.4995) &\rm 1(1) \\
 [-.1ex]
 &\rm      &\rm          0 &\rm     $<10^{-6}$ &\rm    0 & &\rm      &\rm          0 &\rm     $<10^{-9}$ &\rm    0 \\
 \cline{2-5} \cline{7-10}
 \multirow{6}{*}{\rotatebox{90}{Log Normal}}
 &\rm   60 &\rm .500(.500) &\rm .4536(.4549) &\rm 1(1) & &\rm  750 &\rm .500(.500) &\rm .4949(.4950) &\rm 1(1) \\
 [-.1ex]
 &\rm      &\rm         0  &\rm     $<10^{-4}$ &\rm    0 & &\rm      &\rm         0  &\rm     $<10^{-6}$ &\rm    0 \\
 [.3ex]
 &\rm  100 &\rm .500(.500) &\rm .4684(.4691) &\rm 1(1) & &\rm 1501 &\rm .500(.500) &\rm .4971(.4971) &\rm 1(1) \\
 [-.1ex]
 &\rm      &\rm        0   &\rm     $<10^{-4}$ &\rm    0 & &\rm      &\rm        0   &\rm     $<10^{-8}$ &\rm    0 \\
 [.3ex]
 &\rm  301 &\rm .502(.502) &\rm .4863(.4865) &\rm 1(1) & &\rm 5000 &\rm .500(.500) &\rm .4992(.4992) &\rm 1(1) \\
 [-.1ex]
 &\rm      &\rm          0 &\rm     $<10^{-5}$ &\rm    0 & &\rm      &\rm          0 &\rm     $<10^{-8}$ &\rm    0 \\
 \cline{2-5} \cline{7-10}
 \multirow{6}{*}{\rotatebox{90}{Pareto$(15)$}}
 &\rm   60 &\rm .500(.500) &\rm .4913(.4959) &\rm 1(1) & &\rm  750 &\rm .500(.500) &\rm .4997(.4997) &\rm 1(1) \\
 [-.1ex]
 &\rm      &\rm         0  &\rm     $<10^{-4}$ &\rm    0 & &\rm      &\rm         0  &\rm     $<10^{-7}$ &\rm    0 \\
 [.3ex]
 &\rm  100 &\rm .500(.500) &\rm .4956(.4975) &\rm 1(1) & &\rm 1501 &\rm .500(.500) &\rm .4995(.4995) &\rm 1(1) \\
 [-.1ex]
 &\rm      &\rm        0   &\rm     $<10^{-5}$ &\rm    0 & &\rm      &\rm        0   &\rm     $<10^{-8}$ &\rm    0 \\
 [.3ex]
 &\rm  301 &\rm .502(.502) &\rm .4972(.4975) &\rm 1(1) & &\rm 5000 &\rm .500(.500) &\rm .4999(.4999) &\rm 1(1) \\
 [-.1ex]
 &\rm      &\rm          0 &\rm     $<10^{-6}$ &\rm    0 & &\rm      &\rm          0 &\rm    $<10^{-10}$ &\rm    0 \\
 \cline{2-5} \cline{7-10}
 \multirow{6}{*}{\rotatebox{90}{Pareto$(6)$}}
 &\rm   60 &\rm .500(.500) &\rm .4895(.4947) &\rm 1(1) & &\rm  750 &\rm .500(.500) &\rm .4995(.4996) &\rm 1(1) \\
 [-.1ex]
 &\rm      &\rm         0  &\rm     $<10^{-4}$ &\rm    0 & &\rm      &\rm         0  &\rm     $<10^{-7}$ &\rm    0 \\
 [.3ex]
 &\rm  100 &\rm .500(.500) &\rm .4965(.4968) &\rm 1(1) & &\rm 1501 &\rm .500(.500) &\rm .4994(.4995) &\rm 1(1) \\
 [-.1ex]
 &\rm      &\rm        0   &\rm     $<10^{-5}$ &\rm    0 & &\rm      &\rm        0   &\rm     $<10^{-8}$ &\rm    0 \\
 [.3ex]
 &\rm  301 &\rm .502(.502) &\rm .4969(.4973) &\rm 1(1) & &\rm 5000 &\rm .500(.500) &\rm .4999(.4999) &\rm 1(1) \\
 [-.1ex]
 &\rm      &\rm          0 &\rm     $<10^{-6}$ &\rm    0 & &\rm      &\rm          0 &\rm    $<10^{-10}$ &\rm    0 \\
 \bottomrule
 \end{tabular}}
 \end{table}

 \begin{table}[htp]
 \caption{Average of the estimated values and their empirical mean square error for various data distributions and various values of $n$, for the modified squared error loss, in the sample mean case. The notation $t_\nu(5)$ indicates that $X-5$ follows $t_\nu$.}
 \label{table.ModSq}
 \centering{\footnotesize
 \begin{tabular}
 {@{\hspace{0ex}}l@{\hspace{4.5ex}}r@{\hspace{4.5ex}}
 r@{\hspace{2.5ex}}r@{\hspace{2.5ex}}r@{\hspace{0ex}}
 c@{\hspace{7ex}}
 r@{\hspace{2.5ex}}r@{\hspace{2.5ex}}r@{\hspace{2.5ex}}r
 @{\hspace{0ex}}}
 \addlinespace
 \toprule
 & \multicolumn{9}{c}{$\begin{array}{lr}\textrm{\bf Sample Size}&\textrm{\bf Estimator(Theoretical Value)}\\&\textrm{\bf MSE}\end{array}$} \\
 \hline
 $F$ & \multicolumn{1}{l}{$n$} & \multicolumn{1}{l}{$\frac{n_1^{\rm opt}}{n}$} & \multicolumn{1}{l}{$\rho^{\rm opt}$} & \multicolumn{1}{l}{$\frac{k^{\rm opt}}{n}$}
 &
 & \multicolumn{1}{l}{$n$} & \multicolumn{1}{l}{$\frac{n_1^{\rm opt}}{n}$} & \multicolumn{1}{l}{$\rho^{\rm opt}$} & \multicolumn{1}{l}{$\frac{k^{\rm opt}}{n}$}\\
 \hline
 \multirow{6}{*}{\rotatebox{90}{$N(1,1)$}}
 &\rm   60 &\rm .597(.583) &\rm .5168(.5084) &\rm 1(1) & &\rm  750 &\rm .587(.585) &\rm .5151(.5142) &\rm 1(1) \\
 [-.1ex]
 &\rm      &\rm      .0024 &\rm        .0004 &\rm    0 & &\rm      &\rm      .0002 &\rm   $<10^{-4}$ &\rm    0 \\
 [.3ex]
 &\rm  100 &\rm .592(.580) &\rm .5165(.5109) &\rm 1(1) & &\rm 1501 &\rm .586(.586) &\rm .5149(.5145) &\rm 1(1) \\
 [-.1ex]
 &\rm      &\rm      .0016 &\rm        .0002 &\rm    0 & &\rm      &\rm      .0001 &\rm   $<10^{-4}$ &\rm    0 \\
 [.3ex]
 &\rm  301 &\rm .588(.585) &\rm .4975(.4978) &\rm 1(1) & &\rm 5000 &\rm .586(.586) &\rm .5148(.5146) &\rm 1(1) \\
 [-.1ex]
 &\rm      &\rm      .0006 &\rm   $<10^{-4}$ &\rm    0 & &\rm      &\rm $<10^{-4}$ &\rm   $<10^{-5}$ &\rm    0 \\
 \cline{2-5} \cline{7-10}
 \multirow{6}{*}{\rotatebox{90}{$U(0,1)$}}
 &\rm   60 &\rm .787(.783) &\rm .6614(.6631) &\rm 1(1) & &\rm  750 &\rm .794(.795) &\rm .6729(.6729) &\rm 1(1) \\
 [-.1ex]
 &\rm      &\rm      .0003 &\rm        .0003 &\rm    0 & &\rm      &\rm $<10^{-4}$ &\rm   $<10^{-4}$ &\rm    0 \\
 [.3ex]
 &\rm  100 &\rm .790(.790) &\rm .6662(.6673) &\rm 1(1) & &\rm 1501 &\rm .795(.795) &\rm .6733(.6734) &\rm 1(1) \\
 [-.1ex]
 &\rm      &\rm      .0002 &\rm        .0002 &\rm    0 & &\rm      &\rm $<10^{-5}$ &\rm   $<10^{-4}$ &\rm    0 \\
 [.3ex]
 &\rm  301 &\rm .793(.794) &\rm .6713(.6716) &\rm 1(1) & &\rm 5000 &\rm .795(.795) &\rm .6736(.6737) &\rm 1(1) \\
 [-.1ex]
 &\rm      &\rm $<10^{-4}$ &\rm   $<10^{-4}$ &\rm    0 & &\rm      &\rm $<10^{-5}$ &\rm   $<10^{-5}$ &\rm    0 \\
 \cline{2-5} \cline{7-10}
 \multirow{6}{*}{\rotatebox{90}{$t_{12}(5)$}}
 &\rm   60 &\rm .857(.850) &\rm .7552(.7370) &\rm 1(1) & &\rm  750 &\rm .848(.847) &\rm .7421(.7396) &\rm 1(1) \\
 [-.1ex]
 &\rm      &\rm      .0010 &\rm        .0019 &\rm    0 & &\rm      &\rm      .0002 &\rm        .0003 &\rm    0 \\
 [.3ex]
 &\rm  100 &\rm .854(.850) &\rm .7512(.7380) &\rm 1(1) & &\rm 1501 &\rm .847(.846) &\rm .7412(.7397) &\rm 1(1) \\
 [-.1ex]
 &\rm      &\rm      .0007 &\rm        .0013 &\rm    0 & &\rm      &\rm $<10^{-4}$ &\rm        .0002 &\rm    0 \\
 [.3ex]
 &\rm  301 &\rm .850(.847) &\rm .7447(.7393) &\rm 1(1) & &\rm 5000 &\rm .846(.846) &\rm .7402(.7398) &\rm 1(1) \\
 [-.1ex]
 &\rm      &\rm      .0003 &\rm        .0006 &\rm    0 & &\rm      &\rm $<10^{-4}$ &\rm   $<10^{-4}$ &\rm    0 \\
 \cline{2-5} \cline{7-10}
 \multirow{6}{*}{\rotatebox{90}{$t_{6}(5)$}}
 &\rm   60 &\rm .824(.783) &\rm .7152(.6607) &\rm 1(1) & &\rm  750 &\rm .797(.785) &\rm .6796(.6623) &\rm 1(1) \\
 [-.1ex]
 &\rm      &\rm      .0043 &\rm        .0063 &\rm    0 & &\rm      &\rm      .0015 &\rm        .0017 &\rm    0 \\
 [.3ex]
 &\rm  100 &\rm .816(.780) &\rm .7048(.6613) &\rm 1(1) & &\rm 1501 &\rm .794(.785) &\rm .6755(.6624) &\rm 1(1) \\
 [-.1ex]
 &\rm      &\rm      .0037 &\rm        .0047 &\rm    0 & &\rm      &\rm      .0011 &\rm        .0012 &\rm    0 \\
 [.3ex]
 &\rm  301 &\rm .805(.784) &\rm .6885(.6621) &\rm 1(1) & &\rm 5000 &\rm .790(.785) &\rm .6701(.6624) &\rm 1(1) \\
 [-.1ex]
 &\rm      &\rm      .0021 &\rm        .0026 &\rm    0 & &\rm      &\rm      .0007 &\rm        .0006 &\rm    0 \\
 \cline{2-5} \cline{7-10}
 \multirow{6}{*}{\rotatebox{90}{$\exp(1)$}}
 &\rm   60 &\rm .530(.500) &\rm .5020(.4972) &\rm 1(1) & &\rm  750 &\rm .500(.500) &\rm .4998(.4998) &\rm 1(1) \\
 [-.1ex]
 &\rm      &\rm      .0028 &\rm        .0002 &\rm    0 & &\rm      &\rm $<10^{-5}$ &\rm   $<10^{-7}$ &\rm    0 \\
 [.3ex]
 &\rm  100 &\rm .516(.500) &\rm .5003(.4983) &\rm 1(1) & &\rm 1501 &\rm .500(.500) &\rm .4998(.4998) &\rm 1(1) \\
 [-.1ex]
 &\rm      &\rm      .0012 &\rm   $<10^{-4}$ &\rm    0 & &\rm      &\rm $<10^{-7}$ &\rm   $<10^{-9}$ &\rm    0 \\
 [.3ex]
 &\rm  301 &\rm .504(.502) &\rm .4991(.4988) &\rm 1(1) & &\rm 5000 &\rm .500(.500) &\rm .4999(.4999) &\rm 1(1) \\
 [-.1ex]
 &\rm      &\rm $<10^{-4}$ &\rm   $<10^{-5}$ &\rm    0 & &\rm      &\rm $<10^{-11}$&\rm  $<10^{-11}$ &\rm    0 \\
 \cline{2-5} \cline{7-10}
 \multirow{6}{*}{\rotatebox{90}{Log Normal}}
 &\rm   60 &\rm .507(.500) &\rm .4986(.4997) &\rm 1(1) & &\rm  750 &\rm .500(.500) &\rm .4999(.4999) &\rm 1(1) \\
 [-.1ex]
 &\rm      &\rm      .0006 &\rm   $<10^{-4}$ &\rm    0 & &\rm      &\rm          0 &\rm   $<10^{-8}$ &\rm    0 \\
 [.3ex]
 &\rm  100 &\rm .501(.500) &\rm .4988(.4998) &\rm 1(1) & &\rm 1501 &\rm .500(.500) &\rm .4996(.4997) &\rm 1(1) \\
 [-.1ex]
 &\rm      &\rm $<10^{-4}$ &\rm   $<10^{-5}$ &\rm    0 & &\rm      &\rm          0 &\rm   $<10^{-9}$ &\rm    0 \\
 [.3ex]
 &\rm  301 &\rm .502(.502) &\rm .4984(.4983) &\rm 1(1) & &\rm 5000 &\rm .500(.500) &\rm .4999(.5000) &\rm 1(1) \\
 [-.1ex]
 &\rm      &\rm          0 &\rm   $<10^{-7}$ &\rm    0 & &\rm      &\rm          0 &\rm  $<10^{-10}$ &\rm    0 \\
 \cline{2-5} \cline{7-10}
 \multirow{6}{*}{\rotatebox{90}{Pareto$(15)$}}
 &\rm   60 &\rm .919(.883) &\rm .8536(.7991) &\rm 1(1) & &\rm  750 &\rm .895(.887) &\rm .8128(.7997) &\rm 1(1) \\
 [-.1ex]
 &\rm      &\rm      .0026 &\rm        .0061 &\rm    0 & &\rm      &\rm      .0007 &\rm        .0016 &\rm    0 \\
 [.3ex]
 &\rm  100 &\rm .912(.890) &\rm .8428(.7994) &\rm 1(1) & &\rm 1501 &\rm .892(.887) &\rm .8078(.7997) &\rm1(1)) \\
 [-.1ex]
 &\rm      &\rm      .0017 &\rm        .0048 &\rm    0 & &\rm      &\rm      .0005 &\rm        .0010 &\rm    0 \\
 [.3ex]
 &\rm  301 &\rm .901(.887) &\rm .8236(.7996) &\rm 1(1) & &\rm 5000 &\rm .889(.887) &\rm .8032(.7997) &\rm 1(1) \\
 [-.1ex]
 &\rm      &\rm      .0011 &\rm        .0025 &\rm    0 & &\rm      &\rm      .0002 &\rm        .0005 &\rm    0 \\
 \cline{2-5} \cline{7-10}
 \multirow{6}{*}{\rotatebox{90}{Pareto$(6)$}}
 &\rm   60 &\rm .786(.617) &\rm .6780(.5260) &\rm 1(1) & &\rm  750 &\rm .697(.615) &\rm .5861(.5264) &\rm 1(1) \\
 [-.1ex]
 &\rm      &\rm      .0367 &\rm        .0308 &\rm    0 & &\rm      &\rm      .0121 &\rm        .0060 &\rm    0 \\
 [.3ex]
 &\rm  100 &\rm .764(.610) &\rm .6545(.5261) &\rm 1(1) & &\rm 1501 &\rm .680(.615) &\rm .5735(.5264) &\rm 1(1) \\
 [-.1ex]
 &\rm      &\rm      .0313 &\rm        .0228 &\rm    0 & &\rm      &\rm      .0086 &\rm        .0039 &\rm    0 \\
 [.3ex]
 &\rm  301 &\rm .722(.615) &\rm .6123(.5263) &\rm 1(1) & &\rm 5000 &\rm .659(.615) &\rm .5565(.5264) &\rm 1(1) \\
 [-.1ex]
 &\rm      &\rm      .0181 &\rm        .0112 &\rm    0 & &\rm      &\rm      .0051 &\rm        .0018 &\rm    0 \\
 \bottomrule
 \end{tabular}}
 \end{table}

 \begin{table}[htp]
 \caption{Average of the estimated values and their empirical mean square error for various data distributions and various values of $n$, for the double squared error loss, in the case of sample mean.}
 \label{table.DouSq}
 \centering{\footnotesize
 \begin{tabular}
 {@{\hspace{0ex}}l@{\hspace{3.5ex}}r@{\hspace{3.5ex}}
 r@{\hspace{2ex}}r@{\hspace{2ex}}r@{\hspace{0ex}}
 c@{\hspace{6ex}}
 r@{\hspace{2ex}}r@{\hspace{2ex}}r@{\hspace{2ex}}r
 @{\hspace{0ex}}}
 \addlinespace
 \toprule
 & \multicolumn{9}{c}{$\begin{array}{lr}\textrm{\bf Sample Size}&\textrm{\bf Estimator(Theoretical Value)}\\&\textrm{\bf MSE}\end{array}$} \\
 \hline
 $F$ & \multicolumn{1}{l}{$n$} & \multicolumn{1}{l}{$\frac{n_1^{\rm opt}}{n}$} & \multicolumn{1}{l}{$\rho^{\rm opt}$} & \multicolumn{1}{l}{$\frac{k^{\rm opt}}{n}$}
 &
 & \multicolumn{1}{l}{$n$} & \multicolumn{1}{l}{$\frac{n_1^{\rm opt}}{n}$} & \multicolumn{1}{l}{$\rho^{\rm opt}$} & \multicolumn{1}{l}{$\frac{k^{\rm opt}}{n}$}\\
 \hline
 \multirow{6}{*}{\rotatebox{90}{$N(0,1)$}}
 &\rm   60 &\rm .500(.500) &\rm  .500(.500) &\rm .0333(.0333) & &\rm  750 &\rm .500(.500) &\rm  .500(.500) &\rm .0027(.0027) \\
 [-.1ex]
 &\rm      &\rm          0 &\rm  $<10^{-7}$ &\rm            0 & &\rm      &\rm          0 &\rm $<10^{-12}$ &\rm            0 \\
 [.3ex]
 &\rm  100 &\rm .500(.500) &\rm  .500(.500) &\rm .0200(.0200) & &\rm 1501 &\rm .500(.500) &\rm  .500(.500) &\rm .0127(.0127) \\
 [-.1ex]
 &\rm      &\rm          0 &\rm  $<10^{-8}$ &\rm            0 & &\rm      &\rm          0 &\rm $<10^{-16}$ &\rm            0 \\
 [.3ex]
 &\rm  301 &\rm .502(.502) &\rm  .498(.498) &\rm .0233(.0233) & &\rm 5000 &\rm .500(.500) &\rm  .500(.500) &\rm .0004(.0004) \\
 [-.1ex]
 &\rm      &\rm          0 &\rm $<10^{-12}$ &\rm            0 & &\rm      &\rm          0 &\rm $<10^{-15}$ &\rm            0 \\
 \cline{2-5} \cline{7-10}
 \multirow{6}{*}{\rotatebox{90}{$U(-1,1)$}}
 &\rm   60 &\rm .500(.500) &\rm .500(.500) &\rm .0333(.0333) & &\rm  750 &\rm .500(.500) &\rm  .500(.500) &\rm .0027(.0027) \\
 [-.1ex]
 &\rm      &\rm          0 &\rm $<10^{-6}$ &\rm            0 & &\rm      &\rm          0 &\rm $<10^{-10}$ &\rm            0 \\
 [.3ex]
 &\rm  100 &\rm .500(.500) &\rm .500(.500) &\rm .0200(.0200) & &\rm 1501 &\rm .500(.500) &\rm  .500(.500) &\rm .0127(.0127) \\
 [-.1ex]
 &\rm      &\rm          0 &\rm $<10^{-7}$ &\rm            0 & &\rm      &\rm          0 &\rm $<10^{-11}$ &\rm            0 \\
 [.3ex]
 &\rm  301 &\rm .502(.502) &\rm .498(.498) &\rm .0233(.0233) & &\rm 5000 &\rm .500(.500) &\rm  .500(.500) &\rm .0004(.0004) \\
 [-.1ex]
 &\rm      &\rm          0 &\rm $<10^{-9}$ &\rm            0 & &\rm      &\rm          0 &\rm $<10^{-14}$ &\rm            0 \\
 \cline{2-5} \cline{7-10}
 \multirow{6}{*}{\rotatebox{90}{$t_{12}$}}
 &\rm   60 &\rm .500(.500) &\rm  .500(.500) &\rm .0333(.0333) & &\rm  750 &\rm .500(.500) &\rm  .500(.500) &\rm .0027(.0027) \\
 [-.1ex]
 &\rm      &\rm          0 &\rm  $<10^{-7}$ &\rm            0 & &\rm      &\rm          0 &\rm $<10^{-13}$ &\rm            0 \\
 [.3ex]
 &\rm  100 &\rm .500(.500) &\rm  .500(.500) &\rm .0200(.0200) & &\rm 1501 &\rm .500(.500) &\rm  .500(.500) &\rm .0127(.0127) \\
 [-.1ex]
 &\rm      &\rm          0 &\rm  $<10^{-8}$ &\rm            0 & &\rm      &\rm          0 &\rm $<10^{-15}$ &\rm            0 \\
 [.3ex]
 &\rm  301 &\rm .502(.502) &\rm  .498(.498) &\rm .0233(.0233) & &\rm 5000 &\rm .500(.500) &\rm  .500(.500) &\rm .0004(.0004) \\
 [-.1ex]
 &\rm      &\rm          0 &\rm $<10^{-12}$ &\rm            0 & &\rm      &\rm          0 &\rm $<10^{-16}$ &\rm            0 \\
 \cline{2-5} \cline{7-10}
 \multirow{6}{*}{\rotatebox{90}{$t_{9}$}}
 &\rm   60 &\rm .500(.500) &\rm  .500(.500) &\rm .0333(.0333) & &\rm  750 &\rm .500(.500) &\rm  .500(.500) &\rm .0027(.0027) \\
 [-.1ex]
 &\rm      &\rm          0 &\rm  $<10^{-7}$ &\rm            0 & &\rm      &\rm          0 &\rm $<10^{-13}$ &\rm            0 \\
 [.3ex]
 &\rm  100 &\rm .500(.500) &\rm  .500(.500) &\rm .0200(.0200) & &\rm 1501 &\rm .500(.500) &\rm  .500(.500) &\rm .0127(.0127) \\
 [-.1ex]
 &\rm      &\rm          0 &\rm  $<10^{-8}$ &\rm            0 & &\rm      &\rm          0 &\rm $<10^{-15}$ &\rm            0 \\
 [.3ex]
 &\rm  301 &\rm .502(.502) &\rm  .498(.498) &\rm .0233(.0233) & &\rm 5000 &\rm .500(.500) &\rm  .500(.500) &\rm .0004(.0004) \\
 [-.1ex]
 &\rm      &\rm          0 &\rm $<10^{-12}$ &\rm            0 & &\rm      &\rm          0 &\rm $<10^{-16}$ &\rm            0 \\
 \cline{2-5} \cline{7-10}
 \multirow{6}{*}{\rotatebox{90}{$\exp(1)$}}
 &\rm   60 &\rm .500(.500) &\rm .500(.500) &\rm .0368(.0333) & &\rm  750 &\rm .500(.500) &\rm  .500(.500) &\rm .0027(.0027) \\
 [-.1ex]
 &\rm      &\rm          0 &\rm $<10^{-6}$ &\rm        .0034 & &\rm      &\rm          0 &\rm  $<10^{-9}$ &\rm            0 \\
 [.3ex]
 &\rm  100 &\rm .500(.500) &\rm .500(.500) &\rm .0202(.0200) & &\rm 1501 &\rm .500(.500) &\rm  .500(.500) &\rm .0127(.0127) \\
 [-.1ex]
 &\rm      &\rm          0 &\rm $<10^{-7}$ &\rm        .0002 & &\rm      &\rm          0 &\rm $<10^{-10}$ &\rm            0 \\
 [.3ex]
 &\rm  301 &\rm .502(.502) &\rm .498(.498) &\rm .0233(.0233) & &\rm 5000 &\rm .500(.500) &\rm  .500(.500) &\rm .0004(.0004) \\
 [-.1ex]
 &\rm      &\rm          0 &\rm $<10^{-8}$ &\rm            0 & &\rm      &\rm          0 &\rm $<10^{-12}$ &\rm            0 \\
 \cline{2-5} \cline{7-10}
 \multirow{6}{*}{\rotatebox{90}{Log Normal}}
 &\rm   60 &\rm .500(.500) &\rm .500(.500) &\rm .0338(.0333) & &\rm  750 &\rm .500(.500) &\rm  .500(.500) &\rm .0027(.0027) \\
 [-.1ex]
 &\rm      &\rm          0 &\rm $<10^{-6}$ &\rm        .0005 & &\rm      &\rm          0 &\rm $<10^{-11}$ &\rm            0 \\
 [.3ex]
 &\rm  100 &\rm .500(.500) &\rm .500(.500) &\rm .0200(.0200) & &\rm 1501 &\rm .500(.500) &\rm  .500(.500) &\rm .0127(.0127) \\
 [-.1ex]
 &\rm      &\rm          0 &\rm $<10^{-7}$ &\rm            0 & &\rm      &\rm          0 &\rm $<10^{-12}$ &\rm            0 \\
 [.3ex]
 &\rm  301 &\rm .502(.502) &\rm .498(.498) &\rm .0233(.0233) & &\rm 5000 &\rm .500(.500) &\rm  .500(.500) &\rm .0004(.0004) \\
 [-.1ex]
 &\rm      &\rm          0 &\rm $<10^{-9}$ &\rm            0 & &\rm      &\rm          0 &\rm $<10^{-14}$ &\rm            0 \\
 \cline{2-5} \cline{7-10}
 \multirow{6}{*}{\rotatebox{90}{Pareto$(15)$}}
 &\rm   60 &\rm .500(.500) &\rm .495(.499) &\rm     1(1) & &\rm  750 &\rm .500(.500) &\rm  .500(.500) &\rm   .9999(1) \\
 [-.1ex]
 &\rm      &\rm          0 &\rm $<10^{-4}$ &\rm        0 & &\rm      &\rm          0 &\rm  $<10^{-8}$ &\rm $<10^{-4}$ \\
 [.3ex]
 &\rm  100 &\rm .500(.500) &\rm .498(.499) &\rm     1(1) & &\rm 1501 &\rm .500(.500) &\rm  .500(.500) &\rm   .9999(1) \\
 [-.1ex]
 &\rm      &\rm          0 &\rm $<10^{-5}$ &\rm        0 & &\rm      &\rm          0 &\rm  $<10^{-8}$ &\rm $<10^{-4}$ \\
 [.3ex]
 &\rm  301 &\rm .502(.502) &\rm .498(.498) &\rm .9998(1) & &\rm 5000 &\rm .500(.500) &\rm  .500(.500) &\rm   .9999(1) \\
 [-.1ex]
 &\rm      &\rm          0 &\rm $<10^{-6}$ &\rm    .0002 & &\rm      &\rm          0 &\rm $<10^{-10}$ &\rm $<10^{-4}$ \\
 \cline{2-5} \cline{7-10}
 \multirow{6}{*}{\rotatebox{90}{Pareto$(9)$}}
 &\rm   60 &\rm .500(.500) &\rm .496(.500) &\rm .9930(1) & &\rm  750 &\rm .500(.500) &\rm  .500(.500) &\rm .9809(1) \\
 [-.1ex]
 &\rm      &\rm          0 &\rm $<10^{-4}$ &\rm    .0067 & &\rm      &\rm          0 &\rm  $<10^{-8}$ &\rm    .0191 \\
 [.3ex]
 &\rm  100 &\rm .500(.500) &\rm .498(.500) &\rm .9897(1) & &\rm 1501 &\rm .500(.500) &\rm  .500(.500) &\rm .9818(1) \\
 [-.1ex]
 &\rm      &\rm          0 &\rm $<10^{-5}$ &\rm    .0100 & &\rm      &\rm          0 &\rm  $<10^{-8}$ &\rm    .0188 \\
 [.3ex]
 &\rm  301 &\rm .502(.502) &\rm .498(.498) &\rm .9998(1) & &\rm 5000 &\rm .500(.500) &\rm  .500(.500) &\rm .9815(1) \\
 [-.1ex]
 &\rm      &\rm          0 &\rm $<10^{-6}$ &\rm    .0002 & &\rm      &\rm          0 &\rm $<10^{-10}$ &\rm    .0187 \\
 \bottomrule
 \end{tabular}}
 \end{table}

 When cross validation estimators of the generalization error are used the usual recommendation made is to use 70\% - 80\% of the data for training and the remaining for testing. Figures \ref{fig.Sq} and \ref{fig.ModSq} plot the variance of the test set error as a function of the sample splitting for squared error and modified squared error losses. The graphs show that the test set variance is minimized for the optimal value of $n_1$. On the other hand, Figure \ref{fig.n/2} plots the relative efficiency of $\h{\mu}_{{\rm CV},J}$ against the ratio $n_1/n$, for $J=1$ (it corresponds to $\Var(\h{\mu}_{j})$), $10$ and $15$. The relative efficiency of $\h{\mu}_{{\rm CV},J}$ is defined as the ratio of the variance of $\h{\mu}_{{\rm CV},J}$ for any given data splitting over the variance of $\h{\mu}_{{\rm CV},J}$ at the optimal data splitting. Figure \ref{fig.n/2} clearly shows that, for all sample sizes, all loss functions with $n_1^{\rm opt}=\lfloor n/2 \rfloor$, and all distributions, the popular recommendations in terms of splitting exhibit substantial increase in the variance of $\h{\mu}_{{\rm CV},J}$. For example, when 75\% of the total sample size is used for training, leaving 25\% for testing, the increase in the variance of $\h{\mu}_{{\rm CV},10}$ when squared error loss is used and $n=24$ is 17.3\%, while when 80\% (or 90\%) of the total sample is used for training this increase is 24.4\% (or 88.6\%). And if $J=15$, then the increase in variance is 11.86\% when the training set is 75\% of the sample; it is 16.8\% (or 61.2\%) when 80\% (or 90\%) of the sample is used for training.

 Notice that $\h{\Var}(\h{\mu}_{{\rm CV},\infty})=\h{C}$ is almost unaffected by the splitting mechanism across all loss functions, independently of data distribution. This provides further justification of our selection of the optimality rule \eqref{eq.opt-n1}. This indicates that \eqref{eq.opt-n1} is equivalent to minimizing $\sfv-\sfc$ and to minimizing $\Var(\h{\mu}_{{\rm CV},J})$ for all $J$ (see also relation \eqref{eq.c-rCV}).

 Furthermore, $\rho^{\rm opt}=.5$ when $n_1^{\rm opt}=\lfloor n/2 \rfloor$ (both theoretically and empirically), so that $\sfc^{\rm opt}=.5\sfv^{\rm opt}$, and the limited variance $\sfc$ is almost unaffected by the splitting. In view of Figure \ref{fig.n/2}, $\sfv(.75)\simeq1.9\sfv^{\rm opt}\simeq3.8\sfc^{\rm opt}\simeq3.8\sfc(.75)$, while $\sfv(.80)\simeq4.6\sfc(.80)$, $\sfv(.85)\simeq6.2\sfc(.85)$ and $\sfv(.90)\simeq10\sfc(.90)$, where $\sfv(\pi)$ denotes the variance of the test set error $\h{\mu}_j$ for training sample size $n_1=\pi n$. The notations $\sfc(\pi)$ and $\rho(\pi)$ have similar interpretation. Thus, $\rho^{\rm opt}\simeq.5$, $\rho(.75)\simeq.26$, $\rho(.80)\simeq.22$,$\rho(.85)\simeq.16$, and $\rho(.90)\simeq.1$. The relationship of these results to the resampling effectiveness and reduction ratio of $\h{\mu}_{{\rm CV},J}$ are shown in Table \ref{table.RE,RR.for_n/2} for $J=10$, $15$ and when the training sample size is $.5n$ ($=n_1^{\rm opt}$), $.75n$, $.80n$, $.85n$ and $.90n$. These results show that when the training sample size in not optimal then we need to increase the resampling size to obtain acceptable levels of reduction ratio and/or resampling effectiveness, thereby increasing the computational burden of the procedure.

\begin{figure}[htp]
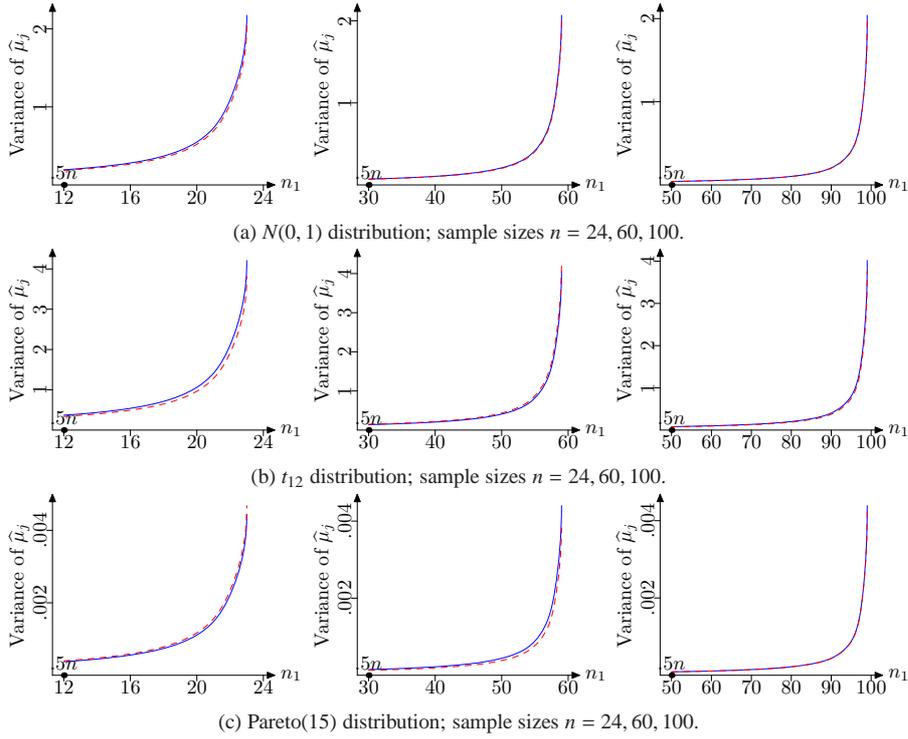

\centering
\begin{subfigure}[a]{1\textwidth}
        \includegraphics[scale = .75]{Fig.311}
\hfill
        \includegraphics[scale = .75]{Fig.312}
\hfill
        \includegraphics[scale = .75]{Fig.313}
        \caption{$N(0,1)$ distribution; sample sizes $n=24,60,100$.}
\end{subfigure}

\begin{subfigure}[a]{1\textwidth}
        \includegraphics[scale = .75]{Fig.321}
\hfill
        \includegraphics[scale = .75]{Fig.322}
\hfill
        \includegraphics[scale = .75]{Fig.323}
        \caption{$t_{12}$ distribution; sample sizes $n=24,60,100$.}
\end{subfigure}

\begin{subfigure}[a]{1\textwidth}
        \includegraphics[scale = .75]{Fig.331}
\hfill
        \includegraphics[scale = .75]{Fig.332}
\hfill
        \includegraphics[scale = .75]{Fig.333}
        \caption{Pareto$(15)$ distribution; sample sizes $n=24,60,100$.}
\end{subfigure}

\caption{Theoretical (blue solid line) and empirical (red dashed line, $N=500$ Monte Carlo repetitions), values of the variance of $\h{\mu}_j$ for various data distributions, total sample $n$ values and $n_1=\lfloor n/2 \rfloor,\ldots,n-1$. The loss is squared error loss and the decision rule is $\overline{X}$.}
\label{fig.Sq}
\end{figure}

\begin{figure}[htp]
\begin{subfigure}[a]{1\textwidth}
        \includegraphics[scale = .75]{Fig.411}
\hfill
        \includegraphics[scale = .75]{Fig.412}
\hfill
        \includegraphics[scale = .75]{Fig.413}
        \caption{$N(1,1)$ distribution; sample sizes $n=24,60,100$.}
\end{subfigure}

\begin{subfigure}[a]{1\textwidth}
        \includegraphics[scale = .75]{Fig.421}
\hfill
        \includegraphics[scale = .75]{Fig.422}
\hfill
        \includegraphics[scale = .75]{Fig.423}
        \caption{$t_{12}(5)$ distribution; sample sizes $n=24,60,100$.}
\end{subfigure}

\begin{subfigure}[a]{1\textwidth}
        \includegraphics[scale = .75]{Fig.431}
\hfill
        \includegraphics[scale = .75]{Fig.432}
\hfill
        \includegraphics[scale = .75]{Fig.433}
        \caption{Pareto$(15)$ distribution; sample sizes $n=24,60,100$.}
\end{subfigure}

\caption{Theoretical (blue solid line) and empirical (red dashed line, $N=500$ Monte Carlo repetitions), values of the variance of $\h{\mu}_j$ for various data distributions, total sample $n$ values and $n_1=\lfloor n/2 \rfloor,\ldots,n-1$. The loss is modified squared error loss and the decision rule is $\overline{X}$.}
\label{fig.ModSq}
\end{figure}

\begin{figure}[htp]
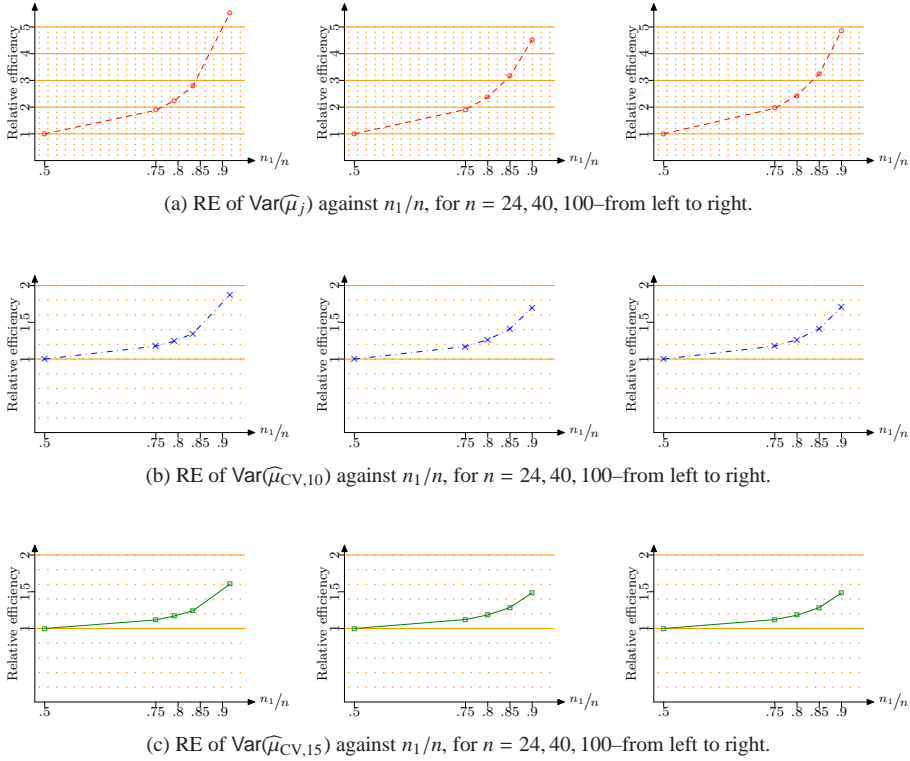

\centering
\begin{subfigure}[a]{1\textwidth}
        \includegraphics[scale = .65]{Fig.511}
\hfill
        \includegraphics[scale = .65]{Fig.512}
\hfill
        \includegraphics[scale = .65]{Fig.513}
\caption{RE of $\Var(\h{\mu}_j)$ against $n_1/n$, for $n=24,40,100$--from left to right.}
\label{sfig.V}
\end{subfigure}
\vspace{.7cm}

\begin{subfigure}[a]{1\textwidth}
        \includegraphics[scale = .65]{Fig.521}
\hfill
        \includegraphics[scale = .65]{Fig.522}
\hfill
        \includegraphics[scale = .65]{Fig.523}
\caption{RE of $\Var(\h{\mu}_{{\rm CV},10})$ against $n_1/n$, for $n=24,40,100$--from left to right.}
\label{sfig.CV10}
\end{subfigure}

\vspace{.7cm}
\begin{subfigure}[a]{1\textwidth}
        \includegraphics[scale = .65]{Fig.531}
\hfill
        \includegraphics[scale = .65]{Fig.532}
\hfill
        \includegraphics[scale = .65]{Fig.533}
\caption{RE of $\Var(\h{\mu}_{{\rm CV},15})$ against $n_1/n$, for $n=24,40,100$--from left to right.}
\label{sfig.CV15}
\end{subfigure}
\caption{Relative efficiency (RE) of $\Var(\h{\mu}_j)$ [sub-figure \ref{sfig.V}], $\Var(\h{\mu}_{{\rm CV},10})$ [sub-figure \ref{sfig.CV10}] and $\Var(\h{\mu}_{{\rm CV},15})$ [sub-figure \ref{sfig.CV15}] (case of sample mean), $N=10000$ Monte Carlo repetitions, for all loss functions with $n_1^{\rm opt}=\lfloor n/2 \rfloor$.}
\label{fig.n/2}
\end{figure}

\begin{table}[htp]
 \caption{The resampling effectiveness and reduction ratio of $\h{\mu}_{{\rm CV}, J}$, $J=10,15$, for the loss function which have $n_1^{\rm opt}=\lfloor n/2 \rfloor$, in the sample mean case.}
 \label{table.RE,RR.for_n/2}
 \centering{
 \begin{tabular}{@{\hspace{0ex}}c@{\hspace{12ex}}c@{\hspace{12ex}}c@{\hspace{17ex}}c@{\hspace{12ex}}c@{\hspace{12ex}}c@{\hspace{0ex}}}
 \addlinespace
 \toprule
  & \multicolumn{2}{c}{\hspace{-17ex}Resampling Effectiveness} & & \multicolumn{2}{c}{Reduction Ratio}\\
 \cline{2-6}
 $n_1$ & $J=10$ & $J=15$ & & $J=10$ & $J=15$ \\
 \hline
 $.50n$ & 90\% & 94\% & & .010 & .0045\\
 $.75n$ & 78\% & 84\% & & .025 & .0114\\
 $.80n$ & 74\% & 81\% & & .029 & .0137\\
 $.85n$ & 66\% & 74\% & & .038 & .0185\\
 $.90n$ & 53\% & 63\% & & .053 & .0268\\
 \bottomrule
 \end{tabular}}
 \end{table}

\subsection{Regression and classification via logistic regression}
\label{ssec:simu.regression}
We empirically now study the variance of generalization error in the cases of linear regression and classification via logistic regression.

Date were generated as $y_i=1+X_1+X_2-X_3+X_4+\varepsilon_i$, $i=1,2,\ldots,n$, where the sample size $n=40$, $60$, $100$ and $200$ and the error vector $\bm{\varepsilon}_n$ follows $N_n(\bm{0},I)$ distribution. The four covariates were generated only once as follows. The first covariate $X_1$ is a binary variable generated from ${\rm Bernoulli}(.6)$ distribution, the second variable is generated from ${\rm Poisson}(2)$ distribution, the third variable is generated from the $U(0,5)$ distribution and variable $X_4$ is generated from the $U(0,3)$ distribution. The number of Monte Carlo repetitions equals 5000.

Table \ref{table.SqRegr} and Figure \ref{fig.SqReg} present the results; notice that, both the table and figure, indicate an increase in the variance of the test set error, $\Var(\h{\mu}_j)$, as well as in $\Var(\h{\mu}_{{\rm CV},J})$ with $J=10,15$, when the training sample size moves away from the optimal value of $n_1=\lfloor n/2 \rfloor$. This is true for all sample sizes.

Table \ref{table.ZeroOne} presents relevant results for classification via logistic regression. The model contains three covariates, $X_1$ is generated once as ${\rm Bernoulli}(.6)$, $X_2$ is ${\rm Poisson}(2)$ and $X_3$ is generated from discrete uniform on the set $\{0,1,2,3,4\}$. The vector of errors is again $N_n(\bm{0},I)$, and we generate 50 Monte Carlo samples. Table \ref{table.ZeroOne} presents estimates of the $\Var(\h{\mu}_j)$ for two values of the sample size $n$, 60 and 100, and for different values of the training sample size $n_1$. Again, $\Var(\h{\mu}_j)$ increases if the popular choices of $n_1$, $.75n$ or $.80n$, are used, and $\Var(\h{\mu}_j)$ is minimized when $n_1=\lfloor n/2 \rfloor$, and this holds for all sample sizes and error distributions.

\begin{table}[htp]
 \caption{Monte Carlo (5000 repetitions) estimators of ${n}_1^{\rm opt}$, and ${\Var}(\h{\mu}_{{\rm CV},J})$, $J=1,10,15,\infty$, for various values the total sample size $n$ and training size $n_1$. The loss function is squared error and the regression model has four covariates, with error vector following $N_n(\bm{0},I)$.}
 \label{table.SqRegr}
 \centering{\scriptsize
 \begin{tabular}
 {@{\hspace{0ex}}l@{\hspace{6ex}}l@{\hspace{4ex}}r@{\hspace{6ex}}c@{\hspace{8.5ex}}c@{\hspace{8.5ex}}c@{\hspace{8.5ex}}c@{\hspace{0ex}}}
 \addlinespace
 \toprule
 ~~$n$ & $\h{n}_1^{\rm opt}~\left({\rm MSE}(\h{n}_1^{\rm opt})\right)$ & $n_1~~~$ & $\h{\Var}(\h{\mu}_j)$ & $\h{\Var}(\h{\mu}_{{\rm CV},10})$ & $\h{\Var}(\h{\mu}_{{\rm CV},15})$ & $\h{\Var}(\h{\mu}_{{\rm CV},\infty})$ \\
 \hline
 ~40 & 20(0) & $\h{n}_1^{\rm opt}=20$ & 1.132 & 0.622 & 0.603 & 0.565 \\
 [-.1ex]
 &    & $.75n=40$                     & 1.934 & 0.685 & 0.638 & 0.546 \\
 [-.1ex]
 &    & $.80n=32$                     & 2.366 & 0.726 & 0.665 & 0.544 \\
 [-.1ex]
 &    & $.85n=34$                     & 3.094 & 0.797 & 0.712 & 0.541 \\
 [-.1ex]
 &    & $.90n=36$                     & 4.560 & 0.942 & 0.808 & 0.540 \\
 [1ex]
 ~60 & 30(0) & $\h{n}_1^{\rm opt}=30$ & 0.710 & 0.390 & 0.378 & 0.355 \\
 [-.1ex]
 &    & $.75n=45$                     & 1.276 & 0.439 & 0.408 & 0.346 \\
 [-.1ex]
 &    & $.80n=48$                     & 1.573 & 0.468 & 0.427 & 0.345 \\
 [-.1ex]
 &    & $.85n=51$                     & 2.071 & 0.517 & 0.459 & 0.344 \\
 [-.1ex]
 &    & $.90n=54$                     & 3.071 & 0.616 & 0.525 & 0.343 \\
 [1ex]
 100 & 50(0) & $\h{n}_1^{\rm opt}=50$ & 0.400 & 0.220 & 0.213 & 0.200 \\
 [-.1ex]
 &    & $.75n=75$                     & 0.749 & 0.252 & 0.234 & 0.197 \\
 [-.1ex]
 &    & $.80n=80$                     & 0.929 & 0.270 & 0.245 & 0.196 \\
 [-.1ex]
 &    & $.85n=85$                     & 1.229 & 0.299 & 0.265 & 0.196 \\
 [-.1ex]
 &    & $.90n=90$                     & 1.831 & 0.359 & 0.305 & 0.196 \\
 [1ex]
 200 &100(0) & $\h{n}_1^{\rm opt}=100$& 0.188 & 0.103 & 0.100 & 0.094 \\
 [-.1ex]
 &    & $.75n=150$                    & 0.363 & 0.120 & 0.111 & 0.093 \\
 [-.1ex]
 &    & $.80n=160$                    & 0.452 & 0.129 & 0.117 & 0.093 \\
 [-.1ex]
 &    & $.85n=170$                    & 0.600 & 0.143 & 0.127 & 0.093 \\
 [-.1ex]
 &    & $.90n=180$                    & 0.898 & 0.173 & 0.147 & 0.093 \\
 \bottomrule
 \end{tabular}}
 \end{table}

\begin{figure}[htp]
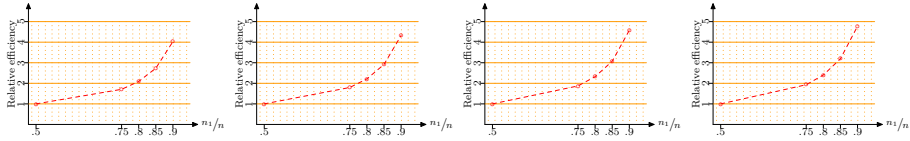
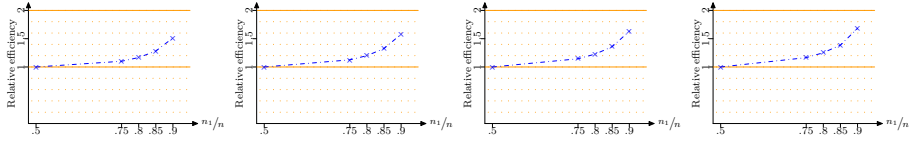
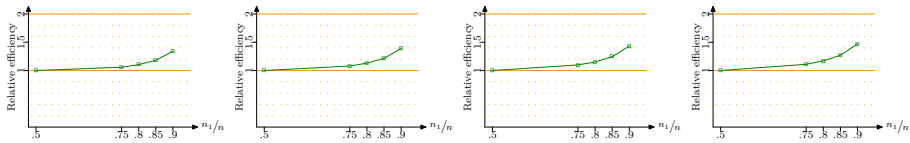

\centering
\begin{subfigure}[a]{1\textwidth}
        \includegraphics[scale = .5]{Fig.611}
\hfill
        \includegraphics[scale = .5]{Fig.612}
\hfill
        \includegraphics[scale = .5]{Fig.613}
\hfill
        \includegraphics[scale = .5]{Fig.614}
\caption{RE of $\Var(\h{\mu}_j)$ against $n_1/n$, for $n=40,60,100,200$--from left to right.}
\label{sfig.V.Reg}
\end{subfigure}
\vspace{.7cm}

\begin{subfigure}[a]{1\textwidth}
        \includegraphics[scale = .5]{Fig.621}
\hfill
        \includegraphics[scale = .5]{Fig.622}
\hfill
        \includegraphics[scale = .5]{Fig.623}
\hfill
        \includegraphics[scale = .5]{Fig.624}
\caption{RE of $\Var(\h{\mu}_{{\rm CV},10})$ against $n_1/n$, for $n=40,60,100,200$--from left to right.}
\label{sfig.CV10.Reg}
\end{subfigure}

\vspace{.7cm}
\begin{subfigure}[a]{1\textwidth}
        \includegraphics[scale = .5]{Fig.631}
\hfill
        \includegraphics[scale = .5]{Fig.632}
\hfill
        \includegraphics[scale = .5]{Fig.633}
\hfill
        \includegraphics[scale = .5]{Fig.634}
\caption{RE of $\Var(\h{\mu}_{{\rm CV},15})$ against $n_1/n$, for $n=40,60,100,200$--from left to right.}
\label{sfig.CV15.Reg}
\end{subfigure}
\caption{The relative efficiency (RE), in the regression case, of $\Var(\h{\mu}_j)$ [sub-figure \ref{sfig.V.Reg}], $\Var(\h{\mu}_{{\rm CV},10})$ [sub-figure \ref{sfig.CV10.Reg}] and $\Var(\h{\mu}_{{\rm CV},15})$ [sub-figure \ref{sfig.CV15.Reg}]. The number of Monte Carlo repetitions is 5000 and squared error loss is used.}
\label{fig.SqReg}
\end{figure}

 \begin{table}[htp]
 \caption{Monte Carlo  (50 repetitions) estimators of $n_1^{\rm opt}$ and $\Var(\h{\mu}_j)$ for specific values of the training set sample size $n_1$ for various values of the total sample size and error distributions.}
 \label{table.ZeroOne}
 \centering{\footnotesize
 \begin{tabular}
 {@{\hspace{0ex}}l@{\hspace{2ex}}r@{\hspace{1ex}}c@{\hspace{1ex}}c@{\hspace{3.8ex}}c@{\hspace{3.8ex}}c@{\hspace{3.8ex}}c@{\hspace{3.8ex}}c@{\hspace{0ex}}}
 \addlinespace
 \toprule
  $F_\varepsilon$ & $n$ & $\h{n}_1^{\rm opt}\left(\h{\Var}(\h{n}_1^{\rm opt})\right)$ & \multicolumn{5}{c}{$\Var(\h{\mu}_j)$} \\
 \cline{4-8}
 & & & $\h{n}_1^{\rm opt}$ & $.75n$ & $.80n$ & $.85n$ & $.90n$ \\
 \hline
 \multirow{2}{*}{\rotatebox{45}{\scriptsize$N(0,1)$}}
 &  60 & 30(0) & $3.45\times10^{-3}$ & $9.24\times10^{-3}$ & $12.2\times10^{-3}$ & $17.0\times10^{-3}$ & $26.8\times10^{-3}$ \\
 & 100 & 50(0) & $2.10\times10^{-3}$ & $5.70\times10^{-3}$ & $7.50\times10^{-3}$ & $10.6\times10^{-3}$ & $16.6\times10^{-3}$ \\
 [2ex]
 \multirow{2}{*}{\rotatebox{45}{\scriptsize$U(-1,1)$}}
 &  60 & 30(0) & $6.50\times10^{-5}$ & $15.0\times10^{-5}$ & $20.0\times10^{-5}$ & $27.0\times10^{-5}$ & $42.0\times10^{-5}$ \\
 & 100 & 50(0) & $5.89\times10^{-5}$ & $12.9\times10^{-5}$ & $16.4\times10^{-5}$ & $22.3\times10^{-5}$ & $34.1\times10^{-5}$ \\
 [2ex]
 \multirow{2}{*}{\rotatebox{45}{\scriptsize$t_{12}$}}
 &  60 & 30(0) & $9.60\times10^{-5}$ & $21.3\times10^{-5}$ & $27.2\times10^{-5}$ & $36.9\times10^{-5}$ & $56.5\times10^{-5}$ \\
 & 100 & 50(0) & $6.68\times10^{-5}$ & $14.8\times10^{-5}$ & $18.9\times10^{-5}$ & $25.7\times10^{-5}$ & $39.3\times10^{-5}$ \\
 \bottomrule
 \end{tabular}}
 \end{table}

\section{An application to real data}
\label{sec:real}

In what follows, we briefly describe the two data sets we use to illustrate our results. One is a data set on birth weight and the second is a part of the data set from the {\it National Health and Nutrition Examination Survey} (NHANES III) conducted by the {\it National Center for Health Statistics} (NCHS) between 1988 and 1994 \citep[see][]{HL2000}.

Low birth weight, defined as a birth weight less than 2500 grams, is an outcome that has been of concern to physicians for years, because of its association between high infant mortality and birth defect rates. Data were collected as part of a large study at Baystate Medical Center in Springfield, MA. The outcome variable was a binary variable taking values 0 if the weight of a baby at birth is greater than or equal to 2500 gr and 1 if it is less than 2500 gr, i.e.\ it is actual low birth weight. Measurements on 10, additional to the outcome,  variables were collected on 189 subjects. Out of 189 subjects 59 were observed as low birth weight and 130 as normal birth weight, so the class proportions were 31.22\% low birth weight and 68.78\% normal birth weight. Because the race variable has three categories, white, black and other, it was coded using two design variables defined as follows. The first design variable takes the value 1 when the race is ``black'' and 0 otherwise, while the second design variable takes the value 1 when the race is ``other'' and 0 when it is ``white'' or ``black''. The other variable we use was mother's weight at the last menstrual period, and the logistic regression model we use contains a constant term \citep[see][p.~38]{HL2000}.

The second data set has a binary outcome variable and five covariates. The binary outcome variable $Y$ takes values 0 if the average systolic blood pressure is less than or equal 140 and is 1 if it is greater than 140. Data were collected via physical examinations and clinical and laboratory tests, and only adults defined as 20 years of age and older were included in the survey. We selected a sample of size 1,000 with complete observations on the covariates representing age (in years), sex (male or female), race (white, black, other), body weight (in pounds) and standing height (in inches). Because the race variable has three categories, it was coded using the same design variables, as in the previous example. Applying our proposed Algorithm \ref{algorithm} (for classification via logistic regression), we obtain that as the size of the training sample increases so does the variance of the test set error. We obtain minimum variance when $n_1=\lfloor n/2 \rfloor$.



\section{Discussion and Recommendations}
\label{sec:conclusions}

In this paper we address the problem of ``optimal'' selection of the size of training sets when interest centers in making inferences about the generalization error of prediction algorithms. We study two types of cross validation estimators of the generalization error. These are random cross validation (or repeated learning-testing method) and $k$-fold cross validation estimators. The statistical rule that defines ``optimality'', in light in Proposition \ref{prop.v,c} and Theorem \ref{theo.var_bound}, calls for the minimization of the variance of the test set error.

Our results indicate that the optimal training sample size selection is a complex problem and it depends primarily on the loss function that is used, as well as on the data distribution. Describe the complexity of the analysis, simple general rules for practical use can be drawn. When the loss function belongs to the $q$-class then, for the case of random cross validation, $n_1^{\rm opt}=\lfloor n/2 \rfloor$ (for both, sample mean and regression decision rules). When the loss function does not belong in the $q$-class then the value of $n_1^{\rm opt}$ may or may not equal $\lfloor n/2 \rfloor$, indicating that $q$-class membership is a sufficient condition for $n_1^{\rm opt}=\lfloor n/2 \rfloor$, but not a necessary condition. We present cases where the loss function does not belong in the $q$-class, yet $n_1^{\rm opt}=\lfloor n/2 \rfloor$ (see, for example, Table \ref{table.Abs}). Furthermore, we illustrate the effects of complex interactions between the loss function a researcher chooses and the data distribution on the optimal training sample size selection in Tables \ref{table.ModSq}.

Furthermore, we studied the effect that popular training set sample size selection, such as $.75n$ or $.80n$, has on the $\Var(\h{\mu}_j$ and $\Var(\h{\mu}_{{\rm CV},J})$, $J=10,15$. We found out that as the training sample size increases away from its optimal value the aforementioned variances increase substantially. To decrease those we need to increase the resampling size $J$ from 15 to a value that achieves acceptable reduction ratio and/or resampling effectiveness, thereby increasing the computational cost.

The selection of the resampling size $J$ of a random cross validation estimator of the generalization error is important, as it contributes to the variance reduction of this estimator. We propose two methods of selecting $J$ and exemplify their use. Our analysis indicates that, when the correlation between two different test sets is moderate, i.e.\ in range of $[.2,.3]$, the resampling size $J\ge14$ if we desire resampling effectiveness greater than or equal to $.85$. The higher the desired resampling effectiveness, the higher the value of $J$. Similar results hold when the reduction ratio forces the variance of $\h{\mu}_{{\rm CV},J}$ closer to its asymptotic value, and provides for larger values of $J$. Our work shows that the correlation coefficient $\rho$ is affected by the training sample size and since $J$ is a function of $\rho$ it is also affected by the splitting of the total sample. The effect of training sample size $n_1$ is to decrease the correlation as $n_1$ moves away from its optimal value, requiring larger and larger values of $J$ to achieve small values of $\Var(\h{\mu}_{{\rm CV},J})$.

A potential limitation to the complete generality of the results presented here is the fact that we require the operating data distributions to have finite moments of at least order six (see Subsections \ref{ssec:mean} and \ref{ssec:regression}). Most commonly used distributions satisfy this condition. But if the data do not follow a distribution satisfying the aforementioned condition, a transformation of the data may be amenable to the analysis presented here.

For the $k$-fold CV estimator of the generalization error we provide rules for obtaining $k^{\rm opt}$, and study the factors that affect the value of it. As in the case of random CV, the value of $k^{\rm opt}$ is affected by the loss function and the data distribution. Our results indicate that LOOCV can be replaced by $k$-fold with $5\le k \le 10$, in most cases.


{\appendix
\renewcommand\tablename{\footnotesize Table}
\renewcommand\figurename{\footnotesize Fig.}
\section{An additional analysis for the regression case}
\label{sec.app.A}
\subsection{An illustration of the convergence in \eqref{eq.V}}
\label{sec.sapp.A1}
 Table \ref{table.V} illustrates the convergence of the matrix sequence $n^{-1}X^{\ta}X$ to the matrix  $V^{-1}=\Sig_{\bm{X}}+\bm{\mu}_{\bm{X}}\bm{\mu}_{\bm{X}}^\ta$. We simulate $10^4$ random samples of size $n=30$, $100$ and $250$ from the 3-dimensional normal distribution with mean vector $\mu=\bm{0}\in\RR^3$ and variance-covariance matrix
 $\Sig={\tiny\left(\begin{array}{@{\hspace{0ex}}r@{\hspace{.5ex}}r@{\hspace{.5ex}}r@{\hspace{0ex}}}
       1 & .5 & .7\\
      .5 & 1 & -.1\\
      .7 & -.1 & 1
       \end{array}\right)}$
 and from the trinomial distribution with number of the possible outcomes $k=10$ and probability vector $\bm{p}=(.1,.2,.7)^\ta$. In the table are shown the average of the estimator $n^{-1}X^{\ta}X=(\h{v}^*_{ij})_{3\times3}$ from the samples, and the corresponding average of the maximum norm, $\|\cdot\|_{\max}$, between the estimators and the true value of $V^{-1}=(v^*_{ij})_{3\times3}$, that is $\|n^{-1}X^{\ta}X-V^{-1}\|_{\max}\doteq\max_{i,j}\{|\h{v}^*_{ij}-v^*_{ij}|\}$. Also, for the trinomial distribution case, in which its components do not standardized, is shown in a parenthesis the average of the maximum norm between the estimators and the true value of $V^{-1}$ of the corresponding standardized random vectors.
 \begin{table}[htp]
 \caption{\footnotesize The average of the estimators $n^{-1}X^TX$, the true $V^{-1}$ and the average of the maximum norm between the estimators and the true value of $V^{-1}$, when $\bm{X}\sim N_3(\bm{0},V^{-1})$ or $\bm{X}\sim {\rm trinomial}(10;(.1,.2,.7))$ for sample size of $n=30,100,250$. For the trinomial case, in parenthesis, the average of the maximum norm of the corresponding standardized vectors, is presented.}
 \label{table.V}
 \centering{\footnotesize
 \begin{tabular}
 {@{\hspace{0ex}}l@{\hspace{1.5ex}}c@{\hspace{2ex}}
 c@{\hspace{2ex}}
 c@{\hspace{2ex}}c@{\hspace{0ex}}}
 \addlinespace
 \toprule
 $\bm{F}$ & $\bm{V^{-1}}$ & \multicolumn{3}{c}{$\begin{array}{c}\bm{n^{-1}X^TX}\\ \bm{\|V^{-1}-n^{-1}X^TX\|_{\max}}\end{array}$} \\
 \hline
  & & \multicolumn{3}{c}{$n$} \\
 \cline{3-5}
  & & 30 & 100 & 250 \\
 \cline{2-5}\\
 [-2ex]
 \raisebox{-1.6ex}{\rotatebox{90}{${\textrm{\footnotesize 3-dim.}\atop \textrm{\footnotesize normal}}$}}
  &\raisebox{1.8ex}{$\left(
    \begin{array}{@{\hspace{0ex}}r@{\hspace{1.4ex}}r@{\hspace{1.4ex}}r@{\hspace{0ex}}}
      1 & .5 & .7\\
      .5 & 1 & -.1\\
      .7 & -.1 & 1
    \end{array}
   \right)$}
   &
   ${\left(
    \begin{array}{@{\hspace{0ex}}r@{\hspace{1.4ex}}r@{\hspace{1.4ex}}r@{\hspace{0ex}}}
      0.996 &  0.498 &  0.697\\
      0.498 &  0.998 & -0.101\\
      0.697 & -0.101 &  0.997
    \end{array}
   \right)\atop \ds 0.344}$
   &
   ${\left(
    \begin{array}{@{\hspace{0ex}}r@{\hspace{1.4ex}}r@{\hspace{1.4ex}}r@{\hspace{0ex}}}
      0.998 &  0.498 & 0.700\\
      0.498 &  0.997 &-0.099\\
      0.700 & -0.099 & 0.999
    \end{array}
   \right)\atop \ds 0.191}$
   &
   ${\left(
    \begin{array}{@{\hspace{0ex}}r@{\hspace{1.4ex}}r@{\hspace{1.4ex}}r@{\hspace{0ex}}}
      1.001 &  0.501 & 0.701\\
      0.501 &  1.000 &-0.099\\
      0.701 & -0.099 & 1.001
    \end{array}
   \right)\atop \ds 0.111}$ \\
  [2.5ex]
  \cline{2-5}\\
  [-2ex]
  \raisebox{-2.2ex}{\rotatebox{90}{\footnotesize trinomial}}
  &\raisebox{1.8ex}{$\left(
    \begin{array}{@{\hspace{0ex}}r@{\hspace{1.4ex}}r@{\hspace{1.4ex}}r@{\hspace{0ex}}}
      1.9 &  1.8 &  6.3\\
      1.8 &  5.6 & 12.6\\
      6.3 & 12.6 & 51.1
    \end{array}
   \right)$}
   &
   ${\left(
    \begin{array}{@{\hspace{0ex}}r@{\hspace{1.4ex}}r@{\hspace{1.4ex}}r@{\hspace{0ex}}}
      1.906 &  1.800 &  6.304\\
      1.800 &  5.612 & 12.587\\
      6.304 & 12.587 & 51.101
    \end{array}
   \right)\atop \ds 3.068(0.304)}$
   &
   ${\left(
    \begin{array}{@{\hspace{0ex}}r@{\hspace{1.4ex}}r@{\hspace{1.4ex}}r@{\hspace{0ex}}}
      1.901 &  1.800 &  6.296\\
      1.800 &  5.600 & 12.600\\
      6.296 & 12.600 & 51.108
    \end{array}
   \right)\atop \ds 1.680(0.169)}$
   &
   ${\left(
    \begin{array}{@{\hspace{0ex}}r@{\hspace{1.4ex}}r@{\hspace{1.4ex}}r@{\hspace{0ex}}}
      1.901 &  1.799 &  6.304\\
      1.799 &  5.596 & 12.594\\
      6.304 & 12.594 & 51.110
    \end{array}
   \right)\atop \ds 1.060(0.107)}$ \\
 \bottomrule
 \end{tabular}}
 \end{table}

\subsection{Linear Regression}
Let $S_j$, $S_j'$ be given training sets, $i\in S_j^c$, $i'\in S_{j'}^c$ and squared error loss such that $L_{s}(\h{y}_{S_j,i},y_i)=(\h{y}_{S_j,i}-y_i)^2=(\bm{x}_i^\ta\h{\beta}_{S_j}-y_i)^2$ and $L_{s}(\h{y}_{S_{j'},{i'}},y_{i'})=(\bm{x}_{i'}^\ta\h{\beta}_{S_{j'}}-y_{i'})^2$.

Define the quantities
\[
\begin{array}{l@{\qquad}l}
  C_1\doteq\sum\limits_{i=1}^{n}\bm{x}_i^\ta V\bm{x}_i,
   &
    C_3\doteq2\mathop{\sum\sum}\limits_{1\le i<i'\le n}(\bm{x}_i^\ta V\bm{x}_{i'})^2=\sum\limits_{i=1}^n\sum\limits_{i=1}^n(\bm{x}_i^\ta V\bm{x}_{i'})^2-C_2,\\
  C_2\doteq\sum\limits_{i=1}^{n}(\bm{x}_i^\ta V\bm{x}_i)^2,
   &
    C_4\doteq2\mathop{\sum\sum}\limits_{1\le i<i'\le n}\bm{x}_i^\ta V\bm{x}_{i}\bm{x}_{i'}^\ta V\bm{x}_{i'}=C_1^2-C_2.
\end{array}
\]
Write the $n\times n$ matrix
\[
nH=nX(X^\ta X)^{-1}X^\ta=XVX^\ta=(\bm{x}_1,\ldots,\bm{x}_n)^\ta V(\bm{x}_1,\ldots,\bm{x}_n)=(\bm{x}_i^\ta V \bm{x}_{i'}),
\]
where $H=(h_{ii'})$ is the hat matrix of regression, and observe that $C_1=n\tr(H)$, $C_2=n^2\sum_{i=1}^{n}h_{ii}^2$ and $C_3=n^22{\sum\sum}_{1\le i<i'\le n}h_{ii'}^2$. It is well known that $\tr(H)=\rank(X)=p$, $0\le\sum_{i=1}^{n}h_{ii}^2\le p$ and $\sum_{i=1}^{n}\sum_{i'=1}^{n}h_{ii'}^2=p$ \citep[see, for example,][]{CH1998}. Setting the parameter $\theta\doteq\sum_{i=1}^{n}h_{ii}^2\in[0,p]$, we have that
\begin{equation}
\label{eq.C2}
C_1=np,
\quad
C_2=n^2\theta,
\quad
C_3=n^2(p-\theta),
\quad
C_4=n^2(p^2-\theta).
\end{equation}

\paragraph{Normal Error Distribution}
Assume the errors in the linear regression model are normally distributed, that is $\bm{\varepsilon}\sim N_n(\bm{0},\sigma^2I)$. Let $\h{\bm{\beta}}_{S_j}$, $\h{\bm{\beta}}_{S_{j'}}$ be the ordinary least square (OLS) estimators computed on the training sets $S_j$, $S_{j'}$. It is well known that $\h{\bm{\beta}}_{S_j}=(X^\ta_{S_j}X_{S_j})^{-1}X^\ta_{S_j}\bm{y}_{S_j}$ (and similarly for the index $j'$), with $X_{S_j}=E_{S_j}X$ and $\bm{y}_{S_j}=E_{S_j}\bm{y}$ (and similarly for the index $j'$) are defined above as the product of $E_{S_j}$ and $X$. This allows one to write $\h{\bm{\beta}}_{S_j}=(X^\ta I_{S_j} X)^{-1} X^\ta I_{S_j} \bm{y}$ and $\h{\bm{\beta}}_{S_{j'}}=(X^\ta I_{S_{j'}} X)^{-1} X^\ta I_{S_{j'}} \bm{y}$.

Let the bivariate random vector
\[
\bm{\xi}_{j,j',i,i'}=(\xi_{j,i},\xi_{j',i'})^\ta \doteq (\bm{x}_i^\ta\h{\bm{\beta}}_{S_j}-y_i, \bm{x}_{i'}^\ta\h{\bm{\beta}}_{S_{j'}}-y_{i'})^\ta.
\]
Then, as $n_1$ becomes large
\begin{equation}
\label{eq.xi_j,j',i,i'}
\begin{split}
&\bm{\xi}_{j,j',i,i'}\stackrel{\rm d}{\approx}N_{2}\big(\bm{0},\Sig^{\bm{\xi}}_{j,j',i,i'}\big), \ \ \textrm{where}\\
&\Sig^{\bm{\xi}}_{j,j',i,i'}\doteq\sigma^2
\left(
  \begin{array}{@{\hspace{0ex}}c@{\hspace{2ex}}c@{\hspace{0ex}}}
    \frac{\bm{x}_i^\ta V \bm{x}_i}{n_1} +1 &
         \frac{\card(S_j\cap S_{j'})}{n_1^2}\bm{x}_{i}^{\ta}V\bm{x}_{i'}-\frac{\one{i'\in S_j}+\one{i\in S_{j'}}}{n_1}\bm{x}_{i}^{\ta}V\bm{x}_{i'}+\one{i=i'}
         \\
     * & \frac{\bm{x}_{i'}^\ta V \bm{x}_{i'}}{n_1} +1 \\
  \end{array}
\right);
\end{split}
\end{equation}
and the moments of $\bm{\xi}_{j,j',i,i'}$ are approximated by the associate moments of the approximated distribution.

\section{A useful lemma}
\label{sec.app.B}

We offer a lemma that is fundamental in proving the results presented in this paper.

\begin{lemma}
\label{lem.indices}
Let $S_j$ and $S_{j'}$ be two index sets of size $n_1<n$ (and $n_2=n-n_1)$ as in random CV; and let us consider the fixed indices $i\neq i'\neq i''\in N=\{1,\ldots,n\}$. Then,
\item{\rm(a)} $\Pr(\one{i\in S_j^c}=1)=\frac{n_2}n$.
\item{\rm(b)} $\Pr(\one{i,i'\in S_j^c}=1)=\frac{n_2(n_2-1)}{n(n-1)}$ and $\Pr(\one{i\in S_j^c,i'\in S_j}=1)=\frac{n_1n_2}{n(n-1)}$.
\item{\rm(c)} $\Pr(\one{i,\in S_j^c,i',i''\in S_j}=1)=\frac{n_1n_2(n_1-1)}{n(n-1)(n-2)}$.
\item{\rm(d)} $\E(\one{i\in S_j^c})=\frac{n_2}n$ and $\Var(\one{i\in S_j^c})=\frac{n_1n_2}{n^2}$.
\item{\rm(e)} $\E(\one{i,i'\in S_j^c})=\frac{n_2(n_2-1)}{n(n-1)}$.
\item{\rm(f)} $\E(\one{i\in S_j^c\cap S_{j'}^c})=\E(\one{i\in S_j^c,i'\in S_{j'}^c})=\frac{n_2^2}{n^2}$.
\item{\rm(g)} $\E(\one{i\in S_j^c,i'\in S_{j'}^c,i'\in S_{j}})=\E(\one{i\in S_j^c,i'\in S_{j'}^c,i\in S_{j'}})=\frac{n_1n_2^2}{n^2(n-1)}$.
\item{\rm(h)} $\E(\one{i\in S_j^c,i'\in S_{j'}^c,i'\in S_{j},i\in S_{j'}})=\frac{n_1^2n_2^2}{n^2(n-1)^2}$.
\item{\rm(i)} $\E(\card(S_j\cap S_{j'})\one{i\in S_j^c\cap S_{j'}^c})=\frac{n_1^2n_2^2}{n^2(n-1)}$, $\E(\card^2(S_j\cap S_{j'})\one{i\in S_j^c\cap S_{j'}^c})=\frac{n_1^2n_2^2(n_1(n_1-1)+n_2-1)}{n^2(n-1)(n-2)}$.
\item{\rm(j)} $\E(\card^2(S_j\cap S_{j'})\one{i\in S_j^c,i'\in S_{j'}^c})=\frac{n_1^2n_2^2[(n-2)^2+(n-3)(n_1-1)^2]}{n^2(n-1)^2(n-2)}$.
\item{\rm(k)} $\E(\card(S_j\cap S_{j'})\one{i\in S_j^c,i'\in S_{j'}^c,i\in S_{j'}})=\E(\card(S_j\cap S_{j'})\one{i\in S_j^c,i' S_{j'}^c,i'\in S_{j}})=\frac{n_1^2n_2^2(n_1-1)}{n^2(n-1)^2}$.
\end{lemma}

 }

\newpage
\textcolor{white}{An empty page}

\setcounter{section}{0}
\setcounter{page}{1}
\setcounter{table}{0} \renewcommand\thetable{OA.\arabic{table}}
\setcounter{theorem}{0} \renewcommand\thetheorem{OA.\arabic{theorem}}
\setcounter{equation}{0}
\thispagestyle{empty}
\begin{center}
\Large\bf
Supplementary material
\end{center}


\section{Proofs}
\label{sec.online.2}
\begin{pr}{Proof of Lemma \ref{lem.indices}}
(a) $\Pr(\one{i\in S_j^c}=1)=\Pr(i\notin S_j)={n-1\choose n_1}{n\choose n_1}^{-1}=\frac{n_2}n$.
\medskip

\noindent
(b) $\Pr(\one{i\ne i'\in S_j^c}=1)=\Pr(i,i'\notin S_j)={n-2\choose n_1}{n\choose n_1}^{-1}=\frac{n_2(n_2-1)}{n(n-1)}$ and
$\Pr(\one{i\in S_j^c,i'\in S_j}=1)=\Pr(i\notin S_j,i'\in S_j^c)={n-2\choose n_1-1}{n\choose n_1}^{-1}=\frac{n_1n_2}{n(n-1)}$.
\medskip

\noindent
(c) As above, $\Pr(\one{i,\in S_j^c,i',i''\in S_j}=1)=\Pr(i',i''\in S_j,i\notin S_j)={n-3\choose n_1-2}{n\choose n_1}^{-1}=\frac{n_1(n_1-1)n_2}{n(n-1)(n-2)}$.
\medskip

\noindent
(d) It follows immediately from (a).
\medskip

\noindent
(e) $\E(\one{i,i'\in S_j^c})=\Pr(\one{i,i'\in S_j^c}=1)=\frac{n_2(n_2-1)}{n(n-1)}$, see (b).
\medskip

\noindent
(f) Due to independence of $S_j$ and $S_{j'}$ and using (a),
\[
\E(\one{i\in S_j^c\cap S_{j'}^c})=\Pr(\one{i\in S_j^c\cap S_{j'}^c}=1)=\Pr(i\in S_j^c,i\in S_{j'}^c)=\Pr(i\in S_j^c)\Pr(i\in S_{j'}^c)=\frac{n_2^2}{n^2}
\]
and similarly for $i\ne i'$.
\medskip

\noindent
(g) As in (f), using (a) and (b),
\[
\E(\one{i\in S_j^c,i'\in S_{j'}^c,i'\in S_{j}})=\Pr(i'\in S_j,i\notin S_j)\Pr(i'\notin S_{j'})=\frac{n_1n_2^2}{n^2(n-1)}.
\]
(h) Similarly to (g), using (b),
\[
\E(\one{i\in S_j^c,i'\in S_{j'}^c,i'\in S_{j},i\in S_{j'}})=\Pr(i'\in S_j,i\notin S_j)\Pr(i\in S_{j'},i'\notin S_{j'})=\frac{n_1^2n_2^2}{n^2(n-1)^2}.
\]
(i) For a simple notation, in the sequel of the proof we use the exchangeability of the indices. Set $i=n$ and $i'=n-1$. Also, observe that $\card(S_j\cap S_{j'})=\sum_{k=1}^{n}\one{k\in S_j\cap S_{j'}}$. So, using (b),
\[
\begin{split}
 \E(\card(S_j\cap S_{j'})\one{n\in S_j^c\cap S_{j'}^c})
 =\sum\limits_{k=1}^{n}\E(\one{k\in S_j\cap S_{j'},n\in S_j^c\cap S_{j'}^c})
 =\sum\limits_{k=1}^{n-1}\E(\one{k\in S_j\cap S_{j'},n\in S_j^c\cap S_{j'}^c})\\
=
 \sum\limits_{k=1}^{n-1}\Pr(k\in S_j,n\notin S_j)\Pr(k\in S_{j'},n\notin S_{j'})
 =(n-1)\frac{n_1^2n_2^2}{n^2(n-1)^2}=\frac{n_1^2n_2^2}{n^2(n-1)}.
\end{split}
\]
Write $\card^2(S_j\cap S_{j'})$ as $\big(\sum_{k=1}^{n-1}\one{k\in S_j\cap S_{j'}}\big)^2+2\one{n\in S_j\cap S_{j'}}\sum_{k=1}^{n-1}\one{k\in S_j\cap S_{j'}}+\one{n\in S_j\cap S_{j'}}$ and, since $\one{n\in S_j\cap S_{j'}}\one{n\in S_j^c\cap S_{j'}^c}=0$, observe that $\card^2(S_j\cap S_{j'})\one{n\in S_j^c\cap S_{j'}^c}$ is
\[
\begin{split}
\card^2(S_j\cap S_{j'})\one{n\in S_j^c\cap S_{j'}^c}&=
\left(\sum\limits_{k=1}^{n-1}\one{k\in S_j\cap S_{j'}}\right)^2\one{n\in S_j^c\cap S_{j'}^c}\\
&=
\sum\limits_{k=1}^{n-1}\one{k\in S_j\cap S_{j'},n\in S_j^c\cap S_{j'}^c}
 +2\mathop{\sum\sum}\limits_{1\le k<k'\le n-1}\one{k,k'\in S_j\cap S_{j'},n\in S_j^c\cap S_{j'}^c}.
\end{split}
\]
Hence, the desired expected value is
\[
\begin{split}
&\sum\limits_{k=1}^{n-1}
\Pr^2(k\in S_j,n\notin S_j)
 +2\mathop{\sum\sum}\limits_{1\le k<k'\le n-1}\Pr^2(k,k'\in S_j,n\notin S_j)\\
=&~
\frac{n_1^2n_2^2}{n^2(n-1)}+(n-1)(n-2)\left(\frac{n_1(n_1-1)n_2}{n(n-1)(n-2)}\right)^2
=\frac{n_1^2n_2^2}{n^2(n-1)}+\frac{n_1^2n_2^2(n_1-1)^2}{n^2(n-1)(n-2)}.
\end{split}
\]
(j) Write $\card(S_j\cap S_{j'})=\sum_{k=1}^{n-2}\one{k\in S_j\cap S_{j'}}+\one{n-1\in S_j\cap S_{j'}}+\one{n\in S_j\cap S_{j'}}$. As in (i)
\[
\begin{split}
\card^2(S_j\cap S_{j'})\one{n\in S_j^c,n-1\in S_{j'}^c}
=&~
\sum\limits_{k=1}^{n-2}\one{k\in S_j\cap S_{j'},n\in S_j^c,n-1\in S_{j'}^c}\\
 &
 +2\mathop{\sum\sum}\limits_{1\le k<k'\le n-1}\one{k,k'\in S_j\cap S_{j'},n\in S_j^c,n-1\in S_{j'}^c}.
\end{split}
\]
The expected value is
\[
\begin{split}
&
\sum\limits_{k=1}^{n-2}\Pr(k\in S_j,n\notin S_j)\Pr(k\in S_{j'},n-1\notin S_{j'})\\
&
+2\mathop{\sum\sum}\limits_{1\le k<k'\le n-2}\Pr(k,k'\in S_j,n\notin S_j)\Pr(k,k'\in S_{j'},n-1\notin S_{j'})\\
=&~
(n-2)\frac{n_1^2n_2^2}{n^2(n-1)^2}+(n-3)\frac{n_1^2n_2^2(n_1-1)^2}{n^2(n-1)^2(n-2)}.
\end{split}
\]
(k) Similarly to (j), $\card(S_j\cap S_{j'})\one{n\in S_j^c,n-1\in S_{j'}^c,n\in S_{j'}}
=\sum_{k=1}^{n-2}\one{k\in S_j\cap S_{j'},n\in S_j^c,n-1\in S_{j'}^c,n\in S_{j'}}$; and the expected value is
\[
\sum\limits_{k=1}^{n-2}\Pr(k\in S_j,n\notin S_j)\Pr(k,n\in S_{j'},n-1\notin S_{j'})=\frac{n_1^2n_2^2(n_1-1)}{n^2(n-1)^2},
\]
completing the proof.
\end{pr}

\begin{pr}{Proof of Theorem \ref{theo.var(mu_j)}}
Write $\h{\mu}_j=\frac{1}{n_2}\sum_{i=1}^n L(\h{y}_{S_j,i},y_i)\one{i\in S_j^c}$. So, $\h{\mu}_j|S_j=\frac{1}{n_2}\sum_{i=1}^n \one{i\in S_j^c}\linebreak
L(\h{y}_{S_j,i},y_i|S_j,i)$
and $\h{\mu}_j^2|S_j=\frac{1}{n_2^2}\Big\{\sum_{i=1}^n \one{i\in S_j^c}L^2(\h{y}_{S_j,i},y_i)|S_j,i
+2\mathop{\sum\sum}_{1\le i<i'\le n}\one{i,i'\in S_j^c}\linebreak
L(\h{y}_{S_j,i},y_i)L(\h{y}_{S_j,i'},y_{i'})|S_j,i,i'\Big\}$, since $\one{A}\one{B}=\one{A\cap B}$ and because the quantities $\one{i\in S_j^c}$, $\one{i,i'\in S_j^c}$ are constants. Therefore,
\[
\E(\h{\mu}_j|S_j)=\frac{1}{n_2}\sum_{i=1}^n \one{i\in S_j^c}\E[L(\h{y}_{S_j,i},y_i)|S_j,i]=\frac{1}{n_2}\sum\limits_{i=1}^n \one{i\in S_j^c}\sfe_i,
\]
\[
\begin{split}
\E(\h{\mu}_j^2|S_j)&
=\frac{1}{n_2^2}\left\{\sum\limits_{i=1}^n \one{i\in S_j^c}\E(L^2(\h{y}_{S_j,i},y_i)|S_j,i)\right.\\
&
\quad\qquad+2\left.\mathop{\sum\sum}\limits_{1\le i<i'\le n}\one{i,i'\in S_j^c}\E[L(\h{y}_{S_j,i},y_i)L(\h{y}_{S_j,i'},y_{i'})|S_j,i,i']\right\}\\
&
=\frac{1}{n_2^2}\left\{\sum\limits_{i=1}^n \one{i\in S_j^c}\sfe_{i,i}
+2\mathop{\sum\sum}\limits_{1\le i<i'\le n}\one{i,i'\in S_j^c}\sfe_{i,i'}\right\}.
\end{split}
\]
Using Lemma \ref{lem.indices}, we get
$\E(\h{\mu}_j)=\E[\E(\h{\mu}_j|S_j)]=\frac{1}{n}\sum_{i=1}^n\sfe_i$, $\E(\h{\mu}_j^2)=\E[\E(\h{\mu}_j^2|S_j)]=\frac{1}{nn_2}\big\{\sum_{i=1}^n\sfe_{i,i}+\frac{2(n_2-1)}{n-1}\mathop{\sum\sum}_{1\le i<i'\le n}\sfe_{i,i'}\big\}$,
completing the proof.
\end{pr}

\begin{pr}{Proof of \eqref{eq.xi_j,j',i,i'}}
$\bm{\xi}_{j,j',i,i'}$ can be written as $(I_{S_j}X(X^\ta I_{S_j} X)^{-1}\bm{x}_i -\bm{e}_i,I_{S_{j'}}X(X^\ta I_{S_j} X)^{-1}\bm{x}_{i'}-\bm{e}_{i'})^\ta\bm{y}$. Using standard properties of multivariate normal distribution, after some algebra, we get $\bm{\xi}_{j,j',i,i'}\sim N_2(\bm{0},\Sig_{j,j',i,i'})$, where $\Sig_{j,j',i,i'}=(\sigma_{rc})$ has elements:
\begin{align*}
\sigma_{11}=\sigma^2[&\bm{x}_i^{\ta}(X^{\ta}I_{S_j}X)^{-1}\bm{x}_i+1],\\
\sigma_{22}=\sigma^2[&\bm{x}_{i'}^{\ta}(X^{\ta}I_{S_{j'}}X)^{-1}\bm{x}_{i'}+1],\\
\sigma_{12}=\sigma^2[&\bm{x}_i^{\ta}(X^{\ta}I_{S_j}X)^{-1}X^{\ta}I_{S_j{\cap}S_j}X (X^{\ta}I_{S_{j'}}X)^{-1}\bm{x}_{i'}
                      -\bm{x}_i^{\ta}(X^{\ta}I_{S_j}X)^{-1}X^{\ta}\bm{x}_{i'}\one{i'\in S_j}\\
                     &-\bm{x}_i^{\ta}(X^{\ta}I_{S_{j'}}X)^{-1} X^{\ta}\bm{x}_{i'}\one{i\in S_{j'}}+\one{i'=i}].
\end{align*}
As $n_1$ becomes large, in view of \eqref{eq.lim.card}, $(X^\ta I_{S_j} X)^{-1}$, $(X^\ta I_{S_{j'}} X)^{-1}$ approximated by $n_1^{-1}V$. Also, when $\card(S_j\cap S_{j'})\ne0$ we write the matrix  $n_1(X^\ta I_{S_{j'}} X)^{-1}(X^\ta I_{S_j\cap S_{j'}} X)\times(X^\ta I_{S_j} X)^{-1}$ as
\[
\frac{\card(S_j\cap S_{j'})}{n_1}(n_1^{-1}X^\ta I_{S_{j'}} X)^{-1}\frac{1}{\card(S_j\cap S_{j'})}X^\ta I_{S_j\cap S_{j'}} X(n_1^{-1}X^\ta I_{S_j} X)^{-1},
\]
if $\card(S_j\cap S_{j'})$ takes large values as $n_1$ becomes large, $\frac{1}{\card(S_j\cap S_{j'})}X^\ta I_{S_j\cap S_{j'}} X\simeq V^{-1}$; and if $\lim_{n_1\to\infty}\frac{\card(S_j\cap S_{j'})}{n_1}=0$, the matrix $\frac{1}{n_1}X^\ta I_{S_j\cap S_{j'}} X$ tends to $p\times p$ zero matrix [$\h{\bm{\beta}}_{S_{j}}$ and $\h{\bm{\beta}}_{S_{j'}}$ tend to be independent] thus this matrix is approximated by $\frac{\card(S_j\cap S_{j'})}{n_1}V^{-1}$. So, for large values of $n_1$ the variance-covariance matrix $\Sig_{j,j',i,i'}$ is approximated by the matrix $\Sig^{\bm{\xi}}_{j,j',i,i'}$. Finally, the moments of $\bm{\xi}_{j,j',i,i'}$ are continues functions of its variance-covariance matrix, which completes the proof.
\end{pr}

\begin{theorem}[\citealp{Isserlis1918}]
\label{theo.isserlis}
If $(X_1,X_2)^\ta\sim N_2\big(\bm{0},\Sig=(\sigma_{ij})\big)$, then $\E(X_i^4)=3\sigma_{ii}^4$ and $\E(X_1^2X_2^2)=\sigma_{11}\sigma_{22}+2\sigma_{12}^2$.
\end{theorem}

\begin{pr}{Proof of Proposition \ref{prop.E,V,C-randomCV}}
Let $\sfe_i=\E[L_{s}(\h{y}_{S_j,i},y_i|S_j,i)]$, $\sfe_{i,i'}=\E[L_{s}(\h{y}_{S_j,i},y_i)L_{s}\big(\h{y}_{S_j,i'},y_{i'})|\\S_j,i,i']$.
From \eqref{eq.xi_j,j',i,i'} and Theorem \ref{theo.isserlis} we obtain $\sfe_i=\E(\xi_{j,i}^2)=\sigma^2\left(1+\frac{\bm{x}_i^\ta V \bm{x}_i}{n_1}\right)$, $\sfe_{i,i}=\E(\xi_{j,i}^4)=3\sigma^4\left(1+\frac{\bm{x}_i^\ta V \bm{x}_i}{n_1}\right)^2$. Now, in \eqref{eq.xi_j,j',i,i'} assume that $j={j'}$ and $i\ne i'\in S_j^c$. Observe that $\card(S_j\cap S_{j'})=n_1$ and the indicators are zero. Thus, using Theorem \ref{theo.isserlis},
$\sfe_{i,i'}=\E(\xi_{j,i}^2\xi_{j,i'}^2)=\sigma^4\left[\left(1+\frac{\bm{x}_i^\ta V\bm{x}_i}{n_1}\right)\left(1+\frac{\bm{x}_{i'}^\ta V\bm{x}_{i'}}{n_1}\right)+2\frac{(\bm{x}_{i}^\ta V\bm{x}_{i'})^2}{n_1^2}\right]$.
Applying Theorem \ref{theo.var(mu_j)}, $\E(\h{\mu}_j)=\sigma^2\left(1+\frac{p}{n_1}\right)$. For the variance of $\h{\mu}_j$ we write $n\sfe_{i,i}-n_2\sfe_{i}^2=\sigma^4(3n-n_2)\Big(1+\frac{\bm{x}_i^\ta V \bm{x}_i}{n_1}\Big)^2=\sigma^4(2n+n_1)\Big(1+\frac{2\bm{x}_i^\ta V \bm{x}_i}{n_1}+\frac{(\bm{x}_i^\ta V \bm{x}_i)^2}{n_1^2}\Big)$; and so,
\[
\sum_{i=1}^n\big(n\sfe_{i,i}-n_2\sfe_{i}^2\big)
=\sigma^4\left(n(2n+n_1)+\frac{2(2n+n_1)}{n_1}C_1+\frac{2n+n_1}{n_1^2}C_2\right).
\]
Also, the quantity $\sigma^{-4}[n(n_2-1)\sfe_{i,i'}-(n-1)n_2\sfe_{i}\sfe_{i'}]$ is
\[
\begin{split}
 n(n_2-1)&
 \textstyle
 \left[\left(1+\frac{\bm{x}_i^\ta V\bm{x}_i}{n_1}\right)\left(1+\frac{\bm{x}_{i'}^\ta V\bm{x}_{i'}}{n_1}\right)
  +\frac{2}{n_1^{2}}(\bm{x}_{i}^\ta V\bm{x}_{i'})^2\right]
  -(n-1)n_2\left(1+\frac{\bm{x}_i^\ta V\bm{x}_i}{n_1}\right)\left(1+\frac{\bm{x}_{i'}^\ta V\bm{x}_{i'}}{n_1}\right)\\
=&
-n_1\left(1+\frac{\bm{x}_i^\ta V\bm{x}_i}{n_1}\right)\left(1+\frac{\bm{x}_{i'}^\ta V\bm{x}_{i'}}{n_1}\right)
  +\frac{2n(n_2-1)}{n_1^{2}}(\bm{x}_{i}^\ta V\bm{x}_{i'})^2\\
=&
-n_1 -(\bm{x}_i^\ta V\bm{x}_i+\bm{x}_{i'}^\ta V\bm{x}_{i'})-\frac{1}{n_1}\bm{x}_i^\ta V\bm{x}_i\bm{x}_{i'}^\ta V\bm{x}_{i'}
  +\frac{2n(n_2-1)}{n_1^{2}}(\bm{x}_{i}^\ta V\bm{x}_{i'})^2.
\end{split}
\]
Observing that
$\sum\sum_{1\le i<i'\le n}\bm{x}_{i'}^\ta V\bm{x}_{i'}
=\sum\sum_{1\le i<i'\le n}\bm{x}_{i}^\ta V\bm{x}_{i}
=\frac{1}{2}\sum_{i=1}^n\sum_{i'=1,i'\ne i}^n\bm{x}_{i}^\ta V\bm{x}_{i}
=\frac{n-1}{2}C_1$,
we get
\[
\begin{split}
&~\mathop{\sum\sum}\limits_{1\le i<i'\le n}[n(n_2-1)\sfe_{i,i'}-(n-1)n_2\sfe_{i}\sfe_{i'}]\\
=&~\sigma^4\left\{-\frac{n(n-1)n_1}{2}-(n-1)C_1-\frac{1}{2n_1}C_4+\frac{n(n_2-1)}{n_1^{2}}C_3\right\}.
\end{split}
\]
Thus, the quantity $\sum_{i=1}^n\big(n\sfe_{i,i}-n_2\sfe_{i}^2\big)+\frac{2}{n-1}\sum\sum_{1\le i<i'\le n}[n(n_2-1)\sfe_{i,i'}-(n-1)n_2\sfe_{i}\sfe_{i'}]$ is $2n^2+\frac{4n}{n_1}C_1+\frac{2n+n_1}{n_1^2}C_2+\frac{2n(n_2-1)}{(n-1)n_1^2}C_3-\frac{1}{(n-1)n_1}C_4$; and using \eqref{eq.C2},
\begin{equation}
\label{eq.var_sqr}
\tag{$*$}
\Var(\h{\mu}_j)=\sigma^4\left\{\frac{2}{n_2}+\frac{4p}{n_1n_2}+\frac{(3n+1)\theta}{(n-1)n_1n_2}+\frac{(2n(n_2-1)-n_1p)p}{(n-1)n_1^2n_2}\right\}.
\end{equation}

Now observe that $\h{\mu}_j|S_j=\frac{1}{n_2}\sum_{i=1}^n \xi_{j,i}^2\one{i\in S_j^c}$ and $\h{\mu}_{j'}|S_{j'}=\frac{1}{n_2}\sum_{i'=1}^n \xi_{j',i'}^2\one{i\in S_{j'}^c}$. Thus,
\[
\h{\mu}_{j}\h{\mu}_{j'}|S_{j},S_{j'}=\frac{1}{n_2^2}\left\{\sum\limits_{i=1}^n\xi_{j,i}^2\xi_{j',i}^2\one{i\in S_j^c\cap S_{j'}^c}
+2\mathop{\sum\sum}\limits_{1\le i<i'\le n}\xi_{j,i}^2\xi_{j',i'}^2\one{i\in S_j^c,i'\in S_{j'}^c}\right\},
\]
noting that, since $S_{j}$ and $S_{j'}$ are given, the indicator functions are constants. Therefore,
\[
\E(\h{\mu}_{j}\h{\mu}_{j'}|S_{j},S_{j'})=\frac{1}{n_2^2}\left\{\sum\limits_{i=1}^n\one{i\in S_j^c\cap S_{j'}^c}\E(\xi_{j,i}^2\xi_{j',i}^2)
+2\mathop{\sum\sum}\limits_{1\le i<i'\le n}\one{i\in S_j^c,i'\in S_{j'}^c}\E(\xi_{j,i}^2\xi_{j',i'}^2)\right\}.
\]
In view of \eqref{eq.xi_j,j',i,i'} and using Theorem \ref{theo.isserlis}, we obtain the following. If $i=i'$, since $\one{i'\in S_j}\one{i\in S_j^c\cap S_{j'}^c}=\one{i\in S_{j'}\cap S_j^c\cap S_{j'}^c}=0$ and $\one{i=i'}=1$,
\[
\frac{\E(\xi_{j,i}^2\xi_{j',i}^2)}{\sigma^{4}}=\left(1+\frac{\bm{x}_i^{\ta}V\bm{x}_i}{n_1}\right)^2+2\left(\frac{\card(S_j\cap S_{j'})}{n_1^2}\bm{x}_i^{\ta}V\bm{x}_i+1\right)^2.
\]
and if $i\ne i'$, $\one{i=i'}=0$, and
\[
\frac{\E(\xi_{j,i}^2\xi_{j',i'}^2)}{\sigma^{4}}
=\left(1+\frac{\bm{x}_i^{\ta}V\bm{x}_i}{n_1}\right)\left(1+\frac{\bm{x}_{i'}^{\ta}V\bm{x}_{i'}}{n_1}\right)
+2(\bm{x}_i^{\ta}V\bm{x}_{i'})^2\left(\frac{\card(S_j\cap S_{j'})}{n_1^2}-\frac{\one{i'\in S_j}+\one{i\in S_{j'}}}{n_1}\right)^2.
\]
Write $\E(\xi_{j,i}^2\xi_{j',i}^2)$ as $\sigma^4\sum_{r=0}^{2}A_{r;i}\card^{r}(S_j\cap S_{j'})$, where $A_{0;i}=3+\frac{2}{n_1}\bm{x}_i^\ta V \bm{x}_i+\frac{1}{n_1^2}(\bm{x}_i^\ta V \bm{x}_i)^2$, $A_{1;i}=\frac{4}{n_1^2}\bm{x}_i^\ta V \bm{x}_i$ and $A_{2;i}=\frac{2}{n_1^4}(\bm{x}_i^\ta V \bm{x}_i)^2$. Using Lemma \ref{lem.indices},
\[
\begin{split}
\E[\one{i\in S_j^c\cap S_{j'}^c}\E(\xi_{j,i}^2\xi_{j',i}^2)]
&=
\sigma^4\sum_{r=0}^{2}A_{r;i}\E[\card^{r}(S_j\cap S_{j'})\one{i\in S_j^c\cap S_{j'}^c}]\\
&=
\sigma^4\left\{\frac{n_2^2}{n^2}A_{0;i}+\frac{n_1^2n_2^2}{n^2(n-1)}A_{1;i}+\frac{n_1^2n_2^2(n_1(n_1-1)+n_2-1)}{n^2(n-1)(n-2)}A_{2;i}\right\}.
\end{split}
\]
Since $\sum_{i=1}^{n}A_{0;i}=3n+\frac{2}{n_1}C_1+\frac{1}{n_1^2}C_2$, $\sum_{i=1}^{n}A_{1;i}=\frac{4}{n_1^2}C_1$ and $\sum_{i=1}^{n}A_{2;i}=\frac{2}{n_1^4}C_2$, after some algebra,
\[
\begin{split}
&~\sum\limits_{i=1}^{n}\E[\one{i\in S_j^c\cap S_{j'}^c}\E(\xi_{j,i}^2\xi_{j',i}^2)]\\
=&~\sigma^4\left\{\frac{3n_2^2}{n}+\frac{2n_2^2(n+2n_1-1)}{n^2(n-1)n_1}C_1+\frac{n_2^2[n(n-1)+2n_1(n_1-2)]}{n^2(n-1)(n-2)n_1^2}C_2\right\}.
\end{split}
\]
Now write $\E(\xi_{j,i}^2\xi_{j',{i'}}^2)=\sigma^4\sum_{r=0}^{1}B_{r;i,i'}\Big(\frac{\card(S_j\cap S_{j'})}{n_1^2}-\frac{\one{i'\in S_j}+\one{i\in S_{j'}}}{n_1}\Big)^{2r}$, where $B_{0;i,i'}=1+\frac{\bm{x}_i^\ta V \bm{x}_i+\bm{x}_{i'}^\ta V \bm{x}_{i'}}{n_1}+\frac{\bm{x}_i^\ta V \bm{x}_i\bm{x}_{i'}^\ta V \bm{x}_{i'}}{n_1^2}$ and $B_{1;i,i'}=2(\bm{x}_{i}^\ta V \bm{x}_{i'})^2$. Denote $\varGamma_{S_j,S_{j'};i,i'}=\Big(\frac{\card(S_j\cap S_{j'})}{n_1^2}-\frac{\one{i'\in S_j}+\one{i\in S_{j'}}}{n_1}\Big)^{2}\one{i\in S_j^c,i'\in S_{j'}^c}$ and observe that
\[
\begin{split}
 \varGamma_{S_j,S_{j'};i,i'}=
 &
 ~\frac{\card^2(S_j\cap S_{j'})\one{i\in S_j^c,i'\in S_{j'}^c}}{n_1^4}\\
 &
 -\frac{2\card(S_j\cap S_{j'})}{n_1^3}[\one{i\in S_j^c,i'\in S_{j'}^c,i'\in S_{j}}+\one{i\in S_j^c,i'\in S_{j'}^c,i\in S_{j'}}]\\
 &
 +\frac{2}{n_1^2}\one{i\in S_j^c,i'\in S_{j'}^c,i'\in S_{j},i\in S_{j'}}
  +\frac{1}{n_1^2}[\one{i\in S_j^c,i'\in S_{j'}^c,i'\in S_{j}}+\one{i\in S_j^c,i'\in S_{j'}^c,i\in S_{j'}}].
\end{split}
\]
From Lemma \ref{lem.indices}, after some algebra, we get
\[
\E(\varGamma_{S_j,S_{j'};i,i'})=\frac{n_2^2[(n-2)(n+n_1^2+2n_1n_2-1)-(n_1-1)^2]}{n^2(n-1)^2(n-2)n_1^2}.
\]
Hence,
\[
\begin{split}
\E[\one{i\in S_j^c,i'\in S_{j'}^c}&\E(\xi_{j,i}^2\xi_{j',i'}^2)]
 =\sigma^4\left\{B_{0;i,i'}\E(\one{i\in S_j^c,i'\in S_{j'}^c})+B_{1;i,i'}\E(\varGamma_{S_j,S_{j'};i,i'})\right\}\\
&=
\sigma^4\left\{\frac{n_2^2}{n^2} B_{0;i,i'}
                 +\frac{n_2^2[(n-2)(n+n_1^2+2n_1n_2-1)-(n_1-1)^2]}{n^2(n-1)^2(n-2)n_1^2}B_{1;i,i'}\right\}.
\end{split}
\]
Since $\sum\sum_{1\le i<i'\le n}B_{0;i,i'}=\frac{n(n-1)}{2}+\frac{n-1}{n_1}C_1+\frac{1}{2n_1^2}C_4$ and $\sum\sum_{1\le i<i'\le n}B_{0;i,i'}=\frac{1}{2n_1^2}C_3$,
\[
\begin{split}
2\mathop{\sum\sum}\limits_{1\le i<i'\le n}\E[\one{i\in S_j^c,i'\in S_{j'}^c}\E(\xi_{j,i}^2\xi_{j',i'}^2)]
=
\sigma^4
\left\{
\frac{(n-1)n_2^2}{n}
+\frac{2(n-1)n_2^2}{n^2n_1}C_1\right.\qquad\qquad\\
\left.+\frac{n_2^2[(n-2)(n+n_1^2+2n_1n_2-1)-(n_1-1)^2]}{n^2(n-1)^2(n-2)n_1^4}C_3
+\frac{n_2^2}{n^2n_1^2}C_4
\right\}.
\end{split}
\]
Thus, by $\E(\h{\mu}_{j}\h{\mu}_{j'})=\E[\E(\h{\mu}_{j}\h{\mu}_{j'}|S_j,S_{j'})]$,
\[
\begin{split}
\E(\h{\mu}_{j}\h{\mu}_{j'})&
=\frac{1}{n_2^2}\left\{\sum\limits_{i=1}^n\E[\one{i\in S_j^c\cap S_{j'}^c}\E(\xi_{j,i}^2\xi_{j',i}^2)]
+2\mathop{\sum\sum}\limits_{1\le i<i'\le n}\E[\one{i\in S_j^c,i'\in S_{j'}^c}\E(\xi_{j,i}^2\xi_{j',i'}^2)]\right\}\\
&=\sigma^4
\left\{
\frac{n+2}{n}
+\frac{2(n^2+n_1-n_2)}{n^2(n-1)n_1}C_1
+\frac{n(n-1)+2n_1(n_1-2)}{n^2(n-1)(n-2)n_1^2}C_2
\right.\\
&\quad\qquad\left.
+\frac{(n-2)(n+n_1^2+2n_1n_2-1)-(n_1-1)^2}{n^2(n-1)^2(n-2)n_1^4}C_3
+\frac{1}{n^2n_1^2}C_4
\right\}.
\end{split}
\]
Since $\E(\h{\mu}_{j})\E(\h{\mu}_{j'})=\E^2(\h{\mu}_{j})=\sigma^4\Big(1+\frac{2C_1}{nn_1}+\frac{C_1^2}{n^2n_1^2}\Big) =\sigma^4\Big(1+\frac{2C_1}{nn_1}+\frac{C_2}{n^2n_1^2}+\frac{C_4}{n^2n_1^2}\Big)$, the relation $\Cov(\h{\mu}_{j},\h{\mu}_{j'})=\E(\h{\mu}_{j}\h{\mu}_{j'})-\E(\h{\mu}_{j})\E(\h{\mu}_{j'})$ gives
\[
\begin{split}
\Cov(\h{\mu}_{j},\h{\mu}_{j'})
&=\sigma^4
\left\{
\frac{2}{n}
+\frac{n+2n_1}{n^2(n-1)n_1}C_1
+\frac{2(n+n_1(n_1-2)-1)}{n^2(n-1)(n-2)n_1^2}C_2
\right.\\
&\quad\qquad\left.
+\frac{(n-2)(n+n_1^2+2n_1n_2-1)-(n_1-1)^2}{n^2(n-1)^2(n-2)n_1^4}C_3
\right\}.
\end{split}
\]
Finally, \eqref{eq.C2} completes the proof.
\end{pr}

\begin{pr}{Proof of Proposition \ref{prop.E,V,C-kfoldCV}}
The behavior of $\h{\mu}_j$ is the same as in random CV case. Thus, $\E(\h{\mu}_j)$ follows by Proposition \ref{prop.E,V,C-randomCV}, replacing $n_1$ by $\frac{(k-1)n}{k}$; also, \eqref{eq.var_sqr} (see above) gives
\[
\Var(\h{\mu}_j)
=\sigma^4\left\{\frac{2k}{n}+\frac{4k^2p}{(k-1)n^2}+\frac{3k^2\theta}{(k-1)n^2}+\frac{pk^3}{(k-1)^2n^2}\right\}+o(\sfrac{1}{n^2}).
\]

To compute $\Cov(\h{\mu}_j,\h{\mu}_{j'})$ we have that $\one{i\in S_j^c\cap S_{j'}^c}=0$ ($S_j^c\cap S_{j'}^c=\varnothing$) and $\card(S_j\cap S_{j'})=\frac{(k-2)n}{k}$, thus $\E(\h{\mu}_{j}\h{\mu}_{j'}|S_{j},S_{j'})=\frac{2k^2}{n^2}\mathop{\sum\sum}_{1\le i<i'\le n}\one{i\in S_j^c,i'\in S_{j'}^c}\E(\xi_{j,i}^2\xi_{j',i'}^2)$, where
\[
\frac{\E(\xi_{j,i}^2\xi_{j',i'}^2)}{\sigma^4}=\left(1+\frac{k\bm{x}_i^{\ta}V\bm{x}_i}{(k-1)n}\right)\left(1+\frac{k\bm{x}_{i'}^{\ta}V\bm{x}_{i'}}{(k-1)n}\right)
+\frac{2k^2(\bm{x}_i^{\ta}V\bm{x}_{i'})^2}{(k-1)^2n^2}\left(\frac{k-2}{k-1}-\one{i'\in S_j}-\one{i\in S_{j'}}\right)^2.
\]
Now observe that
\[
\begin{split}
&~
\one{i\in S_j^c,i'\in S_{j'}^c}\left(\frac{k-2}{k-1}-\one{i'\in S_j}-\one{i\in S_{j'}}\right)^2\\
=&~\left(\one{i\in S_j^c,i'\in S_{j'}^c}\frac{k-2}{k-1}-\one{i\in S_j^c,i'\in S_{j'}^c,i'\in S_j}-\one{i\in S_j^c,i'\in S_{j'}^c,i\in S_{j'}}\right)^2\\
=&~\left(\one{i\in S_j^c,i'\in S_{j'}^c}\frac{k-2}{k-1}-2\one{i\in S_j^c,i'\in S_{j'}^c}\right)^2=\one{i\in S_j^c,i'\in S_{j'}^c}\left(\frac{k-2}{k-1}-2\right)^2
=\one{i\in S_j^c,i'\in S_{j'}^c}\frac{k^2}{(k-1)^2};
\end{split}
\]
and $\E(\one{i\in S_j^c,i'\in S_{j'}^c})
=\Pr(i'\in S_{j'}^c|i\in S_j^c)\Pr(i\in S_j^c)
=\frac{{n-2\choose \frac{n}{k}-1}}{{n-1\choose \frac{n}{k}}}\frac{{n-1\choose \frac{n}{k}-1}}{{n\choose \frac{n}{k}}}
=\frac{n}{k^2(n-1)}$. Thus, $\E[\one{i\in S_j^c,i'\in S_{j'}^c}\E(\xi_{j,i}^2\xi_{j',i'}^2)]=\frac{n\sigma^4}{k^2(n-1)} \bigg[\left(1+\frac{k\bm{x}_i^{\ta}V\bm{x}_i}{(k-1)n}\right)\left(1+\frac{k\bm{x}_{i'}^{\ta}V\bm{x}_{i'}}{(k-1)n}\right)
+\frac{2k^4(\bm{x}_i^{\ta}V\bm{x}_{i'})^2}{(k-1)^4n^2}\bigg]$; and so, $\E(\h{\mu}_{j}\h{\mu}_{j'})=\frac{2\sigma^4}{n(n-1)}\mathop{\sum\sum}\limits_{1\le i<i'\le n} \left[\left(1+\frac{k\bm{x}_i^{\ta}V\bm{x}_i}{(k-1)n}\right)\left(1+\frac{k\bm{x}_{i'}^{\ta}V\bm{x}_{i'}}{(k-1)n}\right)
+\frac{2k^4(\bm{x}_i^{\ta}V\bm{x}_{i'})^2}{(k-1)^4n^2}\right]$. As in proof of Proposition \ref{prop.E,V,C-randomCV},
\[
\E(\h{\mu}_{j}\h{\mu}_{j'})
=\sigma^4
\left\{
1
+\frac{2kp}{(k-1)n}
+\frac{2k^4(p-\theta)}{(k-1)^4n(n-1)}
+\frac{k^2(p^2-\theta)}{(k-1)^2n(n-1)}
\right\}.
\]
Finally, the relation $\Cov(\h{\mu}_{j},\h{\mu}_{j'})=\E(\h{\mu}_{j}\h{\mu}_{j'})-\E(\h{\mu}_{j})\E(\h{\mu}_{j'})$ completes the proof.
\end{pr}

\begin{pr}{Proof of Proposition \ref{prop.E,V,C-randomCV-nonnormality}}
Set $X=W_i\big/\sqrt{n_1}$, $Y=W_{i'}\big/\sqrt{n_1}$, $Z=\varepsilon_i$ and $W=\varepsilon_{i'}$. From Lemma \ref{lem.W}, using Theorem \ref{theo.isserlis},
\[
\sfe_i=\E(X-Z)^2=\Var(X-Z)=\Var(X)+\Var(Z)=\sigma^2\left(1+n_1^{-1}\bm{x}_i^\ta V\bm{x}_i\right),
\]
\[
\sfe_{i,i}=\E(X-Z)^4=\sum\limits_{k=0}^{4}{4\choose k}(-1)^k\E(X^{4-k})\E(Z^k)
=\mu_4+6\sigma^4\frac{\bm{x}_i^\ta V\bm{x}_i}{n_1}+3\sigma^4\frac{(\bm{x}_i^\ta V\bm{x}_i)^2}{n_1^2}.
\]
Due the independence of $\CZ_{S_j}$, $y_i$ and $y_{i'}$, $\sfe_{i,i'}=\E[(X-Z)^2(Y-W)^2]=\E(X^2Y^2)-2\E(X^2Y)\E(W)+\E(X^2)\E(W^2)-2\E(XY^2)\E(Z)+4\E(XY)\E(Z)\E(W)-2\E(X)\E(Z)\times\E(W^2) +\E(Y^2)\E(Z^2)-2\E(Y)\E(Z^2)\E(W)+\E(Z^2)\E(W^2)$. By the fact that the asymptotic expected values are $\E(X)=\E(Y)=\E(Z)=\E(W)=0$, $\E(Z^2)=\E(W^2)=\sigma^2$, $\E(X^2)=\sigma^2n_1^{-1}\bm{x}_i^\ta V\bm{x}_i$, $\E(Y^2)=\sigma^2n_1^{-1}\bm{x}_{i'}^\ta V\bm{x}_{i}$ and $\E(X^2Y^2)=\sigma^4n_1^{-2}\bm{x}_{i}^\ta V\bm{x}_{i}\bm{x}_{i'}^\ta V\bm{x}_{i'}+2\sigma^4n_1^{-2}(\bm{x}_{i}^\ta V\bm{x}_{i'})^2$, we get
\[
\sfe_{i,i'}=\sigma^4\left[\left(1+n_1^{-1}\bm{x}_i^\ta V\bm{x}_i\right)\left(1+n_1^{-1}\bm{x}_{i'}^\ta V\bm{x}_{i'}\right)
+2n_1^{-2}(\bm{x}_{i}^\ta V\bm{x}_{i'})^2\right].
\]
Applying Theorem {\rm\ref{theo.var(mu_j)}} and using \eqref{eq.C2}, the form of the variance of $\h{\mu}_j$ follows.

Note that $\mu_4-\sigma^4=\E(\varepsilon^4)-\E^2(\varepsilon^2)=\Var(\varepsilon^2)>0$.
\end{pr}

\begin{pr}{Proof of \eqref{eq.e_i,e_ii,e_ii',0/1}}
Using the asymptotic distribution of $\bm{\varPsi}_{i,i'}$, the expected values are:
\[
\begin{split}
&\E[L_{0/1}(\h{y}_{S_j,i},y_i)|S_j,i]\\
=&\Pr(L_{0/1}(\h{y}_{S_j,i},y_i)=1|S_j,i)
 =\Pr(\h{y}_{S_j,i}=0,y_i=1|S_j,i)+\Pr(\h{y}_{S_j,i}=1,y_i=0|S_j,i)\\
=&\Pr(\varPsi_i<-\sqrt{n_1}\zeta_i)\Pr(y_i=1)
  +\Pr(\varPsi_i\ge-\sqrt{n_1}\zeta_i)\Pr(y_i=0)\\
=&\varPhi(-\sqrt{n_1}\zeta_i)p_i+\varPhi(\sqrt{n_1}\zeta_i)(1-p_i),
\end{split}
\]
\[
\begin{split}
&\E[L_{0/1}(\h{y}_{S_j,i},y_i)L_{0/1}(\h{y}_{S_j,{i'}},y_{i'})|S_j,i,i']\\
=&\Pr(L_{0/1}(\h{y}_{S_j,i},y_i)L_{0/1}(\h{y}_{S_j,{i'}},y_{i'})=1|S_j,i,i')\\
=&\Pr(\h{y}_{S_j,i}=0,y_i=1,\h{y}_{S_j,i'}=0,y_{i'}=1|S_j,i,i') \\
 &+\Pr(\h{y}_{S_j,i}=0,y_i=1,\h{y}_{S_j,i'}=1,y_{i'}=0|S_j,i,i')\\
 &+\Pr(\h{y}_{S_j,i}=1,y_i=0,\h{y}_{S_j,i'}=0,y_{i'}=1|S_j,i,i')\\
 &+\Pr(\h{y}_{S_j,i}=1,y_i=0,\h{y}_{S_j,i'}=1,y_{i'}=0|S_j,i,i')\\
=&\Pr(\varPsi_i<-\sqrt{n_1}\zeta_i,\varPsi_{i'}<-\sqrt{n_1}\zeta_{i'})\Pr(y_i=1)\Pr(y_{i'}=1)\\
 &+\Pr(\varPsi_i<-\sqrt{n_1}\zeta_i,\varPsi_{i'}\ge-\sqrt{n_1}\zeta_{i'})\Pr(y_i=1)\Pr(y_{i'}=0)\\
 &+\Pr(\varPsi_i\ge-\sqrt{n_1}\zeta_i,\varPsi_{i'}<-\sqrt{n_1}\zeta_{i'})\Pr(y_i=0)\Pr(y_{i'}=1)\\
 &+\Pr(\varPsi_i\ge-\sqrt{n_1}\zeta_i,\varPsi_{i'}\ge-\sqrt{n_1}\zeta_{i'})\Pr(y_i=0)\Pr(y_{i'}=0)\\
=&\Pr(\varPsi_i<-\sqrt{n_1}\zeta_i,\varPsi_{i'}<-\sqrt{n_1}\zeta_{i'})\Pr(y_i=1)\Pr(y_{i'}=1)\\
 &+\Pr(\varPsi_i<-\sqrt{n_1}\zeta_i,-\varPsi_{i'}\le\sqrt{n_1}\zeta_{i'})\Pr(y_i=1)\Pr(y_{i'}=0)\\
 &+\Pr(-\varPsi_i\le\sqrt{n_1}\zeta_i,\varPsi_{i'}<-\sqrt{n_1}\zeta_{i'})\Pr(y_i=0)\Pr(y_{i'}=1)\\
 &+\Pr(-\varPsi_i\le\sqrt{n_1}\zeta_i,-\varPsi_{i'}\le\sqrt{n_1}\zeta_{i'})\Pr(y_i=0)\Pr(y_{i'}=0)\\
=&~F_{\varPsi_i,\varPsi_{i'}}(-\sqrt{n_1}\zeta_i,-\sqrt{n_1}\zeta_{i'})p_ip_{i'}
  +F_{\varPsi_i,-\varPsi_{i'}}(-\sqrt{n_1}\zeta_i,\sqrt{n_1}\zeta_{i'})p_i(1-p_{i'})\\
&+F_{-\varPsi_i,\varPsi_{i'}}(\sqrt{n_1}\zeta_i,-\sqrt{n_1}\zeta_{i'})(1-p_i)p_{i'}
 +F_{-\varPsi_i,-\varPsi_{i'}}(\sqrt{n_1}\zeta_i,\sqrt{n_1}\zeta_{i'})(1-p_i)(1-p_{i'}),
\end{split}
\]
completing the proof.
\end{pr}

\section{Additional simulations results}
\label{sec.online.3}

 \begin{table}[htp]
 \caption{Average of the estimated values and their empirical mean square error for various data distributions and various values of $n$, for the approximated absolute error loss, with $d=n^{-1}$, in the case of sample mean.}
 \label{table.Abs}
 \centering{\footnotesize
\begin{tabular}
 {@{\hspace{0ex}}l@{\hspace{4.2ex}}r@{\hspace{4.2ex}}
 r@{\hspace{2.4ex}}r@{\hspace{2.4ex}}r@{\hspace{0ex}}
 c@{\hspace{6ex}}
 r@{\hspace{2.4ex}}r@{\hspace{2.4ex}}r@{\hspace{2.4ex}}r
 @{\hspace{0ex}}}
 \addlinespace
 \toprule
 & \multicolumn{9}{c}{$\begin{array}{lr}\textrm{\bf Sample Size}&\textrm{\bf Estimator(Theoretical Value)}\\&\textrm{\bf MSE}\end{array}$} \\
 \hline
 $F$ & \multicolumn{1}{l}{$n$} & \multicolumn{1}{l}{$\frac{n_1^{\rm opt}}{n}$} & \multicolumn{1}{l}{$\rho^{\rm opt}$} & \multicolumn{1}{l}{$\frac{k^{\rm opt}}{n}$}
 &
 & \multicolumn{1}{l}{$n$} & \multicolumn{1}{l}{$\frac{n_1^{\rm opt}}{n}$} & \multicolumn{1}{l}{$\rho^{\rm opt}$} & \multicolumn{1}{l}{$\frac{k^{\rm opt}}{n}$}\\
 \hline
 \multirow{6}{*}{\rotatebox{90}{$N(0,1)$}}
 &\rm   60 &\rm .500(.500) &\rm .4794(.4800) &\rm 1(1) & &\rm  750 &\rm .500(.500) &\rm .4982(.4982) &\rm 1(1) \\
 [-.1ex]
 &\rm      &\rm         0  &\rm     $<10^{-5}$ &\rm    0 & &\rm      &\rm         0  &\rm     $<10^{-7}$ &\rm    0 \\
 [.3ex]
 &\rm  100 &\rm .500(.500) &\rm .4874(.4876) &\rm 1(1) & &\rm 1501 &\rm .500(.500) &\rm .4988(.4988) &\rm 1(1) \\
 [-.1ex]
 &\rm      &\rm        0   &\rm     $<10^{-5}$ &\rm    0 & &\rm      &\rm        0   &\rm     $<10^{-9}$ &\rm    0 \\
 [.3ex]
 &\rm  301 &\rm .502(.502) &\rm .4941(.4941) &\rm 1(1) & &\rm 5000 &\rm .500(.500) &\rm .4997(.4997) &\rm 1(1) \\
 [-.1ex]
 &\rm      &\rm          0 &\rm     $<10^{-7}$ &\rm    0 & &\rm      &\rm          0 &\rm    $<10^{-10}$ &\rm    0 \\
\bottomrule
 \end{tabular}}
 \end{table}

\begin{table}[htp]
 \caption{Average by $N=500$ Monte Carlo repetitions of the estimated values $\h{\Var}(\h{\mu}_j)$, $\h{\Var}(\h{\mu}_{{\rm CV},10})$, $\h{\Var}(\h{\mu}_{{\rm CV},15})$, $\h{\Var}(\h{\mu}_{{\rm CV},\infty})=\h{\Cov}(\h{\mu}_j,\h{\mu}_{j'})\equiv\h{C}$ for various data distributions, various values of $n$ and various of $n_1$ (for each $n$), for the squared error loss, in the case of sample mean.}
 \label{table.Sq2}
 \centering{\scriptsize
 \begin{tabular}
 {@{\hspace{0ex}}l@{\hspace{1.5ex}}l@{\hspace{3ex}}r@{\hspace{15ex}}c@{\hspace{12ex}}c@{\hspace{12ex}}c@{\hspace{12ex}}c@{\hspace{0ex}}}
 \addlinespace
 \toprule
 \multicolumn{7}{c}{\sc Squared error loss} \\
 \hline
  & ~~$n$ & $n_1~~~$ & $\h{\Var}(\h{\mu}_j)$ & $\h{\Var}(\h{\mu}_{{\rm CV},10})$ & $\h{\Var}(\h{\mu}_{{\rm CV},15})$ & $\h{\Var}(\h{\mu}_{{\rm CV},\infty})$ \\
 \hline
 \multirow{18}{*}{\rotatebox{90}{$N(0,1)$}}
 &~24 & $n_1^{\rm opt}=12$ & 0.1840 & 0.0943 & 0.0910 & 0.0843 \\
 [-.1ex]
 &    & $.70n=17$          & 0.2999 & 0.1059 & 0.0987 & 0.0843 \\
 [-.1ex]
 &    & $.75n=18$          & 0.3475 & 0.1106 & 0.1019 & 0.0843 \\
 [-.1ex]
 &    & $.80n=19$          & 0.4144 & 0.1173 & 0.1063 & 0.0843 \\
 [-.1ex]
 &    & $.85n=20$          & 0.5151 & 0.1274 & 0.1130 & 0.0843 \\
 [-.1ex]
 &    & $.90n=22$          & 1.0202 & 0.1779 & 0.1467 & 0.0843 \\
\bottomrule
 \end{tabular}}
 \end{table}

  \begin{table}[htp]
 \caption*{Table \ref{table.Sq2}: (continued)}
 \centering{\scriptsize
 \begin{tabular}
 {@{\hspace{0ex}}l@{\hspace{2ex}}l@{\hspace{4ex}}r@{\hspace{14ex}}c@{\hspace{9ex}}c@{\hspace{9ex}}c@{\hspace{9ex}}c@{\hspace{0ex}}}
 \addlinespace
 \toprule
 \multicolumn{7}{c}{\sc Squared error loss} \\
 \hline
  & ~~$n$ & $n_1~~~$ & $\h{\Var}(\h{\mu}_j)$ & $\h{\Var}(\h{\mu}_{{\rm CV},10})$ & $\h{\Var}(\h{\mu}_{{\rm CV},15})$ & $\h{\Var}(\h{\mu}_{{\rm CV},\infty})$ \\
 \hline
 \multirow{18}{*}{\rotatebox{90}{$\exp(1)$}}
 &~24 & $n_1^{\rm opt}=12$ & 0.6800 & 0.3650 & 0.3533 & 0.3300 \\
 [-.1ex]
 &    & $.70n=17$          & 1.1455 & 0.4115 & 0.3843 & 0.3300 \\
 [-.1ex]
 &    & $.75n=18$          & 1.3333 & 0.4303 & 0.3969 & 0.3300 \\
 [-.1ex]
 &    & $.80n=19$          & 1.5966 & 0.4567 & 0.4144 & 0.3300 \\
 [-.1ex]
 &    & $.85n=20$          & 1.9919 & 0.4962 & 0.4408 & 0.3300 \\
 [-.1ex]
 &    & $.90n=22$          & 3.9707 & 0.6940 & 0.5727 & 0.3300 \\
\bottomrule
 \end{tabular}}
 \end{table}

 \begin{table}[htp]
 \caption{Average by $N=500$ Monte Carlo repetitions of the estimated values $\h{\Var}(\h{\mu}_j)$, $\h{\Var}(\h{\mu}_{{\rm CV},10})$, $\h{\Var}(\h{\mu}_{{\rm CV},15})$, $\h{\Var}(\h{\mu}_{{\rm CV},\infty})=\h{\Cov}(\h{\mu}_j,\h{\mu}_{j'})\equiv\h{C}$ for various data distributions, various values of $n$ and various of $n_1$ (for each $n$), for the approximated absolute error loss, in the case of sample mean.}
 \label{table.Abs2}
 \centering{\scriptsize
 \begin{tabular}
 {@{\hspace{0ex}}l@{\hspace{2ex}}l@{\hspace{4ex}}r@{\hspace{15.3ex}}c@{\hspace{10.5ex}}c@{\hspace{10.5ex}}c@{\hspace{10.5ex}}c@{\hspace{0ex}}}
 \addlinespace
 \toprule
 \multicolumn{7}{c}{\sc Approximated absolute error loss} \\
 \hline
  & ~~$n$ & $n_1~~~$ & $\h{\Var}(\h{\mu}_j)$ & $\h{\Var}(\h{\mu}_{{\rm CV},10})$ & $\h{\Var}(\h{\mu}_{{\rm CV},15})$ & $\h{\Var}(\h{\mu}_{{\rm CV},\infty})$ \\
 \hline
 \multirow{18}{*}{\rotatebox{90}{$N(0,1)$}}
 &~24 & $n_1^{\rm opt}=12$ & 0.0304 & 0.0154 & 0.0148 & 0.0137 \\
 [-.1ex]
 &    & $.70n=17$          & 0.0501 & 0.0178 & 0.0166 & 0.0142 \\
 [-.1ex]
 &    & $.75n=18$          & 0.0581 & 0.0186 & 0.0172 & 0.0143 \\
 [-.1ex]
 &    & $.80n=19$          & 0.0694 & 0.0198 & 0.0180 & 0.0143 \\
 [-.1ex]
 &    & $.85n=20$          & 0.0863 & 0.0216 & 0.0192 & 0.0144 \\
 [-.1ex]
 &    & $.90n=22$          & 0.1714 & 0.0301 & 0.0249 & 0.0145 \\
\bottomrule
 \end{tabular}}
 \end{table}

  \begin{table}[htp]
 \caption*{Table \ref{table.Abs2}: (continued)}
 \centering{\scriptsize
 \begin{tabular}
 {@{\hspace{0ex}}l@{\hspace{2ex}}l@{\hspace{4ex}}r@{\hspace{14ex}}c@{\hspace{10.5ex}}c@{\hspace{8ex}}c@{\hspace{8ex}}c@{\hspace{0ex}}}
 \addlinespace
 \toprule
 \multicolumn{7}{c}{\sc Approximated absolute error loss} \\
 \hline
  & ~~$n$ & $n_1~~~$ & $\h{\Var}(\h{\mu}_j)$ & $\h{\Var}(\h{\mu}_{{\rm CV},10})$ & $\h{\Var}(\h{\mu}_{{\rm CV},15})$ & $\h{\Var}(\h{\mu}_{{\rm CV},\infty})$ \\
 \hline
 \multirow{18}{*}{\rotatebox{90}{$\exp(1)$}}
 &~24 & $n_1^{\rm opt}=12$ & 0.0393 & 0.0205 & 0.0198 & 0.0184 \\
 [-.1ex]
 &    & $.70n=17$          & 0.0605 & 0.0229 & 0.0215 & 0.0188 \\
 [-.1ex]
 &    & $.75n=18$          & 0.0695 & 0.0239 & 0.0222 & 0.0188 \\
 [-.1ex]
 &    & $.80n=19$          & 0.0823 & 0.0252 & 0.0231 & 0.0189 \\
 [-.1ex]
 &    & $.85n=20$          & 0.1015 & 0.0272 & 0.0244 & 0.0189 \\
 [-.1ex]
 &    & $.90n=22$          & 0.1986 & 0.0369 & 0.0309 & 0.0190 \\
\bottomrule
 \end{tabular}}
 \end{table}

 \begin{table}[htp]
 \caption{Average by $N=500$ Monte Carlo repetitions of the estimated values $\h{\Var}(\h{\mu}_j)$, $\h{\Var}(\h{\mu}_{{\rm CV},10})$, $\h{\Var}(\h{\mu}_{{\rm CV},15})$, $\h{\Var}(\h{\mu}_{{\rm CV},\infty})=\h{\Cov}(\h{\mu}_j,\h{\mu}_{j'})\equiv\h{C}$ for various data distributions, various values of $n$ and various of $n_1$ (for each $n$), for the Efron's $q$ error loss, with $q(t)=-\sqrt{t^2+1}$, in the case of sample mean.}
 \label{table.Efron2}
 \centering{\scriptsize
 \begin{tabular}
 {@{\hspace{0ex}}l@{\hspace{2ex}}l@{\hspace{4ex}}r@{\hspace{15.3ex}}c@{\hspace{10.5ex}}c@{\hspace{10.5ex}}c@{\hspace{10.5ex}}c@{\hspace{0ex}}}
 \addlinespace
 \toprule
 \multicolumn{7}{c}{\sc Efron's $q$ error loss, with $q(t)=-\sqrt{t^2+1}$} \\
 \hline
  & ~~$n$ & $n_1~~~$ & $\h{\Var}(\h{\mu}_j)$ & $\h{\Var}(\h{\mu}_{{\rm CV},10})$ & $\h{\Var}(\h{\mu}_{{\rm CV},15})$ & $\h{\Var}(\h{\mu}_{{\rm CV},\infty})$ \\
 \hline
 \multirow{18}{*}{\rotatebox{90}{$N(0,1)$}}
 &~24 & $n_1^{\rm opt}=12$ & 0.0202 & 0.0095 & 0.0092 & 0.0084\\
 [-.1ex]
 &    & $.70n=17$          & 0.0313 & 0.0107 & 0.0099 & 0.0084\\
 [-.1ex]
 &    & $.75n=18$          & 0.0360 & 0.0112 & 0.0102 & 0.0084\\
 [-.1ex]
 &    & $.80n=19$          & 0.0426 & 0.0118 & 0.0107 & 0.0084\\
 [-.1ex]
 &    & $.85n=20$          & 0.0526 & 0.0128 & 0.0114 & 0.0084\\
 [-.1ex]
 &    & $.90n=22$          & 0.1031 & 0.0179 & 0.0147 & 0.0084\\
\bottomrule
 \end{tabular}}
 \end{table}

  \begin{table}[htp]
 \caption*{Table \ref{table.Efron2}: (continued)}
 \centering{\scriptsize
 \begin{tabular}
 {@{\hspace{0ex}}l@{\hspace{2ex}}l@{\hspace{4ex}}r@{\hspace{14ex}}c@{\hspace{10.5ex}}c@{\hspace{10.5ex}}c@{\hspace{10.5ex}}c@{\hspace{0ex}}}
 \addlinespace
 \toprule
 \multicolumn{7}{c}{\sc Efron's $q$ error loss, with $q(t)=-\sqrt{t^2+1}$} \\
 \hline
  & ~~$n$ & $n_1~~~$ & $\h{\Var}(\h{\mu}_j)$ & $\h{\Var}(\h{\mu}_{{\rm CV},10})$ & $\h{\Var}(\h{\mu}_{{\rm CV},15})$ & $\h{\Var}(\h{\mu}_{{\rm CV},\infty})$ \\
 \hline
 \multirow{18}{*}{\rotatebox{90}{$\exp(1)$}}
 &~24 & $n_1^{\rm opt}=12$ & 0.0025 & 0.0012 & 0.0012 & 0.0011 \\
 [-.1ex]
 &    & $.70n=17$          & 0.0038 & 0.0013 & 0.0012 & 0.0010 \\
 [-.1ex]
 &    & $.75n=18$          & 0.0043 & 0.0014 & 0.0012 & 0.0010 \\
 [-.1ex]
 &    & $.80n=19$          & 0.0051 & 0.0014 & 0.0013 & 0.0010 \\
 [-.1ex]
 &    & $.85n=20$          & 0.0063 & 0.0015 & 0.0014 & 0.0010 \\
 [-.1ex]
 &    & $.90n=22$          & 0.0123 & 0.0021 & 0.0018 & 0.0010 \\
\bottomrule
 \end{tabular}}
 \end{table}

 \begin{table}[htp]
 \caption{Average by $N=500$ Monte Carlo repetitions of the estimated values $\h{\Var}(\h{\mu}_j)$, $\h{\Var}(\h{\mu}_{{\rm CV},10})$, $\h{\Var}(\h{\mu}_{{\rm CV},15})$, $\h{\Var}(\h{\mu}_{{\rm CV},\infty})=\h{\Cov}(\h{\mu}_j,\h{\mu}_{j'})\equiv\h{C}$ for various data distributions, various values of $n$ and various of $n_1$ (for each $n$), for the double squared error loss, in the case of sample mean.}
 \label{table.DouSq2}
 \centering{\scriptsize
 \begin{tabular}
 {@{\hspace{0ex}}l@{\hspace{2ex}}l@{\hspace{4ex}}r@{\hspace{15.3ex}}c@{\hspace{10.5ex}}c@{\hspace{10.5ex}}c@{\hspace{10.5ex}}c@{\hspace{0ex}}}
 \addlinespace
 \toprule
 \multicolumn{7}{c}{\sc Double squared error loss} \\
 \hline
  & ~~$n$ & $n_1~~~$ & $\h{\Var}(\h{\mu}_j)$ & $\h{\Var}(\h{\mu}_{{\rm CV},10})$ & $\h{\Var}(\h{\mu}_{{\rm CV},15})$ & $\h{\Var}(\h{\mu}_{{\rm CV},\infty})$ \\
 \hline
 \multirow{18}{*}{\rotatebox{90}{$N(0,1)$}}
 &~24 & $n_1^{\rm opt}=12$ &  8.354 &  4.592  & 4.452 & 4.173\\
 [-.1ex]
 &    & $.70n=17$          & 14.473 &  5.264  & 4.923 & 4.241\\
 [-.1ex]
 &    & $.75n=18$          & 16.908 &  5.516  & 5.094 & 4.250\\
 [-.1ex]
 &    & $.80n=19$          & 20.315 &  5.864  & 5.329 & 4.258\\
 [-.1ex]
 &    & $.85n=20$          & 25.423 &  6.381  & 5.676 & 4.266\\
 [-.1ex]
 &    & $.90n=22$          & 50.943 &  8.945  & 7.389 & 4.278\\
\bottomrule
 \end{tabular}}
 \end{table}

  \begin{table}[htp]
 \caption*{Table \ref{table.DouSq2} (continued)}
 \centering{\scriptsize
 \begin{tabular}
 {@{\hspace{0ex}}l@{\hspace{2ex}}l@{\hspace{4ex}}r@{\hspace{15.3ex}}c@{\hspace{10.5ex}}c@{\hspace{10.5ex}}c@{\hspace{10.5ex}}c@{\hspace{0ex}}}
 \addlinespace
 \toprule
 \multicolumn{7}{c}{\sc Double squared error loss} \\
 \hline
  & ~~$n$ & $n_1~~~$ & $\h{\Var}(\h{\mu}_j)$ & $\h{\Var}(\h{\mu}_{{\rm CV},10})$ & $\h{\Var}(\h{\mu}_{{\rm CV},15})$ & $\h{\Var}(\h{\mu}_{{\rm CV},\infty})$ \\
 \hline
 \multirow{18}{*}{\rotatebox{90}{$\exp(1)$}}
 &~24 & $n_1^{\rm opt}=12$ &  1691 &  929 &  901 & 844 \\
 [-.1ex]
 &    & $.70n=17$          &  2908 & 1055 &  986 & 849 \\
 [-.1ex]
 &    & $.75n=18$          &  3394 & 1104 & 1019 & 849 \\
 [-.1ex]
 &    & $.80n=19$          &  4074 & 1172 & 1065 & 850 \\
 [-.1ex]
 &    & $.85n=20$          &  5094 & 1275 & 1133 & 850 \\
 [-.1ex]
 &    & $.90n=22$          & 10194 & 1785 & 1474 & 851 \\
\bottomrule
 \end{tabular}}
 \end{table}

\end{document}